\documentclass[12pt]{amsart}

\usepackage{amsmath,amssymb,amsthm,hyperref,mathrsfs,bbm,graphicx,verbatim,url,xcolor}
\usepackage[margin=1in]{geometry}

\numberwithin{equation}{section}

\theoremstyle{definition}
\newtheorem{definition}{Definition}[section]
\newtheorem{example}{Example}[section]
\newtheorem*{acknowledgments}{Acknowledgments}

\theoremstyle{plain}
\newtheorem{theorem}{Theorem}[section]
\newtheorem{cor}[theorem]{Corollary}
\newtheorem{lemma}[theorem]{Lemma}
\newtheorem{prop}[theorem]{Proposition}

\theoremstyle{remark}
\newtheorem{remark}[theorem]{Remark}

\newcommand{\newname}{$bb$}

\title[Microlocal Methods for Transversely Isotropic Travel Time Tomography]{Microlocal Methods for The Elastic Travel Time Tomography Problem for Transversely Isotropic Media}
\author{Yuzhou Zou}
\address{Mathematics Department\\
University of California, Santa Cruz\\
1156 High St, Santa Cruz, CA 95064, U.S.A.}
\email{yzou34@ucsc.edu}
\date{\today}
\begin{document}
\maketitle
\sloppy

\begin{abstract}
We study the travel time tomography problem for transversely isotropic media using the artificial boundary method used in works in X-ray tomography and boundary rigidity problems. The analysis involved requires a modification of the scattering pseudodifferential calculus in order to provide parametrices for operators which are degenerately elliptic; such a calculus is constructed in this paper and used to solve the tomography problem of interest.
\end{abstract}

\tableofcontents

\section{Introduction}
\label{intro-sec}
In this paper we consider the travel time tomography problem for transversely isotropic media, using methods from microlocal analysis. The goal is to adapt the arguments used in works in X-ray tomography \cite{convex, tensor} and boundary rigidity \cite{br1,br2} to this problem; such works made use of an ``artificial boundary'' argument as well as extensive use of the scattering pseudodifferential calculus. The works \cite{isoe,ti} provide examples where such arguments were applied to travel time tomography problems in elasticity. In particular, in the latter work, the authors obtained recovery for some of the parameters involved, as well as a description of the obstruction of the analysis to extend the argument in recovering the remaining parameters: namely that some of the operators involved were not elliptic considered as scattering pseudodifferential operators, and hence were not known to have parametrices, i.e. approximate inverses in a suitable sense. The main goal for this work is to construct an operator calculus which provides parametrices to the operators considered in \cite{ti}, thus allowing  the analogous arguments to be made in this context.

The setting for this work is a combination of the settings in the previous works, as well as the author's previous work \cite{zou-global} which did not use the artificial boundary argument of the above works. The essential aspects of this setting are described below.

In elasticity, the linear elastic wave equation gives a linearized model for the motion of elastic objects. The equation is given by
\[u_{tt} = Eu,\quad (Eu)_i = \sum_{jkl}{\frac{1}{\rho(x)}\partial_j{c_{ijkl}(x)\partial_lu_k}},\]
where $u:\mathbb{R}_t\times\mathbb{R}^3\to\mathbb{R}^3$ represents the displacement of a material from a rest frame, $\rho(x)>0$ is the density of the material, and $c_{ijkl}(x)$ are components of a tensor known as the elasticity tensor. The principal symbol of the operator $u \mapsto u_{tt}-Eu$ is given by $-\tau^2\text{Id} + \sigma(-E)(x,\xi)$, where 
the matrix $\sigma(-E)(x,\xi)$ is always symmetric and positive definite for all $x$ and all $\xi\ne 0$, and hence in describing the microlocal wavefront set $WF(u)\subset T^*(\mathbb{R}_t\times\mathbb{R}^3_x)\backslash o$ of a solution $u$, we have
\begin{align*} (t,x,\tau,\xi)\in WF(u)&\implies -\tau^2\text{Id}+\sigma(-E)(x,\xi)\text{ is not invertible}\\
&\iff \tau^2 = G(x,\xi)\text{ where }G(x,\xi)\text{ is an eigenvalue of }\sigma(-E)(x,\xi).
\end{align*}
Under the further assumption that the multiplicity of the eigenvalues of $\sigma(-E)$ remain constant in $(x,\xi)$, so that the eigenvalues can be written as $G_j(x,\xi)$ for smooth functions $G_j$ on $T^*\mathbb{R}^3$, it follows that $WF(u)\cap\{\tau^2 = G_j\}$ is invariant under the Hamilton flow of $\tau^2-G_j$. We are thus interested in the dynamics of the so-called ``elastic waves'', i.e. the Hamiltonian dynamics of the eigenvalues $G_j$, the so-called ``squared wave speeds'' of the material.

In the situation of transversely isotropic elasticity, the elasticity tensor can be described in terms of an axis of isotropy, which we can represent naturally as a line subbundle $\Sigma$ of the cotangent bundle $T^*\mathbb{R}^3$, along with five material parameters. Following the conventions of previous works, we denote these parameters as $a_{11}$, $a_{33}$, $a_{55}$, $a_{66}$, and $E^2$. There are three kinds of waves, denoted $qP$, $qSH$, and $qSV$, along with their squared wave speeds $G_{qP}$, $G_{qSH}$, and $G_{qSV}$, and with respect to the parameters above, the squared wave speeds can be written explicitly as
\[G_{qSH} = a_{66}|\xi'|^2+a_{55}\xi_3^2\]
where $|\xi'|^2 = \xi_1^2+\xi_2^2$, and
\begin{align*}
G_{\pm} &= (a_{11}+a_{55})|\xi'|^2+(a_{33}+a_{55})\xi_3^2 \\
&\pm\sqrt{((a_{11}-a_{55})|\xi'|^2+(a_{33}-a_{55})\xi_3^2)^2-4E^2|\xi'|^2\xi_3^2}
\end{align*}
where $+$ refers to the $qP$ wave speed and $-$ refers to the $qSV$ wave speed. Here, we fix a point $x$ and choose coordinates $(x_1,x_2,x_3,\xi_1,\xi_2,\xi_3)$ such that $dx_3$ aligns with the axis of isotropy (i.e. $\text{span }(dx_3)_x = \Sigma_x$), and $(\xi_1,\xi_2,\xi_3)$ is chosen with respect to $(x_1,x_2,x_3)$ so that overall the coordinates form canonical coordinates on $T^*\mathbb{R}^3$, i.e. $(\xi_1,\xi_2,\xi_3)$ corresponds to the covector $\sum_{i=1}^3{\xi_i(dx_i)_x}$.

The inverse problem we wish to investigate is the travel time tomography problem. Consider a domain $\Omega\subset\mathbb{R}^3$, and suppose for each pair of points on the boundary $\partial\Omega$, for each squared wave speed $G$ there exists a unique Hamiltonian trajectory of $G$ (up to scaling) whose projection to the physical space $\mathbb{R}^3_x$ connects the pair of points. (Later on, it will suffice for this condition to hold for all sufficiently close pair of points on the boundary). Suppose, for each pair of points, we know the travel time of the trajectory associated to each wave speed. Does that information suffice to determine the elasticity tensor of the material?

We can formulate the main results of this paper which aim to answer the above question. We assume that $\partial\Omega$ is convex with respect to all Hamiltonian dynamics considered.
\begin{theorem}
[Local recovery of one parameter]
\label{local-recov-thm}
Let $p\in\partial\Omega$. Suppose that $\{\tilde{a}_{\nu}\}$ is another collection of elastic parameters, such that $\tilde{a}_{\nu} = a_{\nu}$ for all parameters except one, and that there exists a neighborhood $U$ of $p$ such that the two collections of elastic parameters produces the same travel time data between points on $U\cap\partial\Omega$. Then there is a (possibly smaller) neighborhood $U'$ such that $\tilde{a}_{\nu} = a_{\nu}$ on $U'$.
\end{theorem}
In other words, assuming just the convexity of $\partial\Omega$, we can uniquely determine the elastic parameters in a neighborhood of $\partial\Omega$ using the travel time data on $\partial\Omega$.

We will particularly be interested in the situation where our domain of interest admits a \emph{convex foliation}, defined below:
\begin{definition}
\label{convex-foliation-def}
A manifold $\Omega$ with boundary $\partial\Omega$ admits a \emph{convex foliation} (with respect to a certain system of dynamics involved, such as geodesic dynamics or more generally Hamiltonian dynamics associated to some Hamiltonian) if there is a function $\rho:\overline{\Omega}\to[0,T]$ for some $T>0$ such that $\rho^{-1}(0) = \partial\Omega$, $\rho^{-1}(t)$ is strictly convex for $0\le t<T$, and $\overline{\Omega}\backslash\cup_{t\in[0,T)}{\rho^{-1}(t)}$ has measure zero. The level sets $\rho^{-1}(t)$ for $0\le t<T$ are called the \emph{leaves} of the convex foliation.
\end{definition}
By ``strictly convex,'' we mean that if $(X(s),\Xi(s))$ is a trajectory of the relevant Hamiltonian dynamics (with $\frac{d}{ds}(X(s),\Xi(s))\ne 0$ everywhere), with $\rho(X(s_0)) = t$ and $\frac{d}{ds}|_{s=s_0}(\rho(X(s))) = 0$ for some $0\le t<T$, then $\frac{d^2}{ds^2}|_{s=s_0}(-\rho(X(s)))>0$.

As mentioned above, the idea of a convex foliation was used in \cite{convex,br1,br2,isoe,ti}, etc., to upgrade local recovery results to global recovery results for domains admitting a convex foliation. See \cite{psuz} for a more detailed description on how to interpret the convex foliation condition, as well as cases when convex foliations must exist.

In the case where we have a convex foliation, we can extend our local recovery result to a global recovery result.
\begin{theorem}[Global recovery of one parameter]
\label{global-recov-scat-thm}
Suppose $\Omega$ admits a convex foliation. Suppose that $\{\tilde{a}_{\nu}\}$ is another collection of elastic parameters, such that $\tilde{a}_{\nu} = a_{\nu}$ for all parameters except one, and the two collections of elastic parameters produces the same travel time data between every pair of points on $\partial\Omega$. Then we must have $\tilde{a}_{\nu} = a_{\nu}$ everywhere in $\Omega$.
\end{theorem}

We expect the same recovery result to hold for the recovery of multiple parameters (cf. the results in \cite{ti} and \cite{zou-global}); the computations are a bit more complicated to write down, but the theory required should be the same as in the one parameter recovery problem.

The main strategy is to take advantage of the Stefanov-Uhlmann pseudolinearization argument, which offers a method to construct an operator $I$ such that $I[\{r_{\nu}\}]=0$, where $r_{\nu}$ are the differences of the parameters. The task is then to show the invertibility of the constructed operator so as to show that $r_{\nu}=0$.

Explicitly, using the notation from \cite{zou-global}, it was shown that if two sets of parameters gave the same travel time data, then the difference $r_{\nu} = \tilde{a}_{\nu}-a_{\nu}$ of the parameters must satisfy $0 = \sum{I^{\nu}(\nabla r_{\nu})} + \sum{\tilde{I}^{\nu}(r_{\nu})}$, where
\begin{equation}
I^{\nu}[f_1,f_2,f_3](x,\xi) = -\int_{\mathbb{R}}{E^{\nu}(X(t),\Xi(t))\frac{\partial\tilde\Xi}{\partial\xi}(\tau(X(t),\Xi(t)),(X(t),\Xi(t)))\cdot\begin{pmatrix} f_1 \\ f_2 \\ f_3 \end{pmatrix}(X(t))\,dt}\label{inu}
\end{equation}
and
\begin{equation}
\tilde{I}^{\nu}[f](x,\xi) = \int_{\mathbb{R}}{\left(-\partial_xE^{\nu}\frac{\partial\tilde\Xi}{\partial\xi}(\tau(\cdot),\cdot) + \partial_{\xi}E^{\nu}\frac{\partial\tilde\Xi}{\partial x}(\tau(\cdot),\cdot)\right)(X(t),\Xi(t))f(X(t))\,dt}.\label{itildenu}
\end{equation}
Here, $(x,\xi)$ are canonical coordinates on $T^*\mathbb{R}^3$, and $(X(t),\Xi(t)) = (X(t,x,\xi),\Xi(t,x,\xi))$ denotes the Hamilton flow with respect to the wave speed $G(x,\xi)$. The function $E^{\nu}$ is defined in \cite{zou-global} in terms of the two collection of parameters $\{a_{\nu}\}$ and $\{\tilde{a}_{\nu}\}$; when $r_{\nu}$ is small one can think of $E^{\nu}$ as approximately equal to $\frac{\partial G}{\partial\nu}$, i.e. the derivative of the squared wave speed with respect to the parameter $\nu$ viewing $G$ as a function of the parameters.

In \cite{ti}, the authors obtained local recovery for the parameter $a_{11}$, which combined with the existence of a convex foliation upgrades to a global recovery result. The argument for the local recovery is a modification of the local invertibility of the X-ray transform first studied in \cite{convex}, where the authors studied the invertibility of the geodesic X-ray transform (i.e. showing that $u\equiv 0$ if $I[u]\equiv 0$, where $I$ is the geodesic X-ray transform which takes a geodesic and a function and integrates the function on the geodesic). The argument there follows as thus: assuming the convexity assumptions, the authors showed that there exists a smooth function $\tilde{x}$ such that $\tilde{x}(p) = 0$, $d\tilde{x}\ne 0$ near $p$, $\tilde{x}<0$ for $x\in\overline{\Omega}\backslash p$ near $p$, $\tilde{x}$ is strictly convex with respect to all relevant dynamics, and for sufficiently small $c>0$ we have $\overline{\Omega_c}:=\{\tilde x>-c\}\cap\overline{\Omega}$ is compact. (Indeed, if $\rho$ is a boundary defining function of $\partial{\Omega}$ near $p$, with $\rho>0$ in the interior, so that $-\rho$ is strictly convex with respect to all dynamics considered near $p$, then we can take $\tilde{x}$ to be a slightly less convex modification of $-\rho$.) 
Near $p$, we view $\overline{\Omega}$ as a closed subset of a manifold without boundary. We can then fix $c>0$ sufficiently small, and let $x = \tilde{x}+c$, so that $X = \{x\ge 0\}$ corresponds to $\{\tilde{x}\ge -c\}$. Note then that $\{x=0\}$, equivalently $\{\tilde{x} = -c\}$, serves as the boundary for the newly constructed space $X$; this boundary is what we refer to as an ``artificial boundary'' (which itself depends on $c$). The situation is illustrated in Figure \ref{domainsketch}.
\begin{figure}[h]
\label{domainsketch}
\begin{center}
\includegraphics[scale=0.9,trim={1in, 7.7in, 1in, 1.25in},clip]{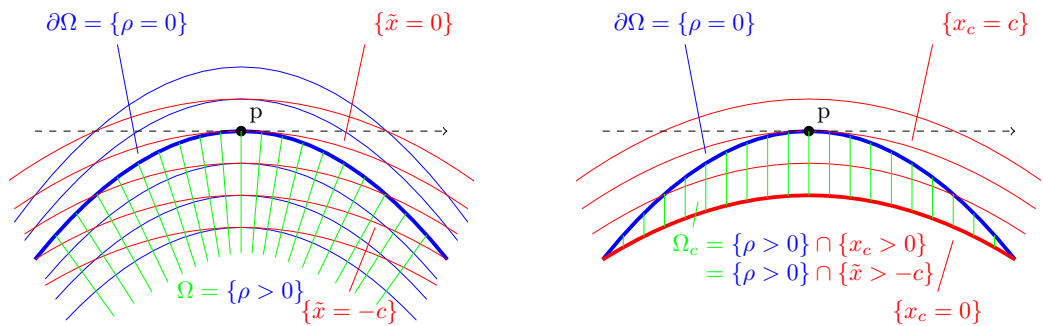}
\end{center}
\caption{Sketches of the domains. In the left figure, the blue curves are the level sets of $\rho$, a boundary-defining function near $p$. The thick blue curve denotes $\partial\Omega=\{\rho=0\}$. The red curves are the level sets of $\tilde{x}$, a ``less convex'' modification of $-\rho$ (but still convex with respect to the relevant dynamics). On the right, the thick red curve denotes $\{\tilde{x}+c=0\} = \{x_c=0\}$ for some choice of $c>0$; this set denotes the ``artificial boundary,'' with the original boundary (the thick blue curve) constraining the support of the difference of parameters (but otherwise not playing the role of a ``boundary'' in an analytic sense).}
\end{figure}

We can then choose two more smooth functions $y_1,y_2$ so that $(x,y_1,y_2)$ give coordinates in a neighborhood of $\overline{\Omega_c}$. We then let
\[Lv(z) = x^{-2}(z)\int_{S_zX}{\chi(\lambda(\nu)/x(z))v(\gamma_{\nu,z})\,d\nu}\]
where $S_zX$ is the sphere bundle of $X$, and for a unit tangent vector $\nu$ we let $\lambda(\nu)$ be the coefficient $\lambda$ appearing in writing $\nu = \lambda\partial_x+\omega\cdot\partial_y$, and $\gamma_{\nu,z}$ is a trajectory passing through $z$ with tangent vector $\nu$, and $\chi\in C_c^{\infty}(\mathbb{R})$. Note that by convexity assumptions we can arrange $\chi$ to be supported sufficiently close to $0$ such that trajectories $\nu$ with $\lambda/x\in\text{supp }\chi$ remain in $X$ for all times. Furthermore, note the definition of $L$ above depends on the choice of $c>0$. The authors then showed that for $\digamma>0$ and
\[N_{\digamma} := e^{-\digamma/x}LIe^{\digamma/x}\]
we have that $N_{\digamma}$ is a \emph{scattering pseudodifferential operator} on $X$, and furthermore that it is elliptic. (See \cite{sslaes} and \cite{grenoble} for an introduction to the scattering pseudodifferential calculus and its properties.) In particular, there exists a parametrix $G$ such that $GN_{\digamma} = \text{Id} + E$, where $E\in\Psi^{-\infty,-\infty}$ is a ``very nice error.'' The idea is that if $I[u]=0$, and we write $u = e^{\digamma/x}\tilde{u}$, then $0 = N_{\digamma}[\tilde{u}]$, and hence applying a parametrix yields $\tilde{u} = E[\tilde{u}]$; if $E$ is sufficiently nice (e.g. small enough in operator norm) then this allows us to conclude $\tilde{u} = 0$.

To formalize this, we now fix $\digamma>0$, but we allow $c$ to vary, and in particular we consider what happens as $c\to 0$. We now write $N_c$ to denote the operator $N_{\digamma}$ constructed for the given value of $c$, and we let $E_c$ be the corresponding error. Then $N_c$ is a continuous (w.r.t. the $\Psi^{-1,0}$ topology) family of scattering $\Psi$DOs, and $E_c$ is a family of scattering $\Psi$DO continuous w.r.t. any $\Psi^{-N,-N}$ topology. We also let $\tilde{u}_c = e^{-\digamma/(\tilde{x}+c)}u$, where we interpret $\tilde{u}_c$ as an element of $L^2(\{\tilde{x}>-c\},\frac{d\tilde{x}\,dy}{(\tilde{x}+c)^{n+1}})$; then $\tilde{u}_c$ is continuous in $c$ after identifying $\{\tilde{x}>-c\}$ with  $[0,\infty)\times\mathbb{R}^2$ via the map $(\tilde{x},y_1,y_2)\mapsto(\tilde{x}+c,y_1,y_2)$. We then have $\tilde{u}_c = E_c[\tilde{u}_c]$. We now let $\phi_c$ satisfy $0\le\phi_c\le 1$, $\phi_c\equiv 1$ on $\{x<c\}$, and $\text{supp }\phi_c\subset\{x<2c\}$. Then $\epsilon(c):=\|\phi_cE_c\phi_c\|_{L^2\to L^2}\to 0$, and since $\phi_c\tilde{u}_c = \tilde{u}_c$, we have
\[0 = \tilde{u}_c + \phi_cE_c\phi_c\tilde{u}_c\implies\|\tilde{u}_c\|_{L^2}\le\epsilon(c)\|\tilde{u}_c\|_{L^2},\]
and hence $\tilde{u}_c\equiv 0$ for all sufficiently small $c>0$.

For the elastic tomography problem (and also boundary rigidity) problem, we perform a similar argument, except that there are two operators of interest $0 = N[\nabla u] + \tilde{N}[u]$ (where $\nabla u=(\partial_xu,\partial_yu)$ are the ``usual'' derivatives with respect to the coordinates $(x,y)$, i.e. not the scattering derivatives which have additional factors of $x^2$ and $x$). Similar analysis would then show that $N$ admits an elliptic parametrix, with the remaining term controlled via a Poincar\'e inequality argument.

In \cite{ti}, the authors argued that one can recover $a_{11}$ from the travel time data, assuming the other parameters are known, in large part because the associated operator is elliptic. For the other parameters $a_{33}$ and $E^2$, the associated operators were not elliptic.

In \cite{zou-global}, an analogous problem was studied globally (i.e. without the ``artificial boundary'' method of reducing to a local result). In that paper, the analogous operators for $a_{33}$ and $E^2$ were not elliptic but did resemble parabolic operators in an appropriate sense, and hence they admit a parametrix in a calculus first studied by Boutet de Monvel \cite{bdm}.

We thus follow the geometric setup of \cite{convex, ti}, etc., i.e. we will use the ``artificial boundary'' method to reduce the problem to a local recovery argument. Thus, we fix a point $p$, let $\tilde{x}$ satisfy the properties described above, and work in the regime $c\to 0$. We analyze the operators involved for fixed but sufficiently small $c>0$ (in particular we may take it small enough to satisfy dynamical assumptions). We show that for fixed $c$ the operators involved admit parametrices. Then in the inversion argument we take $c\to 0$ to obtain a local recovery argument. Once the local recovery argument is established, the process of obtaining global recovery is exactly the same as in the previous papers.

The particular operators of interest are $N = L\circ I$, where we take our formal adjoint $L$ to have the form
\begin{equation}
\label{scat-formal-adj}
Lv(x,y) = x^{-2}\int_{\mathbb{S}^2}{\chi(x^{\epsilon}\hat\lambda)\exp\left(-\frac{\digamma^2\hat\lambda^2}{2\alpha}\right)v(\gamma_{\lambda,\omega})\,d\mathbb{S}^2(\lambda,\omega)}.
\end{equation}
where $0<\epsilon<1/2$ and $I = I^{\nu}$ as defined in \eqref{inu} for some choice of material parameter $\nu$. Here, $\gamma_{\lambda,\omega}$ is the trajectory (passing through $(x,y)$) with tangent vector $\lambda\partial_x+\omega\cdot\partial_y$, $\hat\lambda = \lambda/x$, and $\alpha$ satisfies
\[\gamma_{\lambda,\omega}-x = \lambda t + \alpha t^2 + O(t^3).\]
Note that if we let $\hat{t} = t/x$, we can rewrite the above equation as
\[\frac{\gamma_{\lambda,\omega}-x}{x^2} = \hat\lambda\hat{t} + \alpha\hat{t}^2 + O(x\hat{t}^3).\]
We note that our choice of formal adjoint is slightly different than those in previous works: namely, the previous works used a compactly supported cutoff which is sufficiently close to a particular Gaussian; here we use a cutoff which at the boundary is \emph{equal} to the Gaussian of reference, in order to simplify computations at the boundary. As before, let $N_{\digamma} = e^{-\digamma/x}Ne^{\digamma/x}$ for some fixed $\digamma>0$.

To show that $N_{\digamma}$ admits a parametrix, we will need to study the inversion of scattering pseudodifferential operators whose symbols are not quite elliptic, but degenerate in a certain way. We study the geometry associated to the symbols' degeneracy, as follows:

Let $M = (0,\infty)_x\times\mathbb{R}_y^2$, and consider the \emph{scattering cotangent bundle} $^{sc}T^*M$. We can view this as the dual bundle of the scattering tangent bundle, whose sections form the space of scattering vector fields $\mathcal{V}_{sc} = x\mathcal{V}_b = \text{span}_{C^{\infty}(M)}(x^2\partial_x,\partial_{y_1},x\partial_{y_2})$; alternatively we can view the bundle as spanned by the sections $\frac{dx}{x^2}$ and $\frac{dy_i}{x}$, $i=1,2$; see \cite{sslaes} and \cite{grenoble} for more details. Let $\mathscr{X} = \overline{^{sc}T^*M}$ be the compactification of $^{sc}T^*M$; this is the natural space on which symbols of scattering pseudodifferential operators live. Note that $\mathscr{X}$ is a manifold with corners, as it has two boundary hypersurfaces, namely a ``fiber infinity'' obtained by radially compactifying each fiber, and a ``base infinity\footnote{The term ``base infinity'' corresponds to the identification of $M$ as a subset of the upper half-sphere $\mathbb{S}_+^3$ and subsequently identifying its interior with $\mathbb{R}^3$ via radial compactification; such an identification sends a neighborhood of $\{x=0\}$ to a ``neighborhood of infinity'' in $\mathbb{R}^3$.}'' corresponding to fibers over the boundary $\{x=0\}$. Let $S\subset\mathscr{X}$ be the fiberwise span of the covector field $\overline{\xi}$, and let $\Sigma$ be the intersection of $S$ with fiber infinity. If we choose (local) coordinates $(x,y_1,y_2)$ such that $x$ is a coordinate associated to the convex foliation and $y_1$ is a coordinate associated to foliation generated by the one-form $\overline{\xi}$, and let $(x,y,\xi,\eta)$ be the corresponding coordinates on the scattering cotangent bundle (so that $(x,y,\xi,\eta)$ is associated to the element $\left(\xi\frac{dx}{x^2}+\eta\cdot\frac{dy}{x}\right)|_{x,y}$), then
\[S = \text{span }\frac{dy_1}{x} = \{\xi=0,\eta_2=0\}\]
and
\[\Sigma = \{\xi=0,\eta_2 = 0, \rho = 0\}\quad\text{where}\quad\rho = \langle(\xi,\eta)\rangle^{-1}.\]
If we let $\tau = (\rho\xi,\rho\eta_2)$, then $(x,y,\rho,\tau)$ give local coordinates on $\mathscr{X}$ near $\Sigma$, in which case we can write $S = \{\tau = 0\}$ and $\Sigma = \{\rho = 0,\tau = 0\}$.

In \cite{ti}, it was shown that $N_{\digamma}\in\Psi^{-1}_{sc}$, and that the principal symbol $\sigma(N_{\digamma})$ is elliptic at fiber infinity except on $\Sigma$, and that it is elliptic at all finite points as well. In \cite{zou-global}, it was shown that $\sigma(N_{\digamma})$ has a subprincipal term which is purely imaginary and is non-degenerate away from the boundary (i.e. in the region $\{x>0\}$), but that it degenerates as a power of $x$ relative to the principal symbol as $x\to 0$. This suggests that near fiber infinity (i.e. near $\rho = 0$) we can write
\begin{equation}
\label{symbasym1}
\sigma(N_{\digamma}) = a_{-1}\rho + ixa_{-2}\rho^2 + a_{-3}\rho^3
\end{equation}
where $a_{-1}$ and $a_{-2}$ are homogeneous of degree $0$ in the fiber coordinates, $a_{-1}$ is positive except on $\Sigma$ where it degenerates quadratically, $a_{-2}$ is nonzero near $\Sigma$, and $a_{-3}\in S^{0,0}$ is positive near $\Sigma\cap\{x=0\}$. Note that we can write
\[a_{-1}(x,y,\tau) = \langle q(x,y,\tau)\tau,\tau\rangle\]
for some matrix $q$ which is positive definite near $\tau = 0$.

Note that away from $x=0$, the symbol is non-elliptic on $\Sigma$ but has a non-degenerate subprincipal part which is purely imaginary (relative to the real principal symbol), and hence its inversion theory can be covered by a calculus first studied by Boutet de Monvel \cite{bdm}; see \cite{zou-global} for more details. One way to understand this symbol calculus is to start with the fiber-compactified cotangent bundle, blow up $\Sigma$ parabolically (thus creating an additonal face), and consider symbols which are conormal to this blown-up space. Explicitly, we let
\[d_{\Sigma} = (|\tau|^2+\rho)^{1/2},\]
and we blow up $\Sigma$ parabolically, so that $|\tau|^2$ and $\rho$ have the same ``order'' in the blowup, to obtain $\mathscr{X}_1 = [\mathscr{X}:\Sigma]$, with $d_{\Sigma}$ the boundary defining function of the new front face. Let $\beta_1:\mathscr{X}_1\to\mathscr{X}$ be the blow-down map.

\begin{figure}[h]
\begin{center}
\includegraphics[trim={1.25in, 8in, 1.25in, 1.25in},clip]{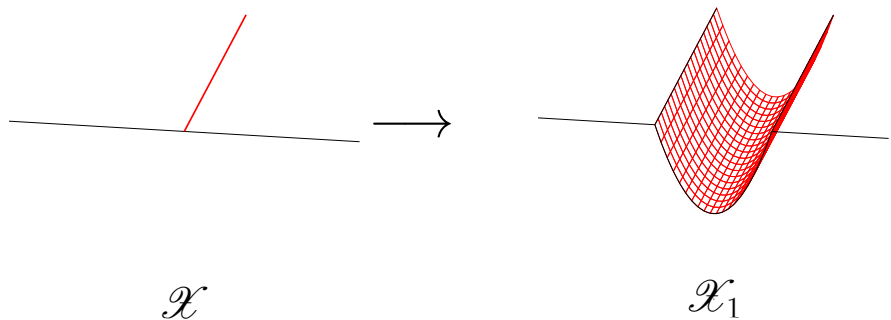}
\end{center}
\caption{The first blow-up.}
\end{figure}
If we let $\tilde\tau = \tau/d_{\Sigma}$ (note that this is a smooth function on the blow-up), then $\rho = d_{\Sigma}^2(1-|\tilde\tau|^2)$, and \eqref{symbasym1} becomes
\begin{equation}
\label{symbasym2}
a = \rho d_{\Sigma}^2\left(\langle q(x,y,\tau)\tilde{\tau},\tilde{\tau}\rangle + ixa_{-2}(1-|\tilde\tau|^2) + a_{-3}d_{\Sigma}^2(1-|\tilde\tau|^2)^2\right).
\end{equation}
Note that, away from $x=0$, and near the new front face $\{d_{\Sigma}=0\}$, the postfactor after $\rho d_{\Sigma}^2$ is uniformly nonzero. Indeed, it would suffice to show the sum of the first two terms is uniformly nonzero, but this follows since the first term is purely real, while the second term is purely imaginary, and neither term vanishes simultaneously. It follows that the symbol is now elliptic with respect to the blown-up space, and one can apply the calculus in \cite{bdm}.

Near $x=0$, the subprincipal symbol degenerates, and thus the postfactor in \eqref{symbasym2} vanishes at 
\[\Gamma := \{x=0,d_{\Sigma}=0,\tilde\tau=0\}.\]
Notice that $\Gamma$ is the lift of $S$ intersected with $x=0$ and $d_{\Sigma}=0$. In particular, it is a sum of terms which vanish as $O(|\tilde\tau|^2)$, $O(x)$, and $O(d_{\Sigma}^2)$. As such, we perform a second blow-up by defining
\[d_{\Gamma} = (x+d_{\Sigma}^2+|\tilde\tau|^2)^{1/2}\]
and letting $\mathscr{X}_2 = [\mathscr{X}_1:\Gamma]$ be the corresponding blow-up where $x^{1/2}$, $d_{\Sigma}$, and $\tilde\tau$ have the same degree of homogeneity, with $d_{\Gamma}$ the new boundary-defining function. Let $\beta_2:\mathscr{X}_2\to\mathscr{X}_1$ be the corresponding blow-down map, and let $\beta:=\beta_1\circ\beta_2:\mathscr{X}_2\to\mathscr{X}$ be the overall blow-down map from $\mathscr{X}_2$ to $\mathscr{X}$.

\begin{figure}[h]
\begin{center}

\includegraphics[trim={1.25in, 8in, 1.25in, 1.25in},clip]{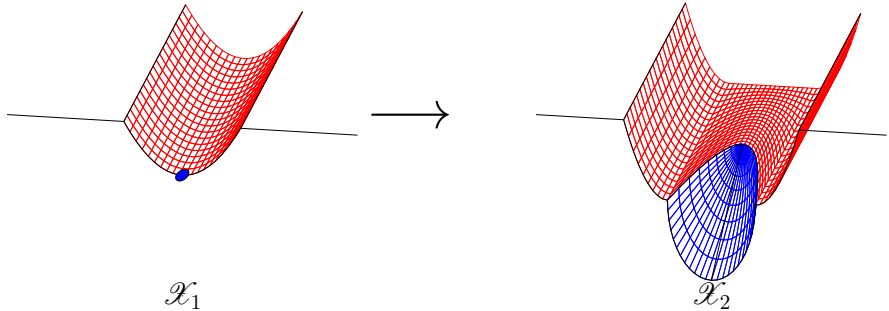}

\end{center}
\caption{The second blow-up.}
\end{figure}

Having defined these blow-ups, we can consider a symbol calculus of symbols conormal to the boundary of $\mathscr{X}_2$, and consider an operator algebra consisting of quantizations of such symbols. This class of operators turns out to be the right class of operators in which one can find a parametrix for $N_{\digamma}$, allowing one to apply the arguments above to prove Theorem \ref{local-recov-thm}, and hence Theorem \ref{global-recov-scat-thm}.

In Section \ref{geo-sec}, we study the geometry of the blow-ups, and in particular for each space we study the space of smooth vector fields tangent to its boundary. In Section \ref{symb-sec}, we study properties of the calculus of symbols conormal to $\partial\mathscr{X}_2$. In Section \ref{op-sec}, we study the corresponding operator algebra, proving that the symbol calculus and operator algebra shares many desirable properties present in the standard scattering calculus. In Section \ref{op-analysis}, we analyze the symbol of $N_{\digamma}$ to show that it is elliptic considered as an element of the operator algebra defined in Section \ref{op-sec}; in particular we show that $N_{\digamma}$ is still a scattering pseudodifferential operator even though we adjusted our formal adjoint from a compactly supported weight to a Gaussian weight. Finally, we apply the machinery constructed to solve the inverse problem in Section \ref{local-recovery-sec}.

\begin{acknowledgments}
The author would like to thank Andr\'as Vasy for helpful comments in this work. The author gratefully acknowledges partial support from the National Science Foundation under grant number DMS-1664683. Much of this article was written as part of the author's Ph.D.\ thesis at Stanford University.
\end{acknowledgments}

\section{Geometry of the blow-ups}
\label{geo-sec}

Recall that, given a smooth map $f:X_1\to X_2$ between two manifolds $X_1$ and $X_2$, and vector fields $V_i$ on $X_i$, that $V_1$ and $V_2$ are \emph{$f$-related} if $f_*\circ V_1 = V_2\circ f$, i.e. for all $p\in X_1$ we have
\[Df_p(V_1)_p = (V_2)_{f(p)}.\]
If $f$ is surjective and $V_1$ and $V_2$ are $f$-related, we will say that $V_2$ \emph{lifts} via $f$ to $V_1$, and we will write $V_2 = f_*V_1$. We will want to study under what conditions a vector field on $\mathscr{X}$ or $\mathscr{X}_1$ lifts, via $\beta_1$ or $\beta_2$ or $\beta$, to a smooth vector field on $\mathscr{X}_1$ or $\mathscr{X}_2$. Note that for $f = \beta_1$, $\beta_2$, or $\beta$, that $f$ is a diffeomorphism of the interiors of its domain and range, and hence if a continuous vector field lifts to a continuous vector field, then such a lift is determined on a dense subset and hence must be unique; as such the base vector field can in some sense be identified with its lift.

To analyze vector fields tangent to the boundary of some manifold (for example, obtained through a blow-up), we use the following observation, formalized as a lemma:
\begin{lemma}
\label{vectestlemma}
Suppose $X$ is a manifold with corners, and let $X^{\circ}$ denote its interior. Suppose $\{y_1,\dots,y_N\}$ is a collection of functions in $C^{\infty}(X)$ such that, for any $p\in\partial X$, there is a neighborhood $U\ni p$ and a subset of functions $\{y_{i_1},\dots,y_{i_n}\}$ such that $(U,(y_{i_1},\dots,y_{i_n}))$ forms a coordinate chart of $X$ near $p$. Suppose $V$ is a $C^{\infty}(X^{\circ})$ vector field on $X^{\circ}$ such that for every $1\le i\le N$, we have that $V(y_i)$, initially defined as a function on $X^{\circ}$, extends smoothly to a $C^{\infty}(X)$ function on all of $X$. Then $V$ extends to a $C^{\infty}(X)$ vector field on all of $X$.
\end{lemma}
\begin{proof}
It suffices to show, for every $p\in\partial X$, that $V$ is a $C^{\infty}(U)$ vector field for some neighborhood $U$ of $p$. Choosing $(U,(y_{i_1},\dots,y_{i_n}))$ to be a coordinate chart, as is possible from the hypothesis, and letting $\{\partial_{y_{i_1}},\dots,\partial_{y_{i_n}}\}$ be the corresponding partial derivative vector fields on $U$, we have that
\[V = \sum_{j=1}^n{V(y_{i_j})\partial_{y_{i_j}}}\quad\text{in }U.\]
Since $V(y_{i_j})\in C^{\infty}(X)$ for each $j$, it follows that the coefficients of $V$ are in $C^{\infty}(U)$, i.e. $V$ is a $C^{\infty}(U)$ vector field, as desired.
\end{proof}

We recall that $\mathscr{X} = \overline{^{sc}T^*M}$ with $M = (0,\infty)_x\times\mathbb{R}^{n-1}_y$ (in practice we take $n=3$, but we can work with a general $n\ge 3$ as well), $\mathscr{X}_1$ is obtained from $\mathscr{X}$ by blowing up $\Sigma = \{\rho = 0, \tau = 0\}$ parabolically, and $\mathscr{X}_2$ is obtained from $\mathscr{X}_1$ by blowing up $\Gamma = \{x=0,d_{\Sigma} = 0,\tilde\tau=0\}$ in a way such that $x^{1/2}$, $d_{\Sigma}$, and $\tilde\tau$ have the same degree of homogeneity.

We first study $\mathscr{X}_1$, and in particular the relationship between vector fields on $\mathscr{X}_1$ tangent to its boundary with the corresponding vector fields on $\mathscr{X}$. Note that on $\mathscr{X}_1$ we have that $d_{\Sigma}$, $\tilde\tau_i = \tau_i/d_{\Sigma}$, and $\tilde{\rho} = \rho/d_{\Sigma}^2$ are smooth, with $|\tilde\tau|^2+\tilde\rho = 1$. Moreover, for any point in $\partial\mathscr{X}_1$ there exists a neighborhood where some subset of the following functions form local coordinates in the neighborhood:
\[\{x,y_1,\dots,y_n,d_{\Sigma},\tilde\tau_1,\dots,\tilde\tau_{n-1},\tilde\rho\}.\]
Indeed, from the equation $|\tilde\tau|^2+\tilde\rho = 1$ we can cover a neighborhood of $\partial\mathscr{X}_1$ by the union
\[\{|\tilde\tau|^2<2/3\}\cup\bigcup_{i=1}^{n-1}{\{\tilde\rho<1/2,|\tilde\tau_i|>\epsilon\}}\]
for a sufficiently small but fixed $\epsilon>0$ (we can in fact take any $\epsilon<\frac{1}{2(n-1)}$). In the first region, we have that $(x,y,d_{\Sigma},\tilde\tau)$ form local coordinates, with $x$ and $d_{\Sigma}$ providing boundary-defining functions for the two faces of $\partial\mathscr{X}_1$ in this region, while in the region $\{\tilde\rho<1/2,|\tilde\tau_i|>\epsilon\}$ local coordinates are given by $(x,y,d_{\Sigma},\tilde\rho,\tilde\tau_1,\dots,\widehat{\tilde\tau_i},\dots,\tilde\tau_{n-1})$, where $\widehat{\tilde\tau_i}$ denotes excluding $\tilde\tau_i$ from the list; moreover in this case $x$, $d_{\Sigma}$, and $\tilde\rho$ provide boundary defining functions of the three faces in this region.

We also recall that $S = \{\tau = 0\}\subset\mathscr{X}$. Let $S_1$ be the lift of $S$ to $\mathscr{X}_1$; explicitly this subset is $\{\tilde\tau = 0\}$.

We now have:
\begin{lemma}
Let $V$ be a $C^{\infty}(\mathscr{X})$-vector field on $\mathscr{X}$ which is tangent to $\partial \mathscr{X}$. Then
\begin{enumerate}
\item $d_{\Sigma}V$ lifts to a $C^{\infty}(\mathscr{X}_1)$-vector field on $\mathscr{X}_1$ which is tangent to $\partial \mathscr{X}_1$.
\item If in addition we have that $V$ is tangent to $S$, then $V$ itself lifts to a $C^{\infty}(\mathscr{X}_1)$-vector field on $\mathscr{X}_1$ tangent to $\partial \mathscr{X}_1$, and furthermore it is tangent to $S_1$.
\end{enumerate}
\end{lemma}

\begin{proof}
In light of Lemma \ref{vectestlemma}, it suffices to compute $d_{\Sigma}V$ (or $V$ itself if it is tangent to $S$) applied to the functions
\begin{equation}
\label{x1coords}
x, y_i, d_{\Sigma}, \tilde\tau_i,\text{ and }\tilde\rho\text{ for }1\le i\le n-1.
\end{equation}
Note that $V$ can be written as a $C^{\infty}(\mathscr{X})$ combination of the vector fields 
\begin{equation}
\label{xvf}
x\partial_x, \partial_{y_j}, \rho\partial_{\rho}, \text{ and }\partial_{\tau_j}\text{ for }1\le j\le n-1,
\end{equation}
while if it is furthermore tangent to $S$, then it can be written as a combination of
\begin{equation}
\label{xtanvf}
x\partial_x, \partial_{y_j}, \rho\partial_{\rho}\text{ and }\tau_i\partial_{\tau_j}\text{ for }1\le i,j\le n-1.
\end{equation}
We now note that $x\partial_x(x) = x$, $\partial_{y_i}(y_i)=1$, and all of the other preceding vector fields applied to $x$ (resp. $y_i$) yield $0$; in addition, $x\partial_x$ and $\partial_{y_i}$ applied to the remaining functions in \eqref{x1coords} also yield zero. It thus remains to compute the vector fields $\rho\partial_{\rho}$ and $d_{\Sigma}\partial_{\tau_j}$ against the functions $d_{\Sigma}$, $\tilde\tau_i$, and $\tilde\rho$. We have
\[\rho\partial_{\rho}(d_{\Sigma}) = \frac{\rho}{2d_{\Sigma}} = \frac{1}{2}d_{\Sigma}\tilde\rho,\quad d_{\Sigma}\partial_{\tau_j}(d_{\Sigma}) = \tau_j = d_{\Sigma}\tilde\tau_j\]
and
\[\rho\partial_{\rho}\left(\frac{\rho}{d_{\Sigma}^2}\right) = \frac{\rho}{d_{\Sigma}^2}-2\frac{\rho}{d_{\Sigma}^3}\frac{\rho}{2d_{\Sigma}} = \tilde{\rho} - \tilde{\rho}^2,\quad d_{\Sigma}\partial_{\tau_j}\left(\frac{\rho}{d_{\Sigma}^2}\right) = -2\frac{\rho}{d_{\Sigma}^3}\tau_j = -2\tilde{\rho}\tilde{\tau_j}\]
and
\[\rho\partial_{\rho}\left(\frac{\tau_i}{d_{\Sigma}}\right) = -\frac{\tau_i}{d_{\Sigma}^2}\frac{\rho}{2d_{\Sigma}} = -\frac{1}{2}\tilde\rho\tilde\tau_i,\quad d_{\Sigma}\partial_{\tau_j}\left(\frac{\tau_i}{d_{\Sigma}}\right) = \delta_{ij} - \frac{\tau_i}{d_{\Sigma}^2}\tau_j = \delta_{ij}-\tilde\tau_i\tilde\tau_j.\]
All of these functions are in $C^{\infty}(\mathscr{X}_1)$, and moreover $\rho\partial_{\rho}$ and $d_{\Sigma}\partial_{\tau_j}$ applied to $d_{\Sigma}$ and $\tilde\rho$ result in $C^{\infty}(\mathscr{X}_1)$ multiples of the respective boundary-defining functions, i.e. they are tangent to $\partial\mathscr{X}_1$. Moreover, since $\tau_i\partial_{\tau_j} = \tilde\tau_i d_{\Sigma}\partial_{\tau_j}$, it follows that if $V$ is tangent to $S$, then $V$ itself applied to the functions in \eqref{x1coords} are in $C^{\infty}(\mathscr{X}_1)$, with $V$ applied to the boundary-defining functions of $\partial\mathscr{X}_1$ resulting in multiples of the respective boundary-defining functions, and since $x\partial_x(\tilde\tau_k) = \partial_{y_j}(\tilde\tau_k) = 0$ and 
\[\rho\partial_{\rho}(\tilde\tau_k) = -\frac{1}{2}\tilde\rho\tilde\tau_k,\quad \tau_i\partial_{\tau_j}(\tilde\tau_k) = \tilde\tau_i d_{\Sigma}\partial_{\tau_j}(\tilde\tau_k) = \tilde\tau_i(\delta_{jk}-\tilde\tau_j\tilde\tau_k),\]
it follows that $V(\tilde\tau_k)$ vanishes at $\tilde\tau=0$. Hence, its $C^{\infty}(\mathscr{X}_1)$-lift is tangent to $S_1 = \{\tilde\tau=0\}$.
\end{proof}

We have a partial converse to the lifting result. Let $\mathcal{V}_1$ denote the space of $C^{\infty}(\mathscr{X}_1)$ vector fields on $\mathscr{X}_1$ which are tangent to $\partial\mathscr{X}_1$, and let $\mathcal{V}_{1,\beta_1}$ denote the vector fields $V\in\mathcal{V}_1$ which are lifts of $C^{\infty}(\mathscr{X})$ vector fields on $\mathscr{X}$. Note that if $V\in\mathcal{V}_{1,\beta_1}$, then $(\beta_1)_*V$ is automatically tangent to $\partial\mathscr{X}$, since the boundary-defining functions of $\partial\mathscr{X}$ are products of boundary-defining functions of $\partial\mathscr{X}_1$.
\begin{lemma}
\label{x1conv}
Any $V\in\mathcal{V}_1$ can be written as a finite sum
\[V = \left(\sum_i{a_iV_i}\right) + \left(\sum_j{b_jd_{\Sigma}W_j}\right)\]
where $a_i,b_j\in C^{\infty}(\mathscr{X}_1)$, $V_i\in\mathcal{V}_{1,\beta_1}$ with $(\beta_1)_*V_i$ tangent to $S$, and $W_j$ are lifts\footnote{More accurately, $d_{\Sigma}W_j$ are lifts in the interior of vector fields which are the product of $d_{\Sigma}$ and a vector field which extends smoothly to a $C^{\infty}(\mathscr{X})$ vector field on $\mathscr{X}$.} of $C^{\infty}(\mathscr{X})$ vector fields on $\mathscr{X}$ which are tangent to $\partial\mathscr{X}$ (but not necessarily to $S$). Moreover, if $V$ is tangent to $S_1$, then only the first sum is needed.
\end{lemma}

\begin{proof} Let $V\in\mathcal{V}_1$. We can write
\[(\beta_1)_*V = (V\rho)\partial_{\rho}+(V\tau)\cdot\partial_{\tau} + (Vx)\partial x + (Vy)\cdot\partial_y\]
where, to make sense of the RHS as a vector field on $\mathscr{X}$, we interpret $V\rho$ as $(\beta_1)^*V\rho$, etc.. Note that all coefficients are in $C^{\infty}(\mathscr{X}_1)$; moreover as $\partial_y$ is tangent to $\partial\mathscr{X}_1$ (and is a $C^{\infty}(\mathscr{X})$ vector field on $\mathscr{X}$), it follows that $(Vy)\cdot\partial_y\in C^{\infty}(\mathscr{X}_1)\otimes\mathcal{V}_{1,\beta_1}$. Furthermore, since $x$ is still a boundary-defining function for one of the faces in $\partial\mathscr{X}_1$, it follows (because $V$ is tangent to $\partial\mathscr{X}_1$) that $Vx = xa$ where $a\in C^{\infty}(\mathscr{X}_1)$, and $x\partial_x$ is tangent to $\mathscr{X}_1$, so $(Vx)\partial x = a\cdot x\partial_x\in C^{\infty}(\mathscr{X}_1)\otimes\mathcal{V}_{1,\beta_1}$. The same ends up being true for $\rho$, since $\{\rho=0\}$ is contained in the boundary of $\partial\mathscr{X}$ (more specifically, since $\rho = \tilde{\rho}d_{\Sigma}^2$ where $\tilde{\rho}\in C^{\infty}(\mathscr{X}_1)$, it follows that
\[V\rho = (V\tilde{\rho})d_{\Sigma}^2 + 2\tilde{\rho}(Vd_{\Sigma})d_{\Sigma};\]
noting that $\tilde{\rho}$ and $d_{\Sigma}$ are boundary-defining functions of two of the faces in $\partial\mathscr{X}$, it follows that $V\tilde{\rho}$ and $Vd_{\Sigma}$ are $C^{\infty}(\mathscr{X}_1)$ multiples of $\tilde{\rho}$ and $d_{\Sigma}$, respectively, so that overall $V\rho$ is a $C^{\infty}(\mathscr{X}_1)$ multiple of $\tilde{\rho}d_{\Sigma}^2 = \rho$). Hence, $(V\rho)\partial_{\rho}$ is a $C^{\infty}(\mathscr{X}_1)$ multiple of $\rho\partial_{\rho}$.

It remains to analyze the term $(V\tau)\cdot\partial_{\tau}$. Since
\[\tau = d_{\Sigma}\tilde\tau\implies V\tau = (Vd_{\Sigma})\tilde\tau + d_{\Sigma}(V\tilde\tau),\]
it follows that $(V\tau)\cdot\partial_{\tau}$ is a sum of the terms $(Vd_{\Sigma})\tilde\tau\cdot\partial_{\tau}$, which is a $C^{\infty}(\mathscr{X}_1)$ multiple of $d_{\Sigma}\tilde\tau\cdot\partial_{\tau} = \tau\cdot\partial_{\tau}$ since $Vd_{\Sigma}$ is a smooth multiple of $d_{\Sigma}$, as well as $d_{\Sigma}(V\tilde\tau)\cdot\partial_{\tau}$, which is a sum of $C^{\infty}(\mathscr{X}_1)$ multiples of $d_{\Sigma}\partial_{\tau_i}$ (i.e. of the form $b_jd_{\Sigma}W_j$ in the lemma statement). 

Finally, we note that if $V$ is tangent to $S_1$, then $V\tilde\tau_i = \sum{a_j\tilde\tau_j}$ for some $a_j\in C^{\infty}(\mathscr{X}_1)$, in which case
\[d_{\Sigma}(V\tilde\tau_i)\partial_{\tau_i} = \left(\sum_j{a_jd_{\Sigma}\tilde\tau_j}\right)\partial_{\tau_i} = \sum_j\left(a_j\tau_j\partial_{\tau_i}\right)\]
and hence is a sum of $C^{\infty}(\mathscr{X}_1)$ multiples of $\tau_j\partial_{\tau_i}$.
\end{proof}
For the second blow-up $\beta_2:\mathscr{X}_2\to\mathscr{X}_1$, we have similar results. We note that on $\mathscr{X}_2$, the functions $d_{\Gamma}$, $\frac{d_{\Sigma}}{d_{\Gamma}}$, $\frac{\tilde\tau_i}{d_{\Gamma}}$, and $\frac{x}{d_{\Gamma}^2}$ are in $C^{\infty}(\mathscr{X}_2)$. Furthermore, the second blow-down $\beta_2$ is a diffeomorphism away from the second front face, i.e. away from $\Gamma = \{x=0,d_{\Sigma}=0,d_{\Gamma}=0\}$ in $\mathscr{X}_1$, it suffices to study the geometry near the second front face in $\mathscr{X}_2$ (i.e. near $\Gamma$ in $\mathscr{X}_1$), as the geometry is unchanged away from there. As such, we note that any point in $\partial\mathscr{X}_2$ near the second front face admits a neighborhood where a subset of the following functions form local coordinates:
\[\left\{\frac{x}{d_{\Gamma}^2},y_1,\dots,y_{n-1},d_{\Gamma},\frac{d_{\Sigma}}{d_{\Gamma}}, \frac{\tilde\tau_1}{d_{\Gamma}},\dots,\frac{\tilde\tau_{n-1}}{d_{\Gamma}}\right\}.\]
Moreover, the boundary-defining functions for $\partial\mathscr{X}_2$ near the second front face are given by $d_{\Gamma}$, $\frac{x}{d_{\Gamma}^2}$, and $\frac{d_{\Sigma}}{d_{\Gamma}}$. The reasoning is similar to the reasoning for the case of $\mathscr{X}_1$.

\begin{lemma}
Let $V\in\mathcal{V}_1$. Then
\begin{enumerate}
\item $d_{\Gamma}V$ lifts to a smooth vector field on $\mathscr{X}_2$ which is tangent to $\partial \mathscr{X}_2$.
\item If in addition we have that $V$ is tangent to $S_1$, then $V$ itself lifts to a smooth vector field on $\mathscr{X}_2$ tangent to $\partial \mathscr{X}_2$.
\end{enumerate}
\end{lemma}

\begin{proof}
Since the second blow-down $\beta_2$ is a diffeomorphism away from $\Gamma$ in $\mathscr{X}_1$, it suffices to analyze $V$ in a neighborhood of $\Gamma$. In particular, $(x,y,d_{\Sigma},\tilde\tau)$ forms local coordinates of $\mathscr{X}_1$ near $\Gamma$, so vector fields in $\mathcal{V}_1$ can be written as $C^{\infty}(\mathscr{X}_1)$ combinations of the vector fields $x\partial_x$, $\partial_{y_j}, d_{\Sigma}\partial_{d_{\Sigma}}$, and $\partial_{\tilde\tau_j}$ for $1\le j\le n-1$ near $\Gamma$; furthermore, vector fields which are also tangent to $S_1$ can be written as a combination of $x\partial_x$, $\partial_{y_j}$, $d_{\Sigma}\partial_{d_{\Sigma}}$, and $\tilde\tau_i\partial_{\tilde\tau_j}$ for $1\le i,j\le n-1$.

We test vector fields like $d_{\Sigma}\partial_{d_{\Sigma}}$, $d_{\Gamma}\partial_{\tilde\tau_j}$, and $x\partial_x$ against the functions
\[d_{\Gamma},\frac{d_{\Sigma}}{d_{\Gamma}},\frac{\tilde\tau_i}{d_{\Gamma}},\frac{x}{d_{\Gamma}^2},\quad 1\le i\le n-1.\]
(Note that we can ignore the vector fields $\partial_{y_j}$ and the coordinates functions $y_i$ for similar reasons as before.)

Thus, we first apply the vector fields to $d_{\Gamma}$, to get
\begin{align*}
d_{\Sigma}\partial_{d_{\Sigma}}(d_{\Gamma}) &= d_{\Sigma}\frac{d_{\Sigma}}{d_{\Gamma}} = d_{\Gamma}\left(\frac{d_{\Sigma}}{d_{\Gamma}}\right)^2, \\
d_{\Gamma}\partial_{\tilde\tau_j}(d_{\Gamma}) &= d_{\Gamma}\frac{\tilde\tau_j}{d_{\Gamma}}, \\
x\partial_x(d_{\Gamma}) &= \frac{x}{2d_{\Gamma}} = d_{\Gamma}\left(\frac{x}{2d_{\Gamma}^2}\right).
\end{align*}
Note that all of the functions on the right-hand sides are smooth on $\mathscr{X}_2$ and vanish at $\{d_{\Gamma} = 0\}$.

We now check
\begin{align*}
d_{\Sigma}\partial_{d_{\Sigma}}\left(\frac{d_{\Sigma}}{d_{\Gamma}}\right) &= d_{\Sigma}\left(\frac{1}{d_{\Gamma}} - \frac{d_{\Sigma}}{d_{\Gamma}^2}\frac{d_{\Sigma}}{d_{\Gamma}}\right) = \frac{d_{\Sigma}}{d_{\Gamma}}\left(1-\left(\frac{d_{\Sigma}}{d_{\Gamma}}\right)^2\right),\\
d_{\Gamma}\partial_{\tilde\tau_j}\left(\frac{d_{\Sigma}}{d_{\Gamma}}\right) &= -d_{\Gamma}\frac{d_{\Sigma}}{d_{\Gamma}^2}\frac{\tilde\tau_j}{d_{\Gamma}} = -\frac{d_{\Sigma}}{d_{\Gamma}}\left(\frac{\tilde\tau_j}{d_{\Gamma}}\right),\\
x\partial_x\left(\frac{d_{\Sigma}}{d_{\Gamma}}\right) &= -x\frac{d_{\Sigma}}{d_{\Gamma}^2}\frac{1}{2d_{\Gamma}} = -\frac{d_{\Sigma}}{d_{\Gamma}}\left(\frac{x}{2d_{\Gamma}^2}\right).
\end{align*}
These functions are smooth on $\mathscr{X}_2$ and vanish at $\{\frac{d_{\Sigma}}{d_{\Gamma}} = 0\}$.

We now check
\begin{align*}
d_{\Sigma}\partial_{d_{\Sigma}}\left(\frac{x}{d_{\Gamma}^2}\right) &= -2d_{\Sigma}\frac{x}{d_{\Gamma}^3}\frac{d_{\Sigma}}{d_{\Gamma}} = -2\frac{x}{d_{\Gamma}^2}\left(\frac{d_{\Sigma}}{d_{\Gamma}}\right)^2,\\
d_{\Gamma}\partial_{\tilde\tau_j}\left(\frac{x}{d_{\Gamma}^2}\right) &= -2d_{\Gamma}\frac{x}{d_{\Gamma}^3}\frac{\tilde\tau_j}{d_{\Gamma}} = -2\frac{x}{d_{\Gamma}^2}\left(\frac{\tilde\tau_j}{d_{\Gamma}}\right),\\
x\partial_x\left(\frac{x}{d_{\Gamma}^2}\right) &= \frac{x}{d_{\Gamma}^2} - 2x\frac{x}{d_{\Gamma}^3}\frac{1}{2d_{\Gamma}} = \frac{x}{d_{\Gamma}^2}\left(1-\frac{x}{d_{\Gamma}^2}\right).
\end{align*}
These functions are smooth on $\mathscr{X}_2$ and vanish at $\{\frac{x}{d_{\Gamma}^2} = 0\}$.

Finally, we check
\begin{align*}
d_{\Sigma}\partial_{d_{\Sigma}}\left(\frac{\tilde\tau_i}{d_{\Gamma}}\right) &= -d_{\Sigma}\frac{\tilde\tau_i}{d_{\Gamma}^2}\frac{d_{\Sigma}}{d_{\Gamma}} = -\left(\frac{d_{\Sigma}}{d_{\Gamma}}\right)^2\frac{\tilde\tau_i}{d_{\Gamma}},\\
d_{\Gamma}\partial_{\tilde\tau_j}\left(\frac{\tilde\tau_i}{d_{\Gamma}}\right) &= -d_{\Gamma}\frac{\tilde\tau_i}{d_{\Gamma}^2}\frac{\tilde\tau_j}{d_{\Gamma}} + d_{\Gamma}\frac{\delta_{ij}}{d_{\Gamma}} = -\frac{\tilde\tau_j}{d_{\Gamma}}\frac{\tilde\tau_i}{d_{\Gamma}}+\delta_{ij},\\
x\partial_x\left(\frac{\tilde\tau_i}{d_{\Gamma}}\right) &= -x\frac{\tilde\tau_i}{d_{\Gamma}^2}\frac{1}{2d_{\Gamma}} = -\frac{\tilde\tau_i}{d_{\Gamma}}\left(\frac{x}{2d_{\Gamma}^2}\right).
\end{align*}
These functions are smooth on $\mathscr{X}_2$. This implies that $d_{\Gamma}V$ lifts to a $C^{\infty}(\mathscr{X}_2)$ vector field on $\mathscr{X}_2$.

Furthermore, since $\tilde\tau_i\partial_{\tilde\tau_j} = \frac{\tilde\tau_i}{d_{\Gamma}}\cdot d_{\Gamma}\partial_{\tilde\tau_j}$, it follows that if $V$ is tangent to $S_1$, then $V$ itself lifts to a $C^{\infty}(\mathscr{X}_2)$ vector field on $\mathscr{X}_2$.
\end{proof}
\begin{remark}
If $V$ is tangent to $S_1$, then similar arguments as before show that the lift of $V$ will be tangent to $\left\{\frac{\tilde\tau}{d_{\Gamma}} = 0\right\}$. However, this fact will not be used, so we do not emphasize this additional fact.
\end{remark}

As a corollary of the above two lemmas, we have the following result regarding the overall blow-down $\beta:\mathscr{X}_2\to\mathscr{X}$:
\begin{cor}
\label{homvf}
Suppose $V$ is a vector field on $\mathscr{X}$ which is tangent to $\partial\mathscr{X}$. Then $d_{\Gamma}d_{\Sigma}V$ lifts to a smooth vector field on $\mathscr{\mathscr{X}}_2$ which is tangent to $\partial \mathscr{\mathscr{X}}_2$. Furthermore, if $V$ is tangent to $S$, then $V$ itself lifts to a smooth vector field on $\mathscr{\mathscr{X}}_2$ which is tangent to $\partial \mathscr{\mathscr{X}}_2$. 
\end{cor}
Examples of vector fields on $\mathscr{X}$ tangent to $\partial\mathscr{X}$ include vector fields which are homogeneous of degree $0$ in the fiber variables $(\xi,\eta)$--such vector fields are commonly used in testing symbolic regularity.

The above corollary shows that certain vector fields tangent to $\partial\mathscr{X}$ (namely those also tangent to $S$) can be lifted to a smooth vector field on $\mathscr{X}_2$ tangent to $\partial\mathscr{X}_2$, and all vector fields tangent to $\partial\mathscr{X}$ can be modified to lift to a smooth vector field on $\mathscr{X}_2$ tangent to $\partial\mathscr{X}_2$ by multiplying by $d_{\Gamma}d_{\Sigma}$.

Let $\mathcal{V}_2$ denote the space of $C^{\infty}(\mathscr{X}_2)$ vector fields on $\mathscr{X}_2$ which are tangent to $\partial\mathscr{X}_2$, and let $\mathcal{V}_{2,\beta_2}$ denote the vector fields $V\in\mathcal{V}_2$ which are $\beta_2$-lifts of $C^{\infty}(\mathscr{X}_1)$ vector fields on $\mathscr{X}_1$. Similar calculations give the corresponding result to Lemma \ref{x1conv} considering blowing down $\beta_2:\mathscr{X}_2\to\mathscr{X}_1$:
\begin{lemma}
\label{x2conv}
Any $V\in\mathcal{V}_2$ can be written as a finite sum
\[V = \left(\sum_i{a_iV_i}\right) + \left(\sum_j{b_jd_{\Gamma}W_j}\right)\]
where $a_i,b_j\in C^{\infty}(\mathscr{X}_2)$, $V_i\in\mathcal{V}_{2,\beta_2}$ with $(\beta_2)_*V$ tangent to $S_1$, and $W_j$ are lifts of $C^{\infty}(\mathscr{X}_1)$ vector fields on $\mathscr{X}_1$ which are tangent to $\partial\mathscr{X}_1$ (but not necessarily to $S_1$).
\end{lemma}

\begin{proof}
We proceed similarly to the Lemma \ref{x1conv}. Since the blowdown $\beta_2$ is a diffeomorphism away from the second front face (corresponding to $\Gamma = \partial S_1$ in $\mathscr{X}_1$), it suffices to analyze $(\beta_2)_*V$ for $V\in\mathcal{V}_2$ in a neighborhood of $S_1$. Valid coordinates in that neighborhood are given by $d_{\Sigma}$, $\tilde\tau_i$, $x$, and $y$. For $V\in\mathcal{V}_2$, we have
\[(\beta_2)_*V = (Vd_{\Sigma})\partial_{d_{\Sigma}} + (V\tilde\tau)\cdot\partial_{\tilde\tau} + (Vx)\partial_x + (Vy)\cdot\partial_y.\]
Since $\partial_y$ is still tangent to $\partial\mathscr{X}_1$, it follows that $(Vy)\cdot\partial_y$ is of the form $a_iV_i$ above. Furthermore, since $\{d_{\Sigma} = 0\}$ and $\{x=0\}$ are still contained in $\partial\mathscr{X}_2$, it follows similar from arguments above that $Vd_{\Sigma}$ and $Vx$ are $C^{\infty}(\mathscr{X}_2)$ multiples of $d_{\Sigma}$ and $x$, respectively, so $(Vd_{\Sigma})\partial_{d_{\Sigma}} + (Vx)\partial_x$ is of the form $\sum{a_iV_i}$ above. Finally, we have
\[(V\tilde\tau)\cdot\partial_{\tilde\tau} = d_{\Gamma}V\left(\frac{\tilde\tau}{d_{\Gamma}}\right)\cdot\partial_{\tilde\tau} + (Vd_{\Gamma})\frac{\tilde\tau}{d_{\Gamma}}\cdot\partial_{\tilde\tau};\]
the first term is of the form $b_jd_{\Gamma}W_j$ while the second term is of the form $a_iV_i$, since $Vd_{\Gamma}$ is a $C^{\infty}(\mathscr{X}_2)$ multiple of $d_{\Gamma}$, i.e. $\frac{Vd_{\Gamma}}{d_{\Gamma}}\in C^{\infty}(\mathscr{X}_2)$. 
\end{proof}
Let $\mathcal{V}_{2,\beta}$ denote vector fields in $\mathcal{V}_2$ which are $\beta$-lifts of vector fields in $\mathcal{V}$. Combining these two lemmas thus gives:
\begin{cor}
Any $V\in\mathcal{V}_2$ can be written as a finite sum
\[V = \left(\sum_i{a_iV_i}\right) + \left(\sum_k{b_jd_{\Gamma}d_{\Sigma}W_j}\right)\]
where $a_i,b_j\in C^{\infty}(\mathscr{X}_2)$, $V_i\in\mathcal{V}_{2,\beta}$ with $\beta_*V_i$ tangent to $S$, and $W_j$ are lifts of $C^{\infty}(\mathscr{X})$ vector fields which are tangent to $\partial\mathscr{X}$ (but not necessarily to $S$).
\end{cor}
\begin{proof}
From Lemma \ref{x2conv}, we have that $V\in\mathcal{V}_2$ can be written as $V = \left(\sum_i{a_i\tilde{V}_i}\right) + \left(\sum_j{b_jd_{\Gamma}\tilde{W}_j}\right)$, with $\tilde{V}_i$ a $\beta_2$-lift of a vector field in $\mathcal{V}_{1}$ which is also tangent to $S_1$, and $d_{\Gamma}\tilde{W}_j$ a $\beta_2$-lift of $d_{\Gamma}$ times a vector field in $\mathcal{V}_1$. From Lemma \ref{x1conv}, we know that $\beta_2^*\tilde{V}_i$ is itself a $\beta_1$-lift of a vector field in $\mathcal{V}$ tangent to $S$, while $\beta_2^*\tilde{W}_j = \sum{\tilde{a}_kV'_{j,k}}+\sum{\tilde{b}_ld_{\Sigma}W'_{j,l}}$, with $V'_{j,k}\in\mathcal{V}_{1,\beta_1}$ and $d_{\Sigma}W'_{j,l}$ a $\beta_1$-lift of $d_{\Sigma}$ times a vector field in $\mathcal{V}$ not necessarily tangent to $S$. It follows that
\[V = \left(\sum_i{a_i\tilde{V}_i} + \sum_{j,k}{b_j\tilde{a}_k d_{\Gamma} V'_{j,k}}\right) + \sum_{j,l}{b_j\tilde{b}_ld_{\Gamma}d_{\Sigma}W'_{j,l}}.\]
The terms in the first parentheses are all in $C^{\infty}(\mathscr{X}_2)\otimes\mathcal{V}_{2,\beta}$, giving the form claimed in the Corollary.
\end{proof}
\begin{cor}
\label{x2genvf}
We have that $\mathcal{V}_2$ is generated over $C^{\infty}(\mathscr{X}_2)$ by the vector fields
\[x\partial_x,\partial_{y_j},\rho\partial_{\rho},d_{\Gamma}d_{\Sigma}\partial_{\tau_j},\quad 1\le j\le n-1.\]
\end{cor}

Note as well that we are usually interested in the behavior of functions on $\mathscr{X}_2$ near the boundary $\partial\mathscr{X}_2$, and hence we can take the boundary-defining functions $\rho$, $x$, $d_{\Sigma}$, and $d_{\Gamma}$ to be bounded from above away from their front faces (and in particular away from the boundary). This is a convention we will adopt from now on.

\section{The \newname-symbol class}
\label{symb-sec}
Much like scattering symbols are, up to factors of $\rho$ and $x$, functions in the interior of $\mathscr{X}$ which extend to distributions conormal to the boundary $\partial \mathscr{X}$, we can consider a symbol class based on distributions conormal to $\partial \mathscr{X}_2$. We recall that the interior of $\mathscr{X}$ is identified with $^{sc}T^*M$ with $M = (0,\infty)\times\mathbb{R}^{n-1}$, and we will identify the interior of $\partial\mathscr{X}_2$ with $^{sc}T^*M$ as well via the blow-down $\beta$.
\begin{definition}
Define $S^{0,0,0,0}(M)$ to consist of functions $a:{^{sc}T^*M}\to\mathbb{C}$ which are smooth on the interior $\mathscr{X}_2^{\circ} = {^{sc}T^*M}$ and are $L^{\infty}$-conormal to the boundary $\partial \mathscr{\mathscr{X}}_2$, i.e. satisfy the estimates
\[|V^{\alpha}a|\le C_{\alpha}\]
on any subset of $^{sc}T^*M$ which is pre-compact\footnote{That is, the estimate should hold with a uniform constant on any set whose projection to the base is pre-compact; in particular on sets of the form $\{x\le C\}$ when using the $(x,y)$ spatial variables. Notably the uniform estimate should hold on the entire fiber, including at fiber infinity.  In most cases we will be concerned with distributions which are compactly supported, in which case we may consider without loss of generality symbols whose spatial support is compact, in which case the estimates are uniform and not just locally uniform.} (when considering $^{sc}T^*M$ as a subset of$\overline{^{sc}T^*M}$) whenever $V^{\alpha}$ is a product of vector fields in $\mathcal{V}_2$, i.e. tangent to $\partial \mathscr{X}_2$. Define 
\[S^{m,l,k,j}(M) := \rho^{-m}x^{-l}d_{\Sigma}^{-k}d_{\Gamma}^{-j}S^{0,0,0,0}(M).\]
We will often write $S^{m,l,k,j}$ in place of $S^{m,l,k,j}(M)$ if the ambient space is clear or not important.
\end{definition}
We will call the collection of symbols belonging in some $S^{m,l,k,j}$ the class of \newname-scattering symbols (or \newname-symbols for short). The notation can be interpreted two ways: first, this is the symbol class associated to making two blowups to the radially compactified cotangent bundle (hence the two $b$'s), and secondly this is in some sense a combination of the \emph{B}outet de Monvel calculus studied in \cite{bdm} in the presence of an additional \emph{b}oundary.

\begin{example}
\begin{enumerate}
\item The space $C^{\infty}(\mathscr{X}_2)$ is contained in $S^{0,0,0,0}$, since functions in $C^{\infty}(\mathscr{X}_2)$ satisfy estimates upon application of \emph{any} vector fields with $C^{\infty}(\mathscr{X}_2)$ coefficients, and not just those which are tangent to $\partial\mathscr{X}_2$.
\item The boundary-defining functions $\rho$ and $x$ belong to $S^{-1,0,0,0}$ and $S^{0,-1,0,0}$, respectively. On the other hand, we can also write
\[\rho = d_{\Sigma}^2\tilde{\rho} = d_{\Sigma}d_{\Gamma}\frac{d_{\Sigma}}{d_{\Gamma}}\tilde{\rho} = d_{\Gamma}^2\left(\frac{d_{\Sigma}}{d_{\Gamma}}\right)^2\tilde\rho\quad\text{where }\tilde\rho = \frac{\rho}{d_{\Sigma}^2},\]
and since $\tilde\rho$ and $\frac{d_{\Sigma}}{d_{\Gamma}}$ both belong to $C^{\infty}(\mathscr{X}_2)\subset S^{0,0,0,0}$, we can also write
\[\rho\in S^{0,0,-2,0}\cap S^{0,0,-1,-1}\cap S^{0,0,0,-2}.\]
Similarly, we have
\[x = d_{\Gamma}^2\tilde{x}\quad\text{where }\tilde{x} = \frac{x}{d_{\Gamma}^2},\]
so $x\in S^{0,0,0,-2}$. Later on we will see that $S^{-1,0,0,0}\subset S^{0,0,-2,0}\subset S^{0,0,-1,-1}\subset S^{0,0,0,-2}$ and $S^{0,-1,0,0}\subset S^{0,0,0,-2}$, so these observations are not contradictory.
\item For any $1\le i\le n-1$, we have
\[\tau_i = d_{\Sigma}\tilde\tau_i = d_{\Sigma}d_{\Gamma}\frac{\tilde\tau_i}{d_{\Gamma}}.\]
Hence, we have 
\begin{equation}
\label{taui00-1-1}
\tau_i\in S^{0,0,-1,-1}.
\end{equation}
\end{enumerate}
\end{example}
The expected algebra property holds:
\begin{prop}
Let $a\in S^{m,l,k,j}$ and $b\in S^{m',l',k',j'}$. Then we have $ab\in S^{m+m',l+l',k+k',j+j'}$.
\end{prop}
\begin{proof}
It suffices to consider the case $(m,l,k,j)=(m',l',k',j')=(0,0,0,0)$. Noting that $V(ab) = (Va)b+a(Vb)$, it follows by induction that for a product $V^{\alpha}$ of vector fields in $\mathcal{V}_2$, we have that $V^{\alpha}(ab)$ is a sum of products of the form $(V^{\beta}a)(V^{\gamma}b)$ for some products $V^{\beta}$, $V^{\gamma}$ of vector fields in $\mathcal{V}_2$. All of these terms satisfy the uniform estimates, so $V^{\alpha}(ab)$ does as well.
\end{proof}
\begin{prop}
Suppose that $a\in S^{m,l,k,j}$ and $V^{\alpha}$ is a product of vector fields in $\mathcal{V}_2$. Then $V^{\alpha}a\in S^{m,l,k,j}$ as well.
\end{prop}
\begin{proof}
If $(m,l,k,j) = (0,0,0,0)$, then this follows easily from the definition. Otherwise, take $a = \rho^{-m}x^{-l}d_{\Sigma}^{-k}d_{\Gamma}^{-j}\tilde{a}$, with $\tilde{a}\in S^{0,0,0,0}$. It suffices to prove the statement for $|\alpha|=1$ as the general case then follows by induction. Since $\rho$, $x$, $d_{\Sigma}$, and $d_{\Gamma}$ are all products of boundary-defining functions of $\mathcal{X}_2$, it follows that for $V\in\mathcal{V}_2$ we have
\[V(\rho^{-m}x^{-l}d_{\Sigma}^{-k}d_{\Gamma}^{-j})\in \rho^{-m}x^{-l}d_{\Sigma}^{-k}d_{\Gamma}^{-j}C^{\infty}(\mathscr{X}_2)\subset S^{m,l,k,j}.\]
Hence
\[Va = V(\rho^{-m}x^{-l}d_{\Sigma}^{-k}d_{\Gamma}^{-j})\tilde{a}+\rho^{-m}x^{-l}d_{\Sigma}^{-k}d_{\Gamma}^{-j}V\tilde{a}\in S^{m,l,k,j}S^{0,0,0,0}+\rho^{-m}x^{-l}d_{\Sigma}^{-k}d_{\Gamma}^{-j}S^{0,0,0,0} = S^{m,l,k,j},\]
as desired.
\end{proof}
We can thus reformulate the definition of $S^{m,l,k,j}$ as follows:
\begin{cor}
A function $a$ is in $S^{m,l,k,j}$ if and only if it satisfies estimates of the form
\[|V^{\alpha}a|\le C_{\alpha}\rho^{-m}x^{-l}d_{\Sigma}^{-k}d_{\Gamma}^{-j}\]
on any pre-compact subset of $^{sc}T^*M$ whenever $V^{\alpha}$ is a product of vector fields in $\mathcal{V}_2$.
\end{cor}
\begin{proof}
From the above proposition, if $a\in S^{m,l,k,j}$, then $V^{\alpha}a\in S^{m,l,k,j}$, and hence it equals $\rho^{-m}x^{-l}d_{\Sigma}^{-k}d_{\Gamma}^{-j}$ times a function in $S^{0,0,0,0}$, which satisfies uniform estimates, so overall $V^{\alpha}a$ wil satisfy the desired estimate above. Conversely, if $a$ satisfies the above estimates, we'd like to show that $\rho^mx^ld_{\Sigma}^kd_{\Gamma}^ja$ is in $S^{0,0,0,0}$. Similar arguments from the above paragraphs show that $V^{\alpha}(\rho^mx^ld_{\Sigma}^kd_{\Gamma}^ja)$ has the form
\[V^{\alpha}(\rho^mx^ld_{\Sigma}^kd_{\Gamma}^ja) = \sum_{\beta\le\alpha}{\rho^mx^ld_{\Sigma}^kd_{\Gamma}^j(V^{\beta}a)b_{\beta}}\]
for some functions $b_{\beta}\in C^{\infty}(\mathscr{X}_2)$. Then by hypothesis we have that $\rho^mx^ld_{\Sigma}^kd_{\Gamma}^j(V^{\beta}a)$ satisfies zeroth-order uniform estimates (i.e. the symbol estimates for $S^{0,0,0,0}$ symbols), while the $b_{\beta}$ also satisfy zeroth-order estimates, and hence overall $V^{\alpha}(\rho^mx^ld_{\Sigma}^kd_{\Gamma}^ja)$ satisfies zeroth-order estimates, i.e. $\rho^mx^ld_{\Sigma}^kd_{\Gamma}^ja\in S^{0,0,0,0}$, as desired.
\end{proof}
Such symbols also satisfy the following estimate:
\begin{prop}
\label{scvfest}
Suppose $W^{\alpha}$ and $V^{\beta}$ are products vector fields on $\mathscr{X}$ which are homogeneous of degree $0$, with the latter a product of vector fields also tangent to $S$. Then for $a\in S^{m,l,k,j}$, we have $W^{\alpha}V^{\beta}a\in S^{m,l,k+|\alpha|,j+|\alpha|}$, and in particular
\[|W^{\alpha}V^{\beta}a|\le C_{\alpha,\beta}\rho^{-m}x^{-l}d_{\Sigma}^{-k-|\alpha|}d_{\Gamma}^{-j-|\alpha|}.\]
\end{prop}
This follows almost immediately from Corollary \ref{homvf}, noting that if $V,W\in\mathcal{V}$ with $V$ tangent to $S$, then $V$ and $d_{\Sigma}d_{\Gamma}W$ are\footnote{More accurately, they lift to vector fields in $\mathcal{V}_2$ on $\mathscr{X}_2$.} also tangent to $\partial\mathscr{X}_2$.

We have the following inclusions:
\begin{prop}
\label{inclusions}
\begin{enumerate}
\item If $m\le m'$, $l\le l'$, $k\le k'$, and $j\le j'$, then $S^{m,l,k,j}\subset S^{m',l',k',j'}$.
\item \label{kexch} If $k\ge k'$, then $S^{m,l,k,j}\subset S^{m+(k-k')/2,l,k',j}$.
\item \label{jexch} If $j\ge j'$, then $S^{m,l,k,j}\subset S^{m,l+(j-j')/2,k,j'}\cap S^{m,l,k+(j-j'),j'}$.
\item We have 
\[S^{m,l,k,j}\subset S^{m+(k-k'+s(j-j')_+)_+/2,l+t(j-j')_+/2,k',j'}\]
for any $m,l,k,j,k',j'$ and any $s,t$ with $s+t = 1$, where $x_+ = \max(0,x)$. In particular, we have $S^{m,l,k,j}\subset S^{m+(k+sj_+)_+/2,l+tj_+/2,0,0}$ for all $s+t=1$, and $S^{m-1,l-1,k-1,j-1}\subset S^{m-1/2,l-1/2,k,j}$.
\end{enumerate}
\end{prop}
Part \ref{kexch} roughly says that we can arbitrarily improve our $k$ order (i.e. order with respect to $d_{\Sigma}$) at the cost of $1/2$ order in $\rho$, while parts \ref{jexch} and 4 says that we can improve our $j$ order (with respect to $d_{\Gamma}$) at the cost of $1/2$ order in either $x$ or $\rho$.
\begin{proof}
\begin{enumerate}
\item We recall that we may assume without loss of generality that $\rho$, $x$, $d_{\Sigma}$, and $d_{\Gamma}$ are bounded away from their associated front faces. Hence, we have
\[\rho^{-m}x^{-l}d_{\Sigma}^{-k}d_{\Gamma}^{-j} = (\rho^{m'-m}x^{l'-l}d_{\Sigma}^{k'-k}d_{\Gamma}^{j'-j})\rho^{-m'}x^{-l'}d_{\Sigma}^{-k'}d_{\Gamma}^{-j'}\le C\rho^{-m'}x^{-l'}d_{\Sigma}^{-k'}d_{\Gamma}^{-j'}\]
since $m'-m,l'-l,k'-k,j'-j\ge 0$, and hence $a\in S^{m,l,k,j}$ implies
\begin{align*} |V^{\alpha}a|\le C_{\alpha}\rho^{-m}x^{-l}d_{\Sigma}^{-k}d_{\Gamma}^{-j}\le C_{\alpha}C\rho^{-m'}x^{-l'}d_{\Sigma}^{-k'}d_{\Gamma}^{-j'},
\end{align*}
i.e. $a\in S^{m',l',k',j'}$.
\item We have $d_{\Sigma}\ge \rho^{1/2}$, so for $s\le 0$ we have $d_{\Sigma}^s\le \rho^{s/2}$. Thus for $k'-k\le 0$ we have $d_{\Sigma}^{k'-k}\le \rho^{-(k-k')/2}$, and hence
\[\rho^{-m}d_{\Sigma}^{-k} = \rho^{-m}d_{\Sigma}^{k'-k}d_{\Sigma}^{-k'}\le \rho^{-(m+(k-k')/2)}d_{\Sigma}^{-k'}.\]
\item Similar reasoning as above, noting that $d_{\Gamma}\ge x^{1/2}$ and $d_{\Gamma}\ge d_{\Sigma}$.
\item We have
\begin{align*}
S^{m,l,k,j}&\subset S^{m,l+t(j-j')_+/2,k,j'+s(j-j')}\\
&\subset S^{m,l+t(j-j')_+/2,k+s(j-j')_+,j'}\\
&\subset S^{m+(k-k'+s(j-j')_+)_+/2,l+t(j-j')_+/2,k',j'},
\end{align*}
using parts \ref{jexch}, \ref{jexch}, and \ref{kexch}, respectively.
\end{enumerate}
\end{proof}
We also have the following ellipticity result:
\begin{prop}
Suppose $a\in S^{m,l,k,j}$ satisfies an ``ellipticity'' estimate
\[|a|\ge c\rho^{-m}x^{-l}d_{\Sigma}^{-k}d_{\Gamma}^{-j}\]
near $\partial\mathscr{X}_2$. Let $r \in\mathbb{Z}$, and $b\in C^{\infty}(^{sc}T^*M)$ agrees with $a^{r}$ near $\partial\mathscr{X}_2$. Then $b\in S^{r m,r l,r k,r j}$. If $a$ takes values in $\mathbb{R}_+$ in a neighborhood of $\partial\mathscr{X}_2$, then the same conclusion holds for any $r\in\mathbb{R}$.
\end{prop}
In particular, if $a\in S^{m,l,k,j}$ satisfies the ellipticity estimate, then any smooth function equaling $a^{-1}$ near $\partial\mathscr{X}_2$ is a symbol in $S^{-m,-l,-k,-j}$.

The properties follow from standard arguments in microlocal analysis, but we repeat them here for convenience:
\begin{proof}
From the generalized power rule
\[V\left(a^{\beta}\right) = \beta a^{\beta-1}Va,\]
it follows by induction that
\[V^{\alpha}\left(a^{r}\right)\in \sum_{i=1}^{|\alpha|}a^{r-i}\prod_{i'=1}^iS^{m,l,k,j},\]
i.e. it is a sum of terms of the form $a^{\beta-i}$ times products of $i$ symbols in $S^{m,l,k,j}$. It follows, from the lower bound on $a$, that
\[\left|V^{\alpha}\left(a^{r}\right)\right|\le C\sum_{i=1}^{|\alpha|}\frac{(\rho^{-m}x^{-l}d_{\Sigma}^{-k}d_{\Gamma}^{-j})^i}{(\rho^{-m}x^{-l}d_{\Sigma}^{-k}d_{\Gamma}^{-j})^{i-r}} = C'\rho^{-r m}x^{-r l}d_{\Sigma}^{-r k}d_{\Gamma}^{-r j},\]
 as desired.
\end{proof}

Note that the above proposition is an ``ellipticity'' statement showing that symbols which are ``elliptic'' in this symbol class admit an inverse which also belongs to this symbol class. To make use of this ellipticity result, we give an example of a class of operators which turn out to be elliptic in this new symbol class. Our operators of interest in the local inverse problem turn out to have symbols of this form, making the following lemma quite applicable:


\begin{lemma}
\label{new-symbol-ell}
Suppose $a$ is a scattering symbol satisfying the asymptotic expansion (in the fiber variables)
\[a\sim\sum_{j\ge 1}{a_{-j}(x,y,\tau)\rho^j}\]
where $a_{-j}\in S^{0,0}$, and the following properties hold regarding $a_{-1}$, $a_{-2}$, and $a_{-3}$:
\begin{itemize}
\item $a_{-1}$ is nonnegative, and it vanishes only on $\Sigma$, where it vanishes non-degenerately quadratically.
\item $a_{-2}=ix\tilde{a}$, where $\tilde{a}\in S^{0,0}$ is real-valued and non-vanishing on $\Sigma$.
\item $a_{-3}$ is positive on $\Sigma\cap\{x=0\}$.
\end{itemize}
Then $a$ is a symbol in $S^{-1,0,-2,-2}$ which is elliptic near fiber infinity, i.e. it satisfies the lower bound $|a|\ge c\rho d_{\Sigma}^2d_{\Gamma}^2$ in a neighborhood of $\rho=0$.
\end{lemma}
Note that the hypotheses of this lemma later on end up being the results of Lemma \ref{new-symbol-prop}.

\begin{proof} The assumptions give uniform bounds $a_{-1}\ge C|\tau|^2$ everywhere, as well as bounds $|a_2|\ge Cx$ near $\Sigma$, and $a_{-3}>C$ in a neighborhood of $\Sigma\cap\{x=0\}$. We first analyze
\[a^0 := a_{-1}(x,y,\tau)\rho + a_{-2}(x,y,0)\rho^2 + a_{-3}(0,y,0)\rho^3.\]
Since $a_{-1}$ and $a_{-3}$ are real (in fact nonnegative), while $a_{-2}$ is purely imaginary, we have
\begin{align*}|a^0|\ge C(|a_{-1}\rho + a_{-3}\rho^3| + |a_{-2}\rho^2|)
&\ge C\rho(|\tau|^2+\rho^2+x\rho) \\
&= C\rho d_{\Sigma}^2\left(|\tilde\tau|^2+d_{\Sigma}^2\left(\frac{\rho}{d_{\Sigma}^2}\right)^2 + x\frac{\rho}{d_{\Sigma}^2}\right) \\
&= C\rho d_{\Sigma}^2d_{\Gamma}^2\left(\frac{|\tilde\tau|^2}{d_{\Gamma}^2} + \frac{d_{\Sigma}^2}{d_{\Gamma}^2}\left(\frac{\rho}{d_{\Sigma}^2}\right)^2 + \frac{x}{d_{\Gamma}^2}\frac{\rho}{d_{\Sigma}^2}\right).
\end{align*}
On the region $\frac{\rho}{d_{\Sigma}^2}\ge 1/2$, the factor in the parentheses on the RHS is at least
\[\frac{|\tilde\tau|^2}{d_{\Gamma}^2} + \frac{d_{\Sigma}^2}{d_{\Gamma}^2}\left(\frac{1}{2}\right)^2 + \frac{x}{d_{\Gamma}^2}\cdot\frac{1}{2}\ge\frac{1}{4}\left(\frac{|\tilde\tau|^2}{d_{\Gamma}^2} + \frac{d_{\Sigma}^2}{d_{\Gamma}^2} + \frac{x}{d_{\Gamma}^2}\right) = \frac{1}{4}.\] Otherwise, since $\frac{\rho}{d_{\Sigma}^2}+|\tilde\tau|^2 = 1$, we have $|\tilde\tau|^2\ge 1/2$, and we can always w.l.o.g. bound $d_{\Gamma}$ from above, so the term $|\tilde\tau|^2/d_{\Gamma}^2$ would be bounded from below. In any case, we have the desired ellipticity bound.

We now analyze how $a$ differs from $a^0$. Note that
\[a_{-2}(x,y,\tau)-a_{-2}(x,y,0) = i\tau\cdot\tilde{a}_{-2}(x,y,\tau)\]
where $\tilde{a}_{-2}$ is a real-(vector-)valued symbol in $S^{0,0}$, and
\[a_{-3}(x,y,\tau)-a_{-3}(0,y,0) = O(x) + O(\tau)\]
and
\[a-\sum_{j=1}^3{a_{-j}\rho^j} = O(\rho^4).\]
Here, the ``big-O'' notation should be interpreted as saying that $O(f)$ means $f$ times a symbol in $S^{0,0}$, or in the case of $O(\tau)$, meaning a $S^{0,0}$ combination of the functions $\tau_i$. Thus overall we have
\[a-a^0 = i\rho^2\tau\cdot\tilde{a}_{-2} + O(x\rho^3) + O(\tau\rho^3) + O(\rho^4).\]
We first show $a^0 + i\rho^2\tau\cdot\tilde{a}_{-2}$ still satisfies the ellipticity bound; note that the additional term is still in $S^{-1,0,-2,-2}$ since $\rho^2\tau_i\in S^{-2,0,-1,-1}\subset S^{-1,0,-2,-2}$). Note that the additional term $i\rho^2\tau\cdot\tilde{a}_{-2}$ is purely imaginary, which helps in that it doesn't ``compete'' with the real-valued terms in $\tilde{a}$. We can argue as follows:
\begin{itemize}
\item Suppose we are in the region where $|\tau\cdot\tilde{a}_{-2}|\le |a_{-2}|/2$. Then the additional term $i\rho^2\tau\cdot\tilde{a}_{-2}$ can be absorbed into the $a_{-2}\rho^2$ term.
\item Otherwise, we have $|\tau\cdot\tilde{a}_{-2}|\ge |a_{-2}|/2$. From the bound $|a_{-2}|\ge Cx$, in this region we necessarily have $|\tau|\ge Cx/2$ for $c=C/2$. This implies $x\le 2C^{-1}d_{\Sigma}|\tilde\tau|\le C^{-1}(d_{\Sigma}^2+|\tilde\tau|^2)$, and thus
\[d_{\Gamma}^2=x+d_{\Sigma}^2+|\tilde\tau|^2\le (1+C^{-1})(d_{\Sigma}^2+|\tilde\tau|^2),\]
and hence $d_{\Sigma}^2+|\tilde\tau|^2\ge (1+C^{-1})^{-1}d_{\Gamma}^2$ in this region. Then
\begin{align*}
|a^0 + i\rho^2\tau\cdot\tilde{a}_{-2}|&\ge|\text{Re}(\tilde{a} + i\rho^2\tau\cdot\tilde{a}_{-2})| \\
&=a_{-1}\rho + a_{-3}\rho^3\\
&\ge C'\rho(|\tau|^2+\rho^2) \\
&= C'\rho d_{\Sigma}^2(|\tilde\tau|^2+d_{\Sigma}^2(1-|\tilde\tau|^2)^2).
\end{align*}
The term in the parentheses is bounded from below by $d_{\Gamma}^2$: indeed, in the region where $|\tilde\tau|\ge 1/2$ we can take $d_{\Gamma}$ to be bounded from above, while in the region where $|\tilde\tau|\le 1/2$ we can use the above estimate to see that in this region we have
\[|\tilde\tau|^2+d_{\Sigma}^2(1-|\tilde\tau|^2)^2\ge |\tilde\tau|^2+\frac{9}{16}d_{\Sigma}^2\ge cd_{\Gamma}^2\]
for $c = \frac{9}{16(1+C^{-1})}$.
\end{itemize}
Thus, $a^0 + i\rho^2\tau\cdot\tilde{a}_{-2}$ satisfies elliptic estimates, regardless of whatever $\tilde{a}_{-2}$ happens to be\footnote{It turns out to vanish at $\{x=0,\tau=0\}$ which would have made computations easier, but this fact is not needed for this argument}.

We now show the remaining terms can be absorbed into the elliptic estimate, at least near fiber infinity $\rho = 0$. Indeed, note that $x\le d_{\Gamma}^2$, $|\tau_j| = d_{\Sigma}|\tilde\tau_j|\le d_{\Sigma}d_{\Gamma}\le d_{\Gamma}^2$, and $\rho \le d_{\Sigma}^2$, and hence
\[x\rho^3\le \rho^2 d_{\Sigma}^2d_{\Gamma}^2, \quad \tau_j\rho^3\le \rho^2d_{\Sigma}^2d_{\Gamma}^2,\quad \rho^4\le \rho^2d_{\Sigma}^4\le \rho^2d_{\Sigma}^2d_{\Gamma}^2.\]
Hence the remaining terms are $\rho O(\rho d_{\Sigma}^2d_{\Gamma}^2)$, and hence the lower bound on the main terms allow the remaining terms to be absorbed into the elliptic estimate, at least near $\rho = 0$. This shows that $a$ satisfies the elliptic estimate near fiber infinity.
\end{proof}
We note that this lemma does not establish ellipticity behavior at finite points; such estimates would need to be proven separately.

We now establish a criterion for a class of symbols to belong to the \newname-symbol class that will be useful for proving certain inclusions:
\begin{lemma}
\label{inc-crit-lemma}
Suppose $S$ is a subspace of functions which are smooth on the interior ${^{sc}T^*M}$ which satisfy the following properties:
\begin{enumerate}
\item Every function $a\in S$ satisfies a bound $|a|\le C_a\rho^{-m}x^{-l}d_{\Sigma}^{-k}d_{\Gamma}^{-j}$ for some constant $C_a$.
\item $S$ is closed under the action of $\mathcal{V}_2$, that is we have $Va\in S$ whenever $a\in S$ and $V\in\mathcal{V}_2$.
\end{enumerate}
Then $S\subset S^{m,l,k,j}$.
\end{lemma}
\begin{proof}
For any product $V^{\alpha}$ of vector fields in $\mathcal{V}_2$, we have that $V^{\alpha}a\in S$ for any $a\in S$ by repeated applications of the second property in the hypothesis. The first property in the hypothesis then gives $|V^{\alpha}a|\le C_{\alpha}\rho^{-m}x^{-l}d_{\Sigma}^{-k}d_{\Gamma}^{-j}$, as desired.
\end{proof}
\begin{cor}
\label{inc-crit-cor}
Suppose $S$ is a subspace of functions which are smooth on the interior ${^{sc}T^*M}$ which satisfy the following properties:
\begin{enumerate}
\item Every function $a\in S$ satisfies a bound $|a|\le C_a\rho^{-m}x^{-l}d_{\Sigma}^{-k}d_{\Gamma}^{-j}$ for some constant $C_a$.
\item For any $a\in S$ and any $V\in\mathcal{V}_2$, we have $Va\in S^{0,0,0,0}\otimes S$.
\end{enumerate}
Then $S\subset S^{m,l,k,j}$. 
\end{cor}
\begin{proof}
Let $\tilde{S} = S^{0,0,0,0}\otimes S$. Then $\tilde{S}$ satisfies the hypothesis in Lemma \ref{inc-crit-lemma}, so $\tilde{S}\subset S^{m,l,k,j}$. Since $S\subset\tilde{S}$ as well since $1\in S^{0,0,0,0}$, it follows that $S\subset S^{m,l,k,j}$, as desired.
\end{proof}

Having introduced the class of symbols we wish to consider, we now consider a slightly modified version of the class of vector fields used to test the regularity of these symbols. We will consider $\mathcal{V}_{2,c} := S^{0,0,0,0}\otimes\mathcal{V}_2$, the space of vector fields whose coefficients are in $S^{0,0,0,0}$ (i.e. functions which are not necessarily smooth up to $\partial\mathscr{X}_2$, but merely conormal to $\partial\mathscr{X}_2$).
\begin{lemma}
For any product $V^{\alpha}$ of vector fields in $\mathcal{V}_{2,c}$, we have $V^{\alpha}a\in S^{m,l,k,j}$ for $a\in S^{m,l,k,j}$.
\end{lemma}
\begin{proof}
It suffices by induction to show this when $|\alpha|=1$. In that case, we can write $V = \sum{b_iV_i}$ with $b_i\in S^{0,0,0,0}$ and $V_i\in\mathcal{V}_2$, in which case $b_iV_ia \in S^{m,l,k,j}$ as $V_ia\in S^{m,l,k,j}$.
\end{proof}
We have a version of Corollary \ref{x2genvf}:
\begin{lemma}
\label{v2cgen}
$\mathcal{V}_{2,c}$ is generated over $S^{0,0,0,0}$ by the vector fields
\begin{equation}
x\partial_x,\partial_{y_j},\rho\partial_{\rho},\tau_i\partial_{\tau_j},\rho\partial_{\tau_j},x^{1/2}\rho^{1/2}\partial_{\tau_j}, 1\le i,j\le n-1.
\end{equation}
\end{lemma}
\begin{proof}
We first show that $\rho\partial_{\tau_j}$ and $x^{1/2}\rho^{1/2}\partial_{\tau_j}$ do belong to $\mathcal{V}_{2,c}$. Indeed, since
$\rho\in S^{-1,0,0,0}\subset S^{0,0,-1,-1}$ and $x^{1/2}\rho^{1/2}\in S^{-1/2,-1/2,0,0}\subset S^{0,0,-1,-1}$, it follows that both $\rho$ and $x^{1/2}\rho^{1/2}$ are $S^{0,0,0,0}$ multiples of $d_{\Sigma}d_{\Gamma}$, and hence these vector fields are a $S^{0,0,0,0}$ multiple of $d_{\Sigma}d_{\Gamma}\partial_{\tau_j}$, which are in $\mathcal{V}_2$.

From Corollary \ref{x2genvf}, it now suffices to show that $d_{\Gamma}d_{\Sigma}\partial_{\tau_j}$ can be obtained as a $S^{0,0,0,0}$ combination of the vector fields $\rho\partial_{\tau_j}$, $x^{1/2}\rho^{1/2}\partial_{\tau_j}$, and $\tau_i\partial_{\tau_j}$; in particular it suffices to show that $d_{\Sigma}d_{\Gamma}$ can be written as a $S^{0,0,0,0}$ combination of $\rho$, $\rho^{1/2}x^{1/2}$, and $\tau_i$. Note
\[d_{\Sigma}-\rho^{1/2} = \frac{d_{\Sigma}^2-\rho}{d_{\Sigma}+\rho^{1/2}} = \frac{|\tau|^2}{d_{\Sigma}+\rho^{1/2}} = \sum_{i=1}^{n-1}{\frac{\tau_i}{d_{\Sigma}+\rho^{1/2}}\tau_i}.\]
We note that $\rho^{1/2}\in S^{-1/2,0,0,0}\subset S^{0,0,-1,0}$, and hence $d_{\Sigma}+\rho^{1/2}\in S^{0,0,-1,0}$, and is elliptic as an element of $S^{0,0,-1,0}$ since $|d_{\Sigma}+\rho^{1/2}|\ge d_{\Sigma}$; hence we have $\frac{1}{d_{\Sigma}+\rho^{1/2}}\in S^{0,0,1,0}$. Moreover, from \eqref{taui00-1-1}, we have $\tau_i\in S^{0,0,-1,-1}$. Thus, we have
\[d_{\Sigma} = \rho^{1/2} + \sum_{i=1}^{n-1}{a_i\tau_i},\quad a_i = \frac{\tau_i}{d_{\Sigma}+\rho^{1/2}} \in S^{0,0,0,-1}.\]
Similarly, we have
\begin{align*}
d_{\Gamma} = x^{1/2} + \frac{d_{\Gamma}^2-x}{d_{\Gamma}+x^{1/2}} &= x^{1/2} + \frac{d_{\Sigma}^2+|\tilde\tau|^2}{d_{\Gamma}+x^{1/2}} \\
&= x^{1/2} + \frac{\rho + (d_{\Sigma}^2+1)|\tilde\tau|^2}{d_{\Gamma}+x^{1/2}} \\
&= x^{1/2} + \frac{\rho^{1/2}}{d_{\Gamma}+x^{1/2}}\rho^{1/2} + \sum_{i=1}^{n-1}{\frac{(1+d_{\Sigma}^2)\tilde\tau_i}{d_{\Gamma}+x^{1/2}}\tilde\tau_i} \\
&= x^{1/2} + \frac{\rho^{1/2}}{d_{\Gamma}+x^{1/2}}\rho^{1/2} + \sum_{i=1}^{n-1}{(1+d_{\Sigma}^2)\frac{d_{\Gamma}}{d_{\Gamma}+x^{1/2}}\frac{\tilde\tau_i}{d_{\Gamma}}\frac{1}{d_{\Sigma}}\tau_i}.
\end{align*}
As before, we have that $d_{\Gamma}+x^{1/2}$ is elliptic as an element of $S^{0,0,0,-1}$, and hence $\frac{1}{d_{\Gamma}+x^{1/2}}\in S^{0,0,0,1}$. Hence,
\[\rho^{1/2}\in S^{0,0,-1,0}\implies \frac{\rho^{1/2}}{d_{\Gamma}+x^{1/2}}\in S^{0,0,-1,1}\subset S^{0,0,0,0}\]
and
\[(1+d_{\Sigma}^2)\frac{d_{\Gamma}}{d_{\Gamma}+x^{1/2}}\frac{\tilde\tau_i}{d_{\Gamma}}\frac{1}{d_{\Sigma}}\in S^{0,0,1,0}.\]
Hence, we have
\[d_{\Gamma}=x^{1/2} + b\rho^{1/2} + \sum_{i=1}^{n-1}c_i\tau_i,\quad b = \frac{\rho^{1/2}}{d_{\Gamma}+x^{1/2}}\in S^{0,0,0,0}, c_i = (1+d_{\Sigma}^2)\frac{d_{\Gamma}}{d_{\Gamma}+x^{1/2}}\frac{\tilde\tau_i}{d_{\Gamma}}\frac{1}{d_{\Sigma}}\in S^{0,0,1,0}.\]
Thus we have
\begin{align*} d_{\Gamma}d_{\Sigma} &= d_{\Gamma}\rho^{1/2} + \sum_{i=1}^{n-1}{d_{\Gamma}a_i\tau_i} \\
&=\left(x^{1/2} + b\rho^{1/2} + \sum_{i=1}^{n-1}c_i\tau_i\right)\rho^{1/2} + \sum_{i=1}^{n-1}{d_{\Gamma}a_i\tau_i} \\
&= x^{1/2}\rho^{1/2} + b\rho + \sum_{i=1}^{n-1}{(c_i\rho^{1/2}+d_{\Gamma}a_i)\tau_i}.
\end{align*}
Noting that $a_i\in S^{0,0,0,-1}\implies d_{\Gamma}a_i\in S^{0,0,0,-2}\subset S^{0,0,0,0}$ and $\rho^{1/2}\in S^{0,0,-1,0}, c_i\in S^{0,0,1,0}\implies \rho^{1/2}c_i\in S^{0,0,0,0}$, we see that $d_{\Gamma}d_{\Sigma}$ can indeed be written as a $S^{0,0,0,0}$ combination of $x^{1/2}\rho^{1/2}$, $\rho$, and $\tau_i$.
\end{proof}
\begin{cor}
\label{bdmvftest}
A function $a:{^{sc}T^*M}\to\mathbb{C}$ which is smooth on the interior belongs to $S^{m,l,k,j}$ if and only if
\begin{equation}
\label{bdmvftesteq}
|W^{\alpha}V^{\beta}a|\le C_{\alpha,\beta}\min(\rho^{-m-|\alpha|}x^{-l}d_{\Sigma}^{-k}d_{\Gamma}^{-j},\rho^{-m-|\alpha|/2}x^{-l-|\alpha|/2}d_{\Sigma}^{-k}d_{\Gamma}^{-j})
\end{equation}
whenever $W^{\alpha}$ and $V^{\beta}$ are products vector fields on $\mathscr{X}$ which are homogeneous of degree $0$, with the latter a product of vector fields also tangent to $S$.
\end{cor}
We compare this result to the definition of the Boutet de Monvel calculus developed in \cite{bdm}, where the symbolic estimates are stated in the above manner of losing decay upon applying vector fields which are not tangent to $S$, as opposed to symbolic estimates formulated in terms of considering vector fields tangent to the boundary of a blow-up.

We now compare the \newname-class of symbols to the ``H\"ormander classes'' of symbols:
\begin{definition}
The symbol class $S_{\delta}^{m,l}$ where $\delta\in[0,1]$ consists of functions $a:{^{sc}T^*M}\to\mathbb{C}$ which are smooth on the interior and satisfy the estimates
\[|V^{\alpha}a|\le C_{\alpha}\rho^{-m-\delta|\alpha|}x^{-l-\delta|\alpha|}\]
whenever $V^{\alpha}$ is a product of vector fields tangent to $\partial \mathscr{X}$.
\end{definition}
Then we have the following relation between the H\"ormander classes of symbols and the \newname-class of symbols, using Corollary \ref{bdmvftest}:
\begin{prop}
\label{hormanderclass}
We have $S_0^{m,l}\subset S^{m,l,0,0}\subset S_{1/2}^{m,l}$. \end{prop}
\begin{proof} 
For the first inclusion\footnote{For the case $m = l = 0$, the statement is just that distributions conormal to $\partial \mathscr{X}$ are also distributions conormal to $\partial \mathscr{X}_2$ by identifying the interiors of $\mathscr{X}$ and $\mathscr{X}_2$. 
This is a standard fact using that the blow-down map $\beta:\mathscr{X}_2\to\mathscr{X}$ is a $b$-map, cf. Chapter 6 in \cite{daomwc}}, we note that any $a\in S_0^{m,l}$ certainly satisfies the estimate in \eqref{bdmvftesteq}, as it satisfies the stronger estimate $|W^{\alpha}V^{\beta}a|\le C\rho^{-m}x^{-l}$ that does not depend on whether the vector fields being applied are tangent to $S$ or not. Hence, by Corollary \ref{bdmvftest}, we have that $a\in S^{m,l,0,0}$.

For the second inclusion, we use Proposition \ref{scvfest} and Part 4 of Proposition \ref{inclusions} (with $s = 0$ and $t = 1$) to note that if $V^{\alpha}$ is a product of vector fields tangent to $\partial \mathscr{X}$, we have
\[ V^{\alpha}a\in S^{m,l,|\alpha|,|\alpha|}\subset S^{m+|\alpha|/2,l+|\alpha|/2,0,0}\]
and hence $|V^{\alpha}a|\le C_{\alpha}\rho^{-(m+|\alpha|/2)}x^{-(l+|\alpha|)}$. This shows $a\in S_{1/2}^{m,l}$.
\end{proof}

For scattering symbols, there is a well-defined notion of principal symbol identifying elements of $S^{m,l}$ with its equivalence class in $S^{m,l}/S^{m-1,l-1}$. It turns out for the \newname-class of symbols that the correct ``lower-order'' space for symbols $S^{m,l,k,j}$ is $S^{m-1,l-1,k+1,j+1}$. Since
\[S^{m-1,l-1,k+1,j+1}\subset S^{m-1/2,l-1/2,k,j},\]
we see that elements in $S^{m-1,l-1,k+1,j+1}$ are at least $1/2$ order better in both $x$ and $\rho$ as elements in $S^{m,l,k,j}$, so it is appropriate to say that those elements are of lower-order. We formalize the subsequent notion of principal symbol as follows:
\begin{definition}
For $a\in S^{m,l,k,j}$, its \emph{principal symbol} $\sigma(a)$ is the equivalence class of $a$ in $S^{m,l,k,j}/S^{m-1,l-1,k+1,j+1}$.
\end{definition}
In addition, we want to characterize the class of ``residual'' symbols which belong in spaces of ``arbitrary lower order,'' i.e. of spaces of the form $S^{m-i,l-i,k+i,j+i}$ for all $i\ge 0$. It turns out that this class coincides with the space of residual scattering symbols:
\begin{prop}
Let
\[S^{-\infty,-\infty} = \cap_{m,l}{S^{m,l}}\]
be the class of residual scattering symbols. Then, for any fixed $(m,l,k,j)$, we have
\[\cap_{i\ge 0}{S^{m-i,l-i,k+i,j+i}} = S^{-\infty,-\infty}.\]
\end{prop}
\begin{proof}
Since
\[S^{-N,-N}\subset S^{-N,-N,0,0}\subset S^{-N+(-k)_+/2,-N+(-j)_+/2,k,j}\]
for any $N> 0$, it follows that we have $S^{-N,-N}\subset S^{m,l,k,j}$ for $N$ sufficiently large. Any element in $S^{-\infty,-\infty}$ belongs to $S^{-N,-N}$ for any $N>0$, so it follows that any element in $S^{-\infty,-\infty}$ belongs in $S^{m,l,k,j}$. The same logic applies to $S^{m-i,l-i,k+i,j+i}$ for any $i\ge 0$, so we have the inclusion $S^{-\infty,-\infty}\subset S^{m-i,l-i,k+i,j+i}$.

Conversely, if $a\in\cap_{i\ge 0}{S^{m-i,l-i,k+i,j+i}}$, we'd like to show that $a\in S^{-N,-N}$ for any fixed $N$. We thus need to show that for any product $V^{\alpha}$ of vector fields tangent to $\partial\mathscr{X}$, we have
\[|V^{\alpha}a|\le C\rho^Nx^N.\]
Using Propositions \ref{hormanderclass} and \ref{inclusions}, we have
\[S^{m-i,l-i,k+i,j+i}\subset S^{m-i+(k+i)_+,l-i+(j+i)_+,0,0}\subset S_{1/2}^{m-i+(k+i)_+,l-i+(j+i)_+},\]
so it follows that by taking $i$ sufficiently large we have
\[a\in S_{1/2}^{-(N+|\alpha|/2),-(N+|\alpha|/2)}.\]
But then this would give
\[|V^{\alpha}a|\le C\rho^{(N+|\alpha|/2)-|\alpha|/2}x^{(N+|\alpha|/2)-|\alpha|/2} = C\rho^Nx^N,\]
as desired.
\end{proof}

To establish the theory of the quantization of these symbols, it will be necessary to consider symbols depending on more variables. We give the formal definition here:

\begin{definition}
\label{fullsymboldef}
Define $S^{0,0,0,0}_{full}$ to consist of functions $a(x,y,X,Y,\xi,\eta): (^{sc}T^*M)_{x,y,\xi,\eta}\times\mathbb{R}^n_{X,Y}\to\mathbb{C}$ which satisfy the estimates $|V^{\alpha}a|\le C$ whenever $V^{\alpha}$ is a product of the following vector fields:
\begin{itemize}
\item Vector fields with no component in the $(X,Y)$ direction which are also tangent to $\partial\mathscr{X}_2\times\mathbb{R}^n_{X,Y}$,
\item The vector fields $\frac{1}{x}\partial_X$ and $\frac{1}{x}\partial_{Y_j}$.
\end{itemize}
Define $S^{m,l,k,j}_{full} := \rho^{-m}x^{-l}d_{\Sigma}^{-k}d_{\Gamma}^{-j}S^{0,0,0,0}_{full}$.
\end{definition}
It follows via the arguments above that we can rephrase the estimates as requiring
\[|V^{\alpha}D^{\beta}_{X,Y}a|\le C_{\alpha,\beta}\rho^{-m}x^{-l+|\beta|}d_{\Sigma}^{-k}d_{\Gamma}^{-j}\]
whenever $V^{\alpha}$ is a product of vector fields on $\mathscr{X}$ which are tangent to $\partial\mathscr{X}_2$ (interpreted as vector fields on $^{sc}T^*M\times\mathbb{R}^n_{X,Y}$ with no $(X,Y)$ components).

\section{Operators quantized by the \newname-symbol class}
\label{op-sec}
\subsection{Oscillatory integral distributions and left reduction}

We want to consider operators whose Schwartz kernels are of the form
\begin{equation}
\label{scat-osc-int}
(2\pi)^{-n}\left(\int_{\mathbb{R}^n}{e^{i(\xi X+\eta\cdot Y)}a(x,y,\xi,\eta)\,d\xi\,d\eta}\right)\,dX\,dY
\end{equation}
with $a\in S^{m,l,k,j}$ where $X = \frac{x'-x}{x^2}$ and $Y = \frac{y'-y}{x}$, and we interpret $(x,y,\xi,\eta)$ as the element $\xi\frac{dx}{x^2}+\eta\cdot\frac{dy}{x}\in ^{sc}T^*M$. (The above expression for the right density should be interpreted as saying that the corresponding operator is given by
\[Au(x,y)=(2\pi)^{-n}\int_{\mathbb{R}^n_{\xi,\eta}\times\mathbb{R}^n_{X,Y}}{e^{i(\xi X+\eta\cdot Y)}a(x,y,\xi,\eta)u(x+x^2X,y+xY)\,d\xi\,d\eta\,dX\,dY};\]
note that $x' = x+x^2X$ and $y' = y+xY$.)

To study the properties of such operators, we first consider the oscillatory integral behavior of the Schwartz kernel viewed as a distribution in the variables $X$ and $Y$, before assigning the interpretation of $X$ and $Y$ as coordinates on the scattering double space. As such, consider oscillatory integral distributions of the form
\begin{equation}
\label{osc-int-dist}
K_a(x,y,X,Y) = (2\pi)^{-n}\int_{\mathbb{R}^n}{e^{i(\xi X+\eta\cdot Y)}a(x,y,\xi,\eta)\,d\xi\,d\eta},\quad a\in S^{m,l,k,j}.
\end{equation}
Here, we consider $X$ and $Y$ just as variables; later we will include the interpretation $X=\frac{x'-x}{x^2}$, and $Y=\frac{y'-y}{x}$. We view this as a family of distributions on $\mathbb{R}^n_{X,Y}$ parametrized by $(x,y)$. We can thus write
\[K_a(x,y,X,Y) = \mathcal{F}^{-1}_{(\xi,\eta)\to(X,Y)}(a(x,y,\cdot,\cdot))\]
where $\mathcal{F}^{-1}_{(\xi,\eta)\to(X,Y)}$ is the inverse Fourier transform in the $(\xi,\eta)$ variables.

It will be convenient to consider more general oscillatory integral distributions of the form
\begin{equation}
\label{osc-int-dist-full}
K_{\tilde{a}}(x,y,X,Y) = (2\pi)^{-n}\int_{\mathbb{R}^n}{e^{i(\xi X+\eta\cdot Y)}\tilde{a}(x,y,X,Y\xi,\eta)\,d\xi\,d\eta},\quad \tilde{a}\in S^{m,l,k,j}_{full}.
\end{equation}
We show that this more general class in fact coincides with our first definition. Formally, we write:

\begin{prop}
\label{left-red}
Let $\tilde{a}\in S^{m,l,k,j}_{full}$ and let $K_{\tilde{a}}$ be defined as in \eqref{osc-int-dist-full}. Then there is a symbol $a\in S^{m,l,k,j}$ such that $K_a = K_{\tilde{a}}$, where $K_a$ is defined as in \eqref{osc-int-dist}. Moreover, $a$ satisfies an asymptotic series formula
\[a(x,y,\xi,\eta)\sim\sum_{\alpha}{\frac{i^{|\alpha|}}{\alpha!}\partial^{\alpha}_{(\xi,\eta)}\partial^{\alpha}_{(X,Y)}\tilde{a}(x,y,0,0,\xi,\eta)},\]
in that $a-\sum_{|\alpha|<N}{\frac{i^{|\alpha|}}{\alpha!}\partial_{(\xi,\eta)}^{\alpha}\partial_{(X,Y)}^{\alpha}\tilde{a}|_{(X,Y)=(0,0)}}\in S^{m-N,l-N,k+N,j+N}$ for every $N$. In particular, $\sigma(a) = \sigma(\tilde{a}|_{(X,Y)=(0,0)})$.
\end{prop}

Note that $a$ can be recovered from $K_a$ by taking the Fourier transform $a(x,y,\xi,\eta) = \mathcal{F}_{(X,Y)\to(\xi,\eta)}(K_a(x,y,\cdot,\cdot))$. Hence, to prove Proposition \ref{left-red}, it suffices to show that ${\mathcal{F}_{(X,Y)\to(\xi,\eta)}(K_{\tilde{a}}(x,y,\cdot,\cdot))}$ belongs in $S^{m,l,k,j}$. The proposition follows from the following two lemmas:

\begin{lemma}
\label{left-red-asymp}
Let $\tilde{a}\in S^{m,l,k,j}_{full}$ and $K_{\tilde{a}}$ be defined as in \eqref{osc-int-dist-full}, and let $a(x,y,\xi,\eta) = \mathcal{F}_{(X,Y)\to(\xi,\eta)}(K_{\tilde{a}}(x,y,\cdot,\cdot))$. Then for every $N>0$ there exists $a_N\in S^{m,l,k,j}$ and $\tilde{a}_N\in S^{-N,-N,0,0}_{full}$ such that
\[K_{\tilde{a}} = K_{a_N} + K_{\tilde{a}_N}.\]
Moreover, the $a_N$ can be chosen to satisfy
\[a_N =  \sum_{|\alpha|\le N'}{\frac{i^{|\alpha|}}{\alpha!}\partial^{\alpha}_{(\xi,\eta)}\partial^{\alpha}_{(X,Y)}\tilde{a}(x,y,0,0,\xi,\eta)}\]
for sufficiently large $N'$.
\end{lemma}

\begin{lemma}
\label{full-symb-weak-est}
Suppose $\tilde{a}\in S^{-N,-N,0,0}_{full}$, and let $a = \mathcal{F}_{(X,Y)\to(\xi,\eta)}{K_{\tilde{a}}(x,y,\cdot,\cdot)}$. Then $a$ satisfies the estimate
\[|V^{\beta}a|\le C\rho^{N-(|\beta|/2+n+1)}x^N\]
whenever $V^{\beta}$ is a product of vector fields tangent to $\partial\mathscr{X}_2$.
\end{lemma}

Lemma \ref{full-symb-weak-est} is essentially a weak version of Proposition \ref{left-red}, since the requirements for symbolicity is that it satisfies the estimate there for an arbitrary number of vector fields. However, the main content of the Proposition is in Lemma \ref{left-red-asymp}, which essentially gives a finite number of terms in an asymptotic expansion with a very weak remainder; the purpose of Lemma \ref{full-symb-weak-est} is to give an estimate on this remainder.

\begin{proof}[Proof of Lemma \ref{left-red-asymp}]
The proof is the same as the classical case: if we write $Z = (X,Y)$ and suppress all other variables, then we use Taylor's theorem with remainder to write
\[\tilde{a}(Z) = \tilde{a}(0) + \sum_{|\alpha|\le N'}{Z^{\alpha}\frac{\partial_Z^{\alpha}\tilde{a}(0)}{\alpha!}} + \sum_{|\beta| = N'+1}{Z^{\beta}R_{\beta}\tilde{a}(Z)}\]
where
\[R_{\beta}\tilde{a}(Z) = \frac{|\beta|}{\beta!}\int_0^1{(1-t)^{|\beta|-1}\partial_Z^{\beta}\tilde{a}(tX)\,dt}.\]
Note that if $\tilde{a}\in S^{m,l,k,j}_{full}$, then $R_{\beta}\tilde{a}\in S^{m,l-|\beta|,k,j}_{full}$ for any $\beta$. It follows that
\begin{align*}
\int{e^{i(\xi X+\eta\cdot Y)}\tilde{a}(Z)\,d\xi\,d\eta} &= \sum_{|\alpha|\le N'}{\frac{1}{\alpha!}\int{Z^{\alpha}e^{i(\xi,\eta)\cdot Z}\partial_Z^{\alpha}\tilde{a}(0)\,d\xi\,d\eta}} + \sum_{|\beta| = N'+1}{\int{Z^{\beta}e^{i(\xi,\eta)\cdot Z}R_{\beta}\tilde{a}(Z)\,d\xi\,d\eta}} \\
&=\int{e^{i(\xi,\eta)\cdot Z}\left(\sum_{|\alpha|\le N'}{\frac{i^{|\alpha|}}{\alpha!}\partial^{\alpha}_{(\xi,\eta)}\partial^{\alpha}_Z\tilde{a}(0)} + \sum_{|\beta| = N'+1}{i^{|\beta|}\partial^{\beta}_{(\xi,\eta)}R_{\beta}(Z)}\right)\,d\xi\,d\eta}.
\end{align*}
Notice that $\xi$ and $\eta$ derivatives are $\rho$ times a vector field tangent to $\partial\mathscr{X}$, so $\partial^{\alpha}_{(\xi,\eta)}\partial^{\alpha}_Z\tilde{a}(0)\in S^{m-|\alpha|,l-|\alpha|,k+|\alpha|,j+|\alpha|}$, and $\partial^{\beta}_{(\xi,\eta)}R_{\beta}(Z)\in S^{m-|\beta|,l-|\beta|,k+|\beta|,j+|\beta|}_{full} = S^{m-(N'+1),l-(N'+1),k+(N'+1),j+(N'+1)}_{full}$ when $|\beta|=N'+1$. We note that, for a fixed $N>0$, we have
\[S^{m-(N'+1),l-(N'+1),k+(N'+1),j+(N'+1)}_{full}\subset S^{-N,-N,0,0}_{full}\]
for sufficiently large $N'$. Thus, for $N'$ sufficiently large, if we let
\[a_N =  \sum_{|\alpha|\le N'}{\frac{i^{|\alpha|}}{\alpha!}\partial^{\alpha}_{(\xi,\eta)}\partial^{\alpha}_Z\tilde{a}(0)}\quad \text{and}\quad \tilde{a}_N = \sum_{|\beta| = N'+1}{i^{|\beta|}\partial^{\beta}_{(\xi,\eta)}R_{\beta}(Z)}\]
then $a_N\in S^{m,l,k,j}$, $\tilde{a}_N\in S^{-N,-N,0,0}_{full}$, and $K_{\tilde{a}} = K_{a_N}+K_{\tilde{a}_N}$, as desired.
\end{proof}
\begin{proof}[Proof of Lemma \ref{full-symb-weak-est}]
We recall from Lemma \ref{v2cgen} that $\mathcal{V}_{2,c}$ is generated over $S^{0,0,0,0}$ by the vector fields
\[x\partial_x, \partial_{y_j},\rho\partial_{\rho},\tau_i\partial_{\tau_j},\rho\partial_{\tau_j},x^{1/2}\rho^{1/2}\partial_{\tau_j}.\]
Rewriting these vector fields in the $(\xi,\eta)$ coordinates, we see that $\mathcal{V}_{2,c}$ is generated over $S^{0,0,0,0}$ by the vector fields
\begin{equation}
\label{vf-gen}
x\partial_x,\partial_{y_j},\partial_{\xi},\partial_{\eta_j},\xi\partial_{\xi},\xi\partial_{\eta_j},\eta_i\partial_{\xi},\eta_i\partial_{\eta_j},\eta_1\partial_{\eta_1},x^{1/2}\langle(\xi,\eta)\rangle^{1/2}\partial_{\xi},x^{1/2}\langle(\xi,\eta)\rangle^{1/2}\partial_{\eta_i}.
\end{equation}
It suffices to prove the lemma when $V^{\beta}$ is a product of the above vector fields. We have the Fourier transform commutation relations
\[\xi\mathcal{F}_{(X,Y)\to(\xi,\eta)}K = \mathcal{F}_{(X,Y)\to(\xi,\eta)}{D_XK}, \eta_i\mathcal{F}_{(X,Y)\to(\xi,\eta)}K = \mathcal{F}_{(X,Y)\to(\xi,\eta)}{D_{Y_i}K}\]
and
\[D_{\xi}\mathcal{F}_{(X,Y)\to(\xi,\eta)}K = \mathcal{F}_{(X,Y)\to(\xi,\eta)}{(-XK)}, D_{\eta_i}\mathcal{F}_{(X,Y)\to(\xi,\eta)}K = \mathcal{F}_{(X,Y)\to(\xi,\eta)}{(-Y_iK)}.\]
We have similar relations in applying differentiation/multiplication in $(X,Y)$ to \eqref{osc-int-dist-full}; the only difference is that the symbol also depends on $(X,Y)$, so applying a derivative in $(X,Y)$ results in another term. Nonetheless, we have
\[XK_{\tilde{a}} = K_{-D_{\xi}\tilde{a}}, Y_iK_{\tilde{a}}  = K_{-D_{\eta_i}\tilde{a}}\]
and
\begin{equation}
\label{XYK}
D_XK_{\tilde{a}} = K_{\xi\tilde{a} + D_X\tilde{a}}, D_{Y_i}K_{\tilde{a}} = K_{\eta_i\tilde{a}+D_{Y_i}\tilde{a}}.
\end{equation}
Thus, we have
\begin{align*}
\eta_i\partial_{\eta_j}(\mathcal{F}_{(X,Y)\to(\xi,\eta)}K_{\tilde{a}}) &= \eta_i\mathcal{F}_{(X,Y)\to(\xi,\eta)}(-Y_jK_{\tilde{a}}) \\
&=\eta_i\mathcal{F}_{(X,Y)\to(\xi,\eta)}{K_{\partial_{\eta_j}\tilde{a}}} \\
&=\mathcal{F}_{(X,Y)\to(\xi,\eta)}{D_{Y_i}K_{\partial_{\eta_j}\tilde{a}}} \\
&=\mathcal{F}_{(X,Y)\to(\xi,\eta)}{K_{\eta_i\partial_{\eta_j}\tilde{a}+D_{Y_i}\partial_{\eta_j}\tilde{a}}},
\end{align*}
i.e. for $V = \eta_i\partial_{\eta_j}$ we have
\[V(\mathcal{F}_{(X,Y)\to(\xi,\eta)}K_{\tilde{a}}) = \mathcal{F}_{(X,Y)\to(\xi,\eta)}K_{V\tilde{a}+\tilde{a}'}\]
where $\tilde{a}'\in S^{-N,-N-1,0,0}_{full}$. It follows that $\tilde{a}'' = V\tilde{a}+\tilde{a}'\in S^{-N,-N,0,0}_{full}$, and hence $V(\mathcal{F}_{(X,Y)\to(\xi,\eta)}K_{\tilde{a}}) = \mathcal{F}_{(X,Y)\to(\xi,\eta)}K_{\tilde{a}''}$ with $\tilde{a}''\in S^{-N,-N,0,0}_{full}$. More generally, the same logic applies if $V$ is one of the following vector fields:
\begin{equation}
\label{vf-gen-nice}
x\partial_x,\partial_{y_j},\partial_{\xi},\partial_{\eta_j},\xi\partial_{\xi},\xi\partial_{\eta_j},\eta_i\partial_{\xi},\eta_i\partial_{\eta_j}, \eta_1\partial_{\eta_1}.
\end{equation}
It follows that if $V^{\beta}$ is a product of the above vector fields (i.e. a product of any of the vector fields in \eqref{vf-gen} except $x^{1/2}\langle(\xi,\eta)\rangle^{1/2}\partial_{\xi}$ and $x^{1/2}\langle(\xi,\eta)\rangle^{1/2}\partial_{\eta_i}$), then
\[V^{\beta}(\mathcal{F}_{(X,Y)\to(\xi,\eta)}K_{\tilde{a}}) = (\mathcal{F}_{(X,Y)\to(\xi,\eta)}K_{\tilde{a}'})\text{ for some }\tilde{a}'\in S^{-N,-N,0,0}_{full}\]
In addition, we have
\[\eta_j(\mathcal{F}_{(X,Y)\to(\xi,\eta)}K_{\tilde{a}}) = \mathcal{F}_{(X,Y)\to(\xi,\eta)}{K_{\eta_j\tilde{a}+D_{Y_j}\tilde{a}}}\]
with $\eta_j\tilde{a}+D_{Y_j}\tilde{a}\in S^{-N+1,-N,0,0}_{full}+S^{-N,-N-1,0,0}_{full}\in S^{-N+1,-N,0,0}_{full}$. Thus in general we have
\[(\xi,\eta)^{\alpha}V^{\beta}(\mathcal{F}_{(X,Y)\to(\xi,\eta)}K_{\tilde{a}}) = (\mathcal{F}_{(X,Y)\to(\xi,\eta)}K_{\tilde{a}'})\text{ for some }\tilde{a}'\in S^{-N+|\alpha|,-N,0,0}_{full}.\]
If $\tilde{a}\in S^{m,l,0,0}_{full}$ with $m<-n$, then the integrand in the Fourier transform is absolutely integrable, and hence we can crudely bound $|(\mathcal{F}_{(X,Y)\to(\xi,\eta)}K_{\tilde{a}'})|\le Cx^{-l}$, and hence
\[|(\xi,\eta)^{\alpha}V^{\beta}(\mathcal{F}_{(X,Y)\to(\xi,\eta)}K_{\tilde{a}})|\le Cx^{-l}\text{ as long as }-N+|\alpha|<-n, i.e. |\alpha|<N-n.\]
Since this holds for all $\alpha$ with $|\alpha|<N-n$, then we have
\begin{equation}
\label{vf-nice-est}
|V^{\beta}(\mathcal{F}_{(X,Y)\to(\xi,\eta)}K_{\tilde{a}})|\le C\langle(\xi,\eta)\rangle^{-(N-n-1)}x^N.
\end{equation}
Finally, for $V = x^{1/2}\langle(\xi,\eta)\rangle^{1/2}\partial_{\xi}$ and $x^{1/2}\langle(\xi,\eta)\rangle^{1/2}\partial_{\eta_i}$, we crudely rewrite these as $(x^{1/2}\langle(\xi,\eta)\rangle^{1/2})(\partial_{\xi})$ and $(x^{1/2}\langle(\xi,\eta)\rangle^{1/2})(\partial_{\eta_i})$. Hence, if we consider a product $V^{\beta}(\mathcal{F}_{(X,Y)\to(\xi,\eta)}K_{\tilde{a}})$ where $V^{\beta}$ may contain products of $x^{1/2}\langle(\xi,\eta)\rangle^{1/2}\partial_{\xi}$ or $x^{1/2}\langle(\xi,\eta)\rangle^{1/2}\partial_{\eta_i}$, then if there are $k$ of such products, we can rewrite $V^{\beta}(\mathcal{F}_{(X,Y)\to(\xi,\eta)}K_{\tilde{a}})$ as a $C^{\infty}(\mathscr{X}_2)$-combination of terms of the form $x^{j/2}\langle(\xi,\eta)\rangle^{j/2}\tilde{V}^{\beta}(\mathcal{F}_{(X,Y)\to(\xi,\eta)}K_{\tilde{a}})$ for $0\le j\le k$, where now $\tilde{V}^{\beta}$ is a product of vector fields of the form \eqref{vf-gen-nice} (by moving factors $x^{1/2}\langle(\xi,\eta)\rangle^{1/2}$ all the way to the left since commutators behave nicely). For each term we can bound
\begin{align*} |x^{j/2}\langle(\xi,\eta)\rangle^{j/2}\tilde{V}^{\beta}(\mathcal{F}_{(X,Y)\to(\xi,\eta)}K_{\tilde{a}})|&\le C\langle(\xi,\eta)\rangle^{j/2}|\tilde{V}^{\beta}(\mathcal{F}_{(X,Y)\to(\xi,\eta)}K_{\tilde{a}})|\\
&\le C\langle(\xi,\eta)\rangle^{|\beta|/2-(N-n-1)}x^N
\end{align*}
using \eqref{vf-nice-est} (note that $j\le k\le |\beta|$). Hence, overall we have
\[|V^{\beta}a|\le C\langle(\xi,\eta)\rangle^{|\beta|/2+n+1-N}x^N \le C\rho^{N-(|\beta|/2+n+1)}x^N,\]
as claimed.
\end{proof}
\begin{proof}[Proof of Proposition \ref{left-red} from Lemma \ref{left-red-asymp} and \ref{full-symb-weak-est}]
For $\tilde{a}\in S^{m,l,k,j}_{full}$, let $a = \mathcal{F}_{(X,Y)\to(\xi,\eta)}(K_{\tilde{a}})$. We want to show that $a\in S^{m,l,k,j}$ and $a-\sum_{|\alpha|<N}{\frac{i^{|\alpha|}}{\alpha!}\partial_{(\xi,\eta)}^{\alpha}\tilde{a}|_{(X,Y)=(0,0)}}\in S^{m-N,l-N,k+N,j+N}$ for any $N$. The former involves showing that, for any product $V^{\beta}$ of vector fields in $\mathcal{V}_2$, we have $|V^{\beta}a|\le C\rho^{-m}x^{-l}d_{\Sigma}^{-k}d_{\Gamma}^{-j}$.

By Lemma \ref{left-red-asymp}, for any $N>0$ we can write $K_{\tilde{a}} = K_{a_N}+ K_{\tilde{a}_N}$ where 
\[a_N =  \sum_{|\alpha|\le N'}{\frac{i^{|\alpha|}}{\alpha!}\partial^{\alpha}_{(\xi,\eta)}\partial^{\alpha}_Z\tilde{a}(0)}\in S^{m,l,k,j}\]
for $N'$ sufficiently large and $\tilde{a}_N\in S^{-N,-N,0,0}$; we'll consider $N$ very large. In particular, we'll take $N$ large enough such that
\[\rho^{N-(|\beta|/2+n+1)}x^N\le C\rho^{-m}x^{-l}d_{\Sigma}^{-k}d_{\Gamma}^{-j}.\]
Then we have
\begin{align*}
a = \mathcal{F}_{(X,Y)\to(\xi,\eta)}(K_{\tilde{a}}) &= \mathcal{F}_{(X,Y)\to(\xi,\eta)}(K_{a_N}) + \mathcal{F}_{(X,Y)\to(\xi,\eta)}(K_{\tilde{a}_N}) \\
&= a_N + \mathcal{F}_{(X,Y)\to(\xi,\eta)}(K_{\tilde{a}_N}).
\end{align*}
Hence, we have
\[|V^{\beta}a|\le |V^{\beta}a_N| + |V^{\beta}\left(\mathcal{F}_{(X,Y)\to(\xi,\eta)}(K_{\tilde{a}_N})\right)|.\]
The term $|V^{\beta}a_N|$ is bounded by $C\rho^{-m}x^{-l}d_{\Sigma}^{-k}d_{\Gamma}^{-j}$ since $a_N\in S^{m,l,k,j}$. By Lemma \ref{full-symb-weak-est}, the second term $|V^{\beta}\left(\mathcal{F}_{(X,Y)\to(\xi,\eta)}(K_{\tilde{a}_N})\right)|$ is bounded by $C\rho^{N-(|\beta|/2+n+1)}x^N$, which in turn is bounded by $C\rho^{-m}x^{-l}d_{\Sigma}^{-k}d_{\Gamma}^{-j}$ since we assume $N$ is sufficiently large. Thus overall we have the bound $|V^{\beta}a|\le C\rho^{-m}x^{-l}d_{\Sigma}^{-k}d_{\Gamma}^{-j}$, as desired. This shows that $a\in S^{m,l,k,j}$.

Moreover, the same argument shows the asymptotic series statement: indeed, for a fixed $N_0$, for sufficiently large $N$ we have $S^{-N,-N,0,0}\subset S^{m-N_0,l-N_0,k+N_0,j+N_0}$, in which case
\[a = a_N + \mathcal{F}_{(X,Y)\to(\xi,\eta)}(K_{\tilde{a}_N})\]
where $\tilde{a}_N\in S_{full}^{-N,-N,0,0}$, and hence from the above result we now know that $\mathcal{F}_{(X,Y)\to(\xi,\eta)}(K_{\tilde{a}_N})\in S^{-N,-N,0,0}\subset S^{m-N_0,l-N_0,k+N_0,j+N_0}$. Since we can arrange for
\[a_N = \sum_{|\alpha|\le N'}{\frac{i^{|\alpha|}}{\alpha!}\partial^{\alpha}_{(\xi,\eta)}\partial^{\alpha}_Z\tilde{a}(0)}\]
for $N'$ sufficiently large (in particular larger than $N_0$), it follows that ${a-\sum_{|\alpha|<N}{\frac{i^{|\alpha|}}{\alpha!}\partial_{(\xi,\eta)}^{\alpha}\partial_{(X,Y)}^{\alpha}\tilde{a}|_{(X,Y)=(0,0)}}}\in S^{m-N_0,l-N_0,k+N_0,j+N_0}$, as desired.
\end{proof}

As an immediate corollary, we note that if $\chi\in C_c^{\infty}(\mathbb{R})$ with $\chi\equiv 1$ in a neighborhood of the origin, then the distributions $\chi(xX,xY)K_a = K_{a(x,y,\xi,\eta)\chi(xX,xY)}$ equals $K_{a'}$ for some $a'\in S^{m,l,k,j}$; moreover $a-a'\in S^{-\infty,-\infty}$. So philosophically we can multiply by a cutoff of the form $\chi(xX,xY)$ without ``changing the symbol $a$ by too much''. This is convenient once we introduce the interpretation of $X$ and $Y$ as the variables on the scattering double space, i.e. $X = \frac{x'-x}{x^2}$ and $Y = \frac{y'-y}{x}$. In that case, we need the cutoff $\chi(xX,xY)$ to guarantee that the distribution $K_a$ vanishes (to infinite order) at the left and right faces ($\{x'>0,x=0\}$ and $\{x'=0,x>0\}$) and the $b$-front face away from the scattering front face (i.e. $\{x+x'=0,\frac{x'-x}{x+x'}\ne 0\}$).

We now explore some basic properties of distributions of the form \eqref{osc-int-dist}.

\begin{lemma}
Let $K = K_a$ as in \eqref{osc-int-dist} for some $a\in S^{m,l,k,j}$. Then:
\begin{itemize}
\item $K$ is smooth on $\{x>0,|(X,Y)|\ne 0\}$.
\item If we let $X_b = xX$ and $Y_b = xY$, then in the limit where $x\to 0$ but $(X_b,Y_b)$ is a fixed nonzero value, we have that $K$ is smooth and vanishes to infinite order as $x\to 0$.
\end{itemize}
\end{lemma}
Once we assign the interpretation of $X = \frac{x'-x}{x^2}$ and $Y = \frac{y'-y}{x}$ on the double space (so that $X_b = \frac{x'-x}{x}$ and $Y_b = y'-y$), then the last condition can be rephrased as saying that $K$ vanishes to infinite order on the b-face of the scattering double space (see e.g. \cite{grieser}).

\begin{proof}
\begin{itemize}
\item On $\{x>0,|(X,Y)|\ne 0\}$, at least one of $xX$ or $xY_j$ ($1\le j\le n-1$) is nonzero; in fact any point admits a neighborhood on which one of these functions is bounded from below. If $xX$ is bounded from below in this neighborhood, we have
\begin{equation}
\label{xXK}
(xX)^NK_a = K_{x^N(-D_{\xi})^Na},\quad\text{with }x^N(-D_{\xi})^Na\in S^{m-N,l-N,k+N,j+N}.
\end{equation}
Taking $N$ sufficiently large so that $S^{m-N,l-N,k+N,j+N}\subset S^{-n-1,0,0,0}$, we see that $(xX)^NK_a$ is continuous, since the integral defining the inverse Fourier transform is integrable. A similar argument applies for the remaining coordinate functions\footnote{Note that for $Y_1$ the situation is in fact even better since $x^N(-D_{\eta_1})^Na\in S^{m-N,l-N,k,j}$.} to show that $K_a$ is continuous in $\{x>0,(X,Y)\ne(0,0)\}$.

Furthermore, from \eqref{XYK}, we see that
\[D_{X,Y}^{\alpha}K_a = K_{\tilde{a}},\quad \tilde{a}\in S^{m+|\alpha|,l,k,j}.\]
From the above result, we have that $D_{X,Y}^{\alpha}$ is also continuous in $\{x>0,|(X,Y)|\ne 0\}$. This shows that $K_a$ is smooth in the region $\{x>0,|(X,Y)|\ne 0\}$.
\item Suppose we are in the regime where $x\to 0$ but $X_b$ is bounded away from zero. Then \eqref{xXK} gives
\[X_b^NK_a = K_{\tilde{a}},\quad \tilde{a}\in S^{m-N,l-N,k+N,j+N}.\]
For any fixed $N_0$, we can take $N$ sufficiently large so that $S^{m-N,l-N,k+N,j+N}\subset S^{-n-1,-N_0,0,0}$, in which case we have
\[|X_b^NK_a|\le Cx^{N_0}.\]
Thus, in the regime where $X_b$ is bounded away from zero, we have that $|K_a|\le Cx^{N_0}$ for some constant $C$ for all $N_0$, i.e. it vanishes to infinite order in $x$. A similar argument holds if $Y_b$ is bounded away from zero.
\end{itemize}
\end{proof}

As such, we have that $K_a$ vanishes to infinite order on the b-face for any $a\in S^{m,l,k,j}$. However, there is no guarantee that it vanishes on the left and right faces $\{x'=0\}\cap\{x>0\}$ and $\{x=0\}\cap\{x'>0\}$, which the Schwartz kernels of (ordinary) scattering $\Psi$DO must satisfy (cf. \cite{sslaes}). Hence, in defining our operator class, we build the vanishing on the left and right faces into the definition, as follows:

\begin{definition}
The class of \newname-pseudodifferential operators $\Psi^{m,l,k,j}$ is defined as the operators whose Schwartz kernels $K(x,y,X,Y)$ satisfy the following properties:
\begin{itemize}
\item $K = K_a$ as defined in \eqref{osc-int-dist} for some $a\in S^{m,l,k,j}$, and
\item $K$ vanishes to infinite order at the left and right faces $\{x'=0\}\cap\{x>0\}$ and $\{x=0\}\cap\{x'>0\}$.
\end{itemize}
\end{definition}

A consequence of Proposition \ref{left-red} is the following: if we start with \emph{any} symbol $a\in S^{m,l,k,j}$, then the oscillatory integral $K_a$ defined in \eqref{osc-int-dist} may not satisfy the property that it vanishes to infinite order at the left and right faces $\{x'=0\}\cap\{x>0\}$ and $\{x=0\}\cap\{x'>0\}$ when we introduce $x'=x+x^2X$ and $y'=y+xY$; however, by multiplying $K_a$ by a cutoff $\chi(xX,xY)$ where $\chi$ is identically $1$ near $0$ but supported in the ball of radius $1$, then the oscillatory distribution \emph{will} vanish to infinite order at the left and right faces; moreover, that oscillatory distribution will still equal $K_{\tilde{a}}$ for some $\tilde{a}\in S^{m,l,k,j}$, and we can be sure that $a$ and $\tilde{a}$ differ by a residual difference in $S^{-\infty,-\infty}$. In other words, we can quantize any symbol in $S^{m,l,k,j}$, so long as we first modify it by some residual symbol in $S^{-\infty,-\infty}$. We summarize the results as follows:

\begin{prop}
Let $a\in S^{m,l,k,j}$. Then there exists $\tilde{a}\in S^{m,l,k,j}$ with $\tilde{a}-a\in S^{-\infty,-\infty}$ such that $K_{\tilde{a}}$ vanishes to infinite order at the left and right faces $\{x'=0\}\cap\{x>0\}$ and $\{x=0\}\cap\{x'>0\}$, i.e. that $K_{\tilde{a}}\in\Psi^{m,l,k,j}$. Moreover, if $\tilde{a}'\in S^{m,l,k,j}$ satisfies the same properties as $\tilde{a}$ in the previous sentence, then $K_{\tilde{a}'}-K_{\tilde{a}}\in\Psi^{-\infty,-\infty}$, i.e. the two operators differ by a (scattering) residual operator.
\end{prop}

As such, we can define:
\begin{definition}
Let $a\in S^{m,l,k,j}$. We will say that $A\in\Psi^{m,l,k,j}$ is \emph{a quantization} of $a$ if $A = K_{\tilde{a}}$, where $\tilde{a}-a\in S^{-\infty,-\infty}$.
\end{definition}
The above proposition then says that for any $a\in S^{m,l,k,j}$, there exists a quantization of $a$, and such a quantization is uniquely determined up to a residual operator in $\Psi^{-\infty,-\infty}$.

\subsection{Adjoints}

We would like these operators to satisfy the usual boundedness condition: namely, that order $0$ symbols are bounded on $L^2$, so that all other symbols are bounded between appropriately defined Sobolev spaces (which can be defined in terms of this calculus). The volume form we take will be that associated to a ``scattering metric''; explicitly in local coordinates such a volume can be given by $d\text{Vol}_{g_{sc}} = \frac{dx\,dy}{x^{(n+1)}}$. To do so, we consider adjoints with respect to the above volume form, and we show these can be expressed as quantizations of another symbol, with ``principal symbol'' equal to $\overline{a}$. This, combined with the composition result obtained in the next section, will give $L^2$-boundedness, via the usual arguments in microlocal analysis.

For convenience, we will consider quantizations of full symbols supported in $\{(xX,xY)\}<1/2$; this is permissible in light of the above discussion.
\begin{prop}
\label{adj-prop}
Let
\[Au(x,y) = (2\pi)^{-n}\int_{\mathbb{R}^n_{X,Y}\times\mathbb{R}^n_{\xi,\eta}}{e^{i(\xi X+\eta\cdot Y)}a(x,y,X,Y,\xi,\eta)u(x+x^2X,y+xY)\,dX\,dY\,d\xi\,d\eta}\]
with $a\in S^{m,l,k,j}_{full}$ supported in $|(xX,xY)|<1/2$. Then the adjoint with respect to the $L^2$ inner product associated with the measure $\frac{dx}{x^{n+1}}dy$ on $[0,\infty)_x\times\mathbb{R}^{n-1}_y$ is another operator quantized by a full symbol whose principal symbol is $(x,y,\xi,\eta)\mapsto\overline{a}(x,y,0,0,\xi,\eta)$.
\end{prop}

\begin{proof}
Since we can write the ``$(X,Y)$-Schwartz kernel'' $K_A(x,y,X,Y)$ (satisfying \\$Au(x,y) = \int{K_A(x,y,X,Y)u(x+x^2X,y+xY)\,dX\,dY}$) as
\[K_A(x,y,X,Y) = (2\pi)^{-n}\int_{\mathbb{R}^n}{e^{i(\xi X+\eta\cdot Y)}a(x,y,X,Y,\xi,\eta)\,d\xi\,d\eta},\]
it will be helpful to obtain an expression for the ``$(X,Y)$-Schwartz kernel'' $K_{A*}$ of the adjoint $A^*$ in terms of $K_A$. We note that
\begin{align*}
\langle Au,v\rangle &= \int{K_A(x,y,X,Y)u(x+x^2X,y+xY)\overline{v(x,y)}\,dX\,dY\,\frac{dx}{x^{n+1}}\,dy} \\
&=\int{K_A\left(x,y,\frac{x'-x}{x^2},\frac{y'-y}{x}\right)u(x',y')\overline{v(x,y)}\,dx\,dy\,dx'\,dy'} \\
&=\int{\overline{K_{A^*}}\left(x',y',\frac{x-x'}{(x')^2},\frac{y-y'}{(y')^2}\right)\overline{v(x,y)}u(x',y')\,dx'\,dy'\,dx\,dy} \\
&= \overline{\langle A^*v,u\rangle} = \langle u,A^*v\rangle,
\end{align*}
as long as
\[\overline{K_{A^*}}\left(x',y',\frac{x-x'}{(x')^2},\frac{y-y'}{(y')^2}\right) = K_A\left(x,y,\frac{x'-x}{x^2},\frac{y'-y}{x}\right).\]
Changing the roles of $(x,y)$ and $(x',y')$, and noting that
\[\frac{x-x'}{(x')^2} = -\frac{(x'-x)/x^2}{(x'/x)^2} = -\frac{X}{(1+xX)^2}\]
and
\[\frac{y-y'}{x'} = -\frac{(y'-y)/x}{x'/x} = -\frac{Y}{1+xX},\]
it follows that
\begin{align*}
&K_{A^*}(x,y,X,Y) \\
&= \overline{K_A}\left(x+x^2X,y+xY,-\frac{X}{(1+xX)^2},-\frac{Y}{1+xX}\right)\\
&=(2\pi)^{-n}\overline{\int_{\mathbb{R}^n}{e^{i(\xi\cdot\frac{-X}{(1+xX)^2}+\eta\cdot\frac{-Y}{1+xX})}a\left(x+x^2X,y+xY,-\frac{X}{(1+xX)^2},-\frac{Y}{1+xX},\xi,\eta\right)\,d\xi\,d\eta}} \\
&=(2\pi)^{-n}\int_{\mathbb{R}^n}{e^{i(\xi X+\eta\cdot Y)}\tilde{a}(x,y,X,Y,\xi,\eta),d\xi\,d\eta},
\end{align*}
where
\begin{align*}
&\tilde{a}(x,y,X,Y,\xi,\eta) \\
&= \overline{a}\left(x+x^2X,y+xY,-\frac{X}{(1+xX)^2},-\frac{Y}{1+xX},(1+xX)^2\xi,(1+xX)\eta\right)(1+xX)^{n+1}.
\end{align*}
Note that in going from the second-to-last line to the last line we rescaled $\xi$ and $\eta$ to $(1+xX)^2\xi$ and $(1+xX)\eta$, respectively.

We can now check that $\tilde{a}\in S^{m,l,k,j}_{full}$. The idea is that $1+xX$ is bounded from above and below in the support of $a$; moreover we can compute quantities like
\begin{align*}
&x\partial_x\tilde{a} = \\
&\left[(x+2x^2X)\partial_x+xY\cdot\partial_y+\frac{2(xX)^2}{(1+xX)^3}\frac{1}{x}\partial_X+\frac{(xX)(xY)}{(1+xX)^2}\frac{1}{x}\partial_Y+2xX(1+xX)\xi\partial_\xi+xX\eta\cdot\partial_{\eta}\right]\overline{a}
\end{align*}
and
\begin{align*}
&\frac{1}{x}\partial_X\tilde{a} =\\
&\left[x\partial_x+\left(-\frac{1}{(1+xX)^2}+\frac{2xX}{(1+xX)^3}\right)\frac{1}{x}\partial_X-\frac{xY}{(1+xX)^2}\frac{1}{x}\partial_Y+2(1+xX)\xi\partial_{\xi}+\eta\cdot\partial_{\eta}\right]\overline{a}.
\end{align*}
to show that $\tilde{a}$ indeed satisfies the desired estimates. We note that the principal symbol of $\tilde{a}$ is given by
\[\tilde{a}(x,y,0,0,\xi,\eta) = \overline{a}(x,y,0,0,\xi,\eta),\]
so that in particular if $a(x,y,X,Y,\xi,\eta)$ is given by a left-reduced symbol $a(x,y,\xi,\eta)$ times a cutoff in $(xX,xY)$, then the principal symbol of the adjoint is given by $\overline{a}(x,y,\xi,\eta)$, as expected.
\end{proof}

\subsection{Quantization: composition}
We now investigate composition. 
We assume all quantities are supported in a fixed (foliated) chart.

Thus, let $(x,y)$, $(x',y')$, and $(x'',y'')$ denote the necessary coordinates. Note if $X = \frac{x'-x}{x^2}$ and $Y = \frac{y'-y}{x}$, then $x' = x+x^2X$ and $y' = y+xY$. We note that if $X' = \frac{x''-x'}{(x')^2}$ and $Y' = \frac{y''-y'}{x'}$, then we can write
\[X' = \frac{x''-x-x^2X}{x^2(1+xX)^2} = \frac{1}{(1+xX)^2}\tilde{X} - \frac{X}{(1+xX)^2}\]
where $\tilde{X} = \frac{x''-x}{x^2}$. Similarly
\[Y' = \frac{y''-y-xY}{x(1+xX)} = \frac{1}{1+xX}\tilde{Y} - \frac{Y}{1+xX},\quad\tilde{Y} = \frac{y''-y}{x^2}.\]
Thus, suppose $A$ and $B$ are quantized by $a$ and $b$. Let $\tilde{A}$ be the operator whose Schwartz kernel is that of $A$ times $\chi(xX,xY)$ where $\chi$ is identically $1$ in a neighborhood of $0$ and is supported in the ball of radius $1/2$; note then that $(A-\tilde{A})B\in\Psi^{-\infty,-\infty}$. We have
\begin{align*}
\tilde{A}Bu(x,y) &= (2\pi)^{-n}\int{e^{i(\xi X+\eta\cdot Y)}a(x,y,\xi,\eta)\chi(xX,xY)Bu(x+x^2X,y+xY)\,dX\,dY\,d\xi\,d\eta} \\
&=(2\pi)^{-2n}\int{e^{i(\xi X+\eta\cdot Y+\xi'X'+\eta'\cdot Y')}a(x,y,\xi,\eta)b(x+x^2X,y+xY,\xi',\eta')\chi(xX,xY)}\\
&\cdot u(x'',y'')\,dX'\,dY'\,d\xi'\,d\eta'\,dX\,dY\,d\xi\,d\eta
\end{align*}
where $(x'',y'') = (x'+(x')^2X',y'+x'Y')$. Changing coordinates from $(X',Y')$ to $(X'',Y'')$ (in which case $dX'\,dY'$ transforms as $\frac{1}{(1+xX)^{n+1}}\,d\tilde{X}\,d\tilde{Y}$, the integral becomes
\begin{align*}(2\pi)^{-2n}\int &e^{i\left(\xi X+\eta\cdot Y+\frac{\xi'}{(1+xX)^2}(\tilde{X}-X)+\frac{\eta'}{1+xX}\cdot(\tilde{Y}-Y)\right)}a(x,y,\xi,\eta)b(x+x^2X,y+xY,\xi',\eta')\chi(xX,xY)\\
&\cdot u(x+x^2\tilde{X},y+x\tilde{Y})\frac{d\tilde{X}\,d\tilde{Y}}{(1+xX)^{n+1}}d\xi'\,d\eta'\,dX\,dY\,d\xi\,d\eta.
\end{align*}
If we now let $\tilde{\xi} = \frac{\xi'}{(1+xX)^2}$ and $\tilde{\eta} = \frac{\eta'}{1+xX}$, then the integral becomes
\begin{align*}(2\pi)^{-2n}\int &e^{i\left(\xi X+\eta\cdot Y+\tilde{\xi}(\tilde{X}-X)+\tilde{\eta}\cdot (\tilde{Y}-Y)\right)}a(x,y,\xi,\eta)b(x+x^2X,y+xY,(1+xX)^2\tilde\xi,(1+xX)\tilde\eta)\\
&\cdot \chi(xX,xY)u(x+x^2\tilde{X},y+x\tilde{Y})\,d\tilde{X}\,d\tilde{Y}\,d\tilde\xi\,d\tilde\eta\,dX\,dY\,d\xi\,d\eta.
\end{align*}
Replacing $\xi$ and $\eta$ by $\xi+\tilde\xi$ and $\eta+\tilde\eta$, it follows that the integral becomes
\[(2\pi)^{-n}\int{e^{i(\tilde\xi\tilde{X}+\tilde\eta\cdot\tilde{Y})}c(x,y,\tilde\xi,\tilde\eta)u(x+x^2\tilde{X},y+x\tilde{Y})d\tilde{X}\,d\tilde{Y}\,d\tilde\xi\,d\tilde\eta}\]
where
\begin{align*}
c(x,y,\tilde\xi,\tilde\eta) =& (2\pi)^{-n}\int e^{i(\xi X+\eta\cdot Y)}a(x,y,\xi+\tilde\xi,\eta+\tilde\eta)\\
&\cdot b(x+x^2X,y+xY,(1+xX)^2\tilde\xi,(1+xX)\tilde\eta)\chi(xX,xY)\,dX\,dY\,d\xi\,d\eta.
\end{align*}
We now note that
\[\tilde{b}(x,y,X,Y,\xi,\eta) = \chi(xX,xY)b(x+x^2X,y+xY,(1+xX)^2\xi,(1+xX)\eta)\]
is a full symbol, with $\tilde{b}(x,y,0,0,\xi,\eta) = b(x,y,\xi,\eta)$: the main calculations to note are
\begin{align*} x\partial_x\tilde{b} = &(xX\partial_x\chi(xX)+xY\cdot\partial_y\chi(xY))b + \chi(xX)\cdot \\
&\left((1+2xX)x\partial_xb + xY\cdot\partial_yb+2xX(1+xX)\xi\partial_{\xi}b+xX\eta\cdot\partial_{\eta}b\right)
\end{align*}
where $b$ and its derivatives are evaluated at $(x+x^2X,y+xY,(1+xX)^2\xi,(1+xX)\eta)$, and
\[\frac{1}{x}\partial_X\tilde{b} = \partial_x\chi(xX,xY)b + \chi(xX,xY)(x\partial_xb+2(1+xX)\xi\cdot\partial_{\xi}b+\eta\cdot\partial_{\eta}b).\]
As such, for $a\in S^{m,l,k,j}$ and $b\in S^{m',l',k',j'}_{full}$, define
\begin{equation}
\label{compsymbol}
(a\# b)(x,y,\xi,\eta) := (2\pi)^{-n}\int_{\mathbb{R}^{2n}}{e^{i(\xi'\cdot X + \eta'\cdot Y)}a(x,y,\xi'+\xi,\eta'+\eta)b(x,y,X,Y,\xi,\eta)\,dX\,dY\,d\xi'\,d\eta'}.
\end{equation}
Note that by replacing $\xi'$ and $\eta'$ by $\xi'-\xi$ and $\eta'-\eta$, we can also rewrite the above oscillatory integral as
\[(a\# b) = (2\pi)^{-n}\int_{\mathbb{R}^{2n}}{e^{i[(\xi'-\xi)\cdot X + (\eta'-\eta)\cdot Y]}a(x,y,\xi',\eta')b(x,y,X,Y,\xi,\eta)\,dX\,dY\,d\xi'\,d\eta'}.\]
The main result is:
\begin{theorem}
\label{comp-symb}
For $a\in S^{m,l,k,j}$ and $b\in S^{m',l',k',j'}_{full}$ we have $a\# b\in S^{M,L,K,J}$, where $M = m+m'$, $L=l+l'$, $K=k+k'$, and $J=j+j'$. Moreover,
\[(a\# b)(x,y,\xi,\eta) - a(x,y,\xi,\eta)b(x,y,0,0,\xi,\eta)\in S^{M-1,L-1,K+1,J+1},\]
i.e. the principal symbol of $a\# b$ equals\footnote{In the sense that they have the same equivalence class under the principal symbol mapping} $a(x,y,\xi,\eta)b(x,y,0,0,\xi,\eta)$.
\end{theorem}
For notational convenience, we will write $a(x,y,\xi,\eta)b(x,y,0,0,\xi,\eta)$, the expected principal symbol, as $ab|_{X=0,Y=0}$.

This will follow from the following lemmas:

\begin{lemma}[Asymptotic series]
\label{asymp}
For any $N$, we have
\begin{equation}
\label{asympeq}
a\# b - \sum_{|\alpha|<N}{\frac{i^{|\alpha|}}{|\alpha|!}\partial^{\alpha}_{\xi,\eta}a\partial^{\alpha}_{X,Y}b|_{X=0,Y=0}} = \sum_{i=1}^{N'}{a_i\# b_i}
\end{equation}
for some choices of $a_i\in S^{m-N,l-N,k+N,j+N}$ and $b_i\in S^{m',l',k',j'}_{full}$.
\end{lemma}

\begin{lemma}[$0$th order estimate]
\label{init}
For $a\in S^{m,l,k,j}$ and $b\in S^{m',l',k',j'}_{full}$, we have $|(a\# b)|\le C\rho^{-M}x^{-L}d_{\Sigma}^{-K}d_{\Gamma}^{-J}$.
\end{lemma}
(The proof in fact uses the asymptotic series expansion in Lemma \ref{asymp}.)

\begin{lemma}[Vector fields]
\label{vf}
Let $V$ be any vector field tangent to $\partial\mathscr{X}_2$. Then
\[V(a\# b) = \sum_{i=1}^N{a_i\# b_i}\]
for some choices of $a_i\in S^{m_i,l_i,k_i,j_i}$ and $b_i\in S^{m'_i,l'_i,k'_i,j'_i}_{full}$ satisfying $m_i+m'_i = M$, $l_i+l'_i = L$, $k_i+k'_i = K$, and $j_i+j'_i = J$.
\end{lemma}

\begin{proof}[Proof of Theorem \ref{comp-symb}]
Note that directly from Lemma \ref{init} we have that $a\# b$ satisfies the zeroth order symbol estimate $|(a\# b)|\le C\rho^{-M}x^{-L}d_{\Sigma}^{-K}d_{\Gamma}^{-J}$, so it suffices to show the same estimate holds after applying a product of vector fields $V_1\dots V_N$ which are tangent to $\partial\mathscr{X}_2$. But by applying Lemma \ref{vf} repeatedly, we see that
\[V_1\dots V_N(a\# b) = \sum_{i=1}^{N'}{a_i\# b_i}\]
where $a_i\in S^{m_i,l_i,k_i,j_i}$ and $b_i\in S^{m'_i,l'_i,k'_i,j'_i}_{full}$ with $m_i+m'_i = M$, $l_i+l'_i = L$, $k_i+k'_i = K$, and $j_i+j'_i = J$. Then, applying the zeroth order estimate in Lemma \ref{init} gives
\[|V_1\dots V_N(a\# b)|\le C\rho^{-M}x^{-L}d_{\Sigma}^{-K}d_{\Gamma}^{-J},\]
as desired. Hence, we have shown that $a\# b\in S^{M,L,K,J}$ for $a\in S^{m,l,k,j}$ and $b\in S^{m',l',k',j'}_{full}$. 

To show the principal symbol statement, we observe that applying Lemma \ref{asymp} with $N=1$ gives $a\#b - ab|_{X=0,Y=0} = \sum_{i=1}^{N'}{a_i\# b_i}$ with $a_i\in S^{m-1,l-1,k+1,j+1}$ and $b_i\in S^{m',l',k',j'}_{full}$. Applying the conclusion of the above paragraph gives $a\#b-ab|_{X=0,Y=0}\in S^{M-1,L-1,K+1,J+1}$, as desired.
\end{proof}

\begin{proof}[Proof of Lemma \ref{asymp}]
The terms in the asymptotic series can be derived by applying the method of stationary phase; however in this case we can make a more direct integration by parts argument similar to the left reduction argument: indeed, letting $Z=(X,Y)$ and suppressing the variables $x$, $y$, $\xi$, and $\eta$, we have
\[b(Z) =  \sum_{0\le|\alpha|<N}{Z^{\alpha}\frac{\partial_Z^{\alpha}b(0)}{\alpha!}} + \sum_{|\beta| = N}{Z^{\beta}R_{\beta}b(Z)}\]
where
\[R_{\beta}b(Z) = \frac{|\beta|}{\beta!}\int_0^1{(1-t)^{|\beta|-1}\partial_Z^{\beta}b(tX)\,dt}.\]
Note that $\partial_Z^{\alpha}b(0)\in S^{m',l'-|\alpha|,k',j'}$ and $R_{\beta}b\in S^{m',l'-|\beta|,k',j'}_{full}$. It follows that
\begin{align*}
a\# b &= \sum_{0\le|\alpha|\le n}{(2\pi)^{-n}\int_{\mathbb{R}^n\times\mathbb{R}^n}{e^{i[(\xi'-\xi)X+(\eta'-\eta)\cdot Y]}a(\xi',\eta')(X,Y)^{\alpha}\frac{(\partial_{X,Y}^{\alpha}b)|_{(X,Y)=(0,0)}}{\alpha!}\,d\xi'\,d\eta'\,dX\,dY}} \\
&+ \sum_{|\beta|=N}{(2\pi)^{-n}\int_{\mathbb{R}^n\times\mathbb{R}^n}{e^{i[(\xi'-\xi)X+(\eta'-\eta)\cdot Y]}a(\xi',\eta')(X,Y)^{\beta}R_{\beta}(X,Y)\,d\xi'\,d\eta'\,dX\,dY}}.
\end{align*}
To evaluate the terms in the first sum, note that $(\partial_{X,Y}^{\alpha}b)|_{(X,Y) = (0,0)}$ is independent of the variables of integration, so we can write the term as
\begin{align*}
&(2\pi)^{-n}\frac{(\partial_{X,Y}^{\alpha}b)|_{(X,Y)=(0,0)}}{\alpha!}\int_{\mathbb{R}^n\times\mathbb{R}^n}{e^{i[(\xi'-\xi)X+(\eta'-\eta)\cdot Y]}a(\xi',\eta')(X,Y)^{\alpha}\,d\xi'\,d\eta'\,dX\,dY}\\
&=(2\pi)^{-n}\frac{i^{|\alpha|}(\partial_{X,Y}^{\alpha}b)|_{(X,Y)=(0,0)}}{\alpha!}\int_{\mathbb{R}^n\times\mathbb{R}^n}{(-\partial_{\xi,\alpha})^{\alpha}\left(e^{i[(\xi'-\xi)X+(\eta'-\eta)\cdot Y]}\right)a(\xi',\eta')\,d\xi'\,d\eta'\,dX\,dY} \\
&=(2\pi)^{-n}\frac{i^{|\alpha|}(\partial_{X,Y}^{\alpha}b)|_{(X,Y)=(0,0)}}{\alpha!}\int_{\mathbb{R}^n\times\mathbb{R}^n}{e^{i[(\xi'-\xi)X+(\eta'-\eta)\cdot Y]}\partial_{\xi,\eta}^{\alpha}a(\xi',\eta')\,d\xi'\,d\eta'\,dX\,dY} \\
&=\frac{i^{|\alpha|}(\partial_{X,Y}^{\alpha}b)|_{(X,Y)=(0,0)}}{\alpha!}\int_{\mathbb{R}^n}{\partial_{\xi,\eta}^{\alpha}a(\xi',\eta')\delta(\xi'-\xi,\eta'-\eta)\,d\xi'\,d\eta'} \\
&=\frac{i^{|\alpha|}}{\alpha!}\partial_{\xi,\eta}^{\alpha}a(\xi,\eta)\partial_{X,Y}^{\alpha}b|_{(X,Y) = (0,0)}.
\end{align*}
For each term in the second sum, we apply the same integration by parts argument to rewrite the term as
\[(2\pi)^{-n}i^{|\beta|}\int_{\mathbb{R}^n\times\mathbb{R}^n}{e^{i[(\xi'-\xi)X+(\eta'-\eta)\cdot Y]}\partial_{\xi,\eta}^{\beta}a(\xi',\eta')R_{\beta}(X,Y)\,d\xi'\,d\eta'\,dX\,dY}.\]
By definition, this is $(i^{|\beta|}\partial_{\xi,\eta}^{\beta}a)\#(R_{\beta}b)$. Noting that $\partial_{\xi,\eta}^{\beta}a\in S^{m-|\beta|,l,k+|\beta|,j+|\beta|} = S^{m-N,l,k+N,j+N}$ and $R_{\beta}b\in S^{m',l'-N,k',j'}$, the desired statement follows by noting that we can multiply $\partial_{\xi,\eta}^{\beta}a$ and divide $R_{\beta}b$ by $x^N$ without changing the overall integral.
\end{proof}

\begin{proof}[Proof of Lemma \ref{init}]
We first prove the much weaker statement: for $m_0\in\mathbb{N}$, we have
\[m<-n-m_0, a\in S^{m,l,0,0}, b\in S^{m',l',k',j'}\implies |a\# b|\le C\rho^{m_0-m'}x^{-(l+l')}d_{\Sigma}^{-k'}d_{\Gamma}^{-j'}.\]
By multiplying $a$ by a power of $x$ and $b$ by powers of $\rho$, $x$, $d_{\Sigma}$, and $d_{\Gamma}$, we may assume that $l=m'=l'=k'=j'=0$. We can write the integral as
\[(a\# b) = \int_{\mathbb{R}^n}{e^{-i(\xi X+\eta\cdot Y)}\mathcal{F}_{(\xi',\eta')\to(X,Y)}^{-1}a(x,y,X,Y)b(x,y,X,Y,\xi,\eta)\,dX\,dY}.\]
We note that $\mathcal{F}^{-1}a(x,y,X,Y)$ is Schwartz in $(X,Y)$ away from the origin $(X,Y)=(0,0)$, and moreover it satisfies uniform Schwartz estimates in $(x,y)$, since $a$ is smooth uniformly in $(x,y)$. Near the origin, we have that $\mathcal{F}^{-1}a(x,y,X,Y)$ is in $C^{m_0}$, with uniformly bounded derivatives up to order $m_0$: indeed, note that
\[D_{X,Y}^{\alpha}\mathcal{F}^{-1}a(x,y,X,Y) = \mathcal{F}^{-1}\left((i(\xi,\eta))^{\alpha}a\right)(x,y,X,Y),\]
and for $|\alpha|\le m_0$ we have $|(i(\xi,\eta))^{\alpha}a|\le C\langle(\xi,\eta)\rangle^{|\alpha|+m}$ with $|\alpha|+m\le m_0+m<-n$, and hence the integral defining its inverse Fourier transform is uniformly integrable, i.e. $\mathcal{F}^{-1}\left((i(\xi,\eta))^{\alpha}a\right)$ is uniformly bounded. As such, we note that for $|\alpha|\le m_0$ we have
\begin{align*}
(\xi,\eta)^{\alpha}(a\# b) &= \int_{\mathbb{R}^n}{(-D_{X,Y})^{\alpha}\left(e^{-i(\xi X+\eta\cdot Y)}\right)\mathcal{F}_{(\xi',\eta')\to(X,Y)}^{-1}a(x,y,X,Y)b(x,y,X,Y,\xi,\eta)\,dX\,dY} \\
&= \int_{\mathbb{R}^n}{e^{-i(\xi X+\eta\cdot Y)}D_{X,Y}^{\alpha}\left(\mathcal{F}_{(\xi',\eta')\to(X,Y)}^{-1}a(x,y,X,Y)b(x,y,X,Y,\xi,\eta)\right)\,dX\,dY}
\end{align*}
Noting that $D_{X,Y}^{\alpha}\left(\mathcal{F}_{(\xi',\eta')\to(X,Y)}^{-1}a(x,y,X,Y)b(x,y,X,Y,\xi,\eta)\right)$ is uniformly bounded near $(X,Y)=(0,0)$ since $\mathcal{F}^{-1}a$ is in $C^{m_0}$ with $|\alpha|\le m_0$, and the term is also uniformly Schwartz away from $(X,Y)=(0,0)$, it follows that $(\xi,\eta)^{\alpha}(a\# b)$ is uniformly bounded for all $|\alpha|\le m_0$. This implies that $\langle(\xi,\eta)\rangle^{m_0}(a\# b)$ is uniformly bounded, i.e. $|(a\# b)|\le C\rho^{m_0}$, as desired.

To show the lemma in full generality, we use the asymptotic expansion \eqref{asympeq} to write, for any $N$,
\[a\# b = \sum_{|\alpha|<N}{\frac{i^{|\alpha|}}{|\alpha|!}\partial^{\alpha}_{\xi,\eta}a\partial^{\alpha}_{X,Y}b|_{X=0,Y=0}} + \sum_{i=1}^{N'}{a_i\# b_i},\]
where $a_i\in S^{m-N,l-N,k+N,j+N}$ and $b_i\in S^{m',l',k',j'}_{full}$. Each term in the first sum is in $S^{M-|\alpha|,L-|\alpha|,K+|\alpha|,J+|\alpha|}\subset S^{M-|\alpha|/2,L-|\alpha|/2,K,J}$ and hence bounded by $C\rho^{-M}x^{-L}d_{\Sigma}^{-K}d_{\Gamma}^{-J}$. For the terms $a_i\# b_i$ in the remainder term, we note that if $N$ is sufficiently large so that $k+N$ and $j+N$ are both positive, then we have
\[a_i\in S^{m-N,l-N,k+N,j+N}\subset S^{m+k/2-N/2,l+j/2-N/2,0,0}.\]
If we take $m_0$ sufficiently large so that $m_0+m\ge (k_-)/2$, and we take $N$ sufficiently large so that $m+k/2-N/2<-n-m_0$, then by the weaker estimate proven above, we have
\begin{align*}
\left|\sum_{i=1}^{N'}{a_i\# b_i}\right|&\le C\rho^{m_0-m'}x^{N/2-j/2-(l+l')}d_{\Sigma}^{-k'}d_{\Gamma}^{-j'} \\
&= C\rho^{m_0+m}x^{N/2-j/2}\left(\rho^{-(m+m')}x^{-(l+l')}d_{\Sigma}^{-k'}d_{\Gamma}^{-j'}\right).
\end{align*}
If we also require $N$ to be sufficiently large so that $N/2-j/2\ge (j_-)/2$, then we have
\[\rho^{m_0+m}x^{N/2-j/2}\le \rho^{(k_-)/2}x^{(j_-)/2}\le Cd_{\Sigma}^{-k}d_{\Gamma}^{-j}.\]
Thus, if $N$ is sufficiently large, we have that the remainder term is also bounded by $C\rho^{-M}x^{-L}d_{\Sigma}^{-K}d_{\Gamma}^{-J}$, thus showing the estimate in general.
\end{proof}

\begin{proof}[Proof of Lemma \ref{vf}]
Since the collection of vector fields tangent to $\partial\mathscr{X}_2$ are generated over $S^{0,0,0,0}$ by the vector fields
\[x\partial_x,\partial_y,\xi\partial_{\xi},\eta_i\partial_{\xi},\xi\partial_{\eta_j},\eta_i\partial_{\eta_j},\eta_j\partial_{\eta_1},\partial_{\xi},\partial_{\eta_j},x^{1/2}\rho^{-1/2}\partial_{\xi},x^{1/2}\rho^{-1/2}\partial_{\eta_i},\quad i=2,\dots n-1, j = 1,\dots,n-1,\]
and that furthermore $c(a\# b) = a\#(cb)$ when $c\in S^{0,0,0,0}$, it suffices to show the lemma for the above vector fields.

For the horizontal vector fields $V = x\partial_x,\partial_y$, from \eqref{compsymbol} we have
\[V(a\# b) = (Va)\# b + a\#(Vb).\]
Thus the lemma holds for these vector fields since symbols of a given order are preserved under $V$.

For the vertical vector fields, we rewrite \eqref{compsymbol} as
\begin{equation}
(a\# b)(x,y,\xi,\eta):=(2\pi)^{-n}\int_{\mathbb{R}^n\times\mathbb{R}^n}{e^{i[\xi'X+\eta'\cdot Y]}a(x,y,\xi'+\xi,\eta'+\eta)b(x,y,X,Y,\xi,\eta)\,d\xi'\,d\eta'\,dX\,dY}.
\end{equation}
to show
\[\partial_{\eta_j}(a\# b) = (\partial_{\eta_j}a)\# b + a\#(\partial_{\eta_j} b).\]
For all $j$ we have $\partial_{\eta_j}a\in S^{m-1,l,k+1,j+1}\subset S^{m,l,k,j}$ and $\partial_{\eta_j}b\in S^{m'-1,l,k'+1,j'+1}\subset S^{m',l',k',j'}_{full}$, so the lemma holds for $V = \partial_{\eta_j}$ (the same logic holds for $\partial_{\xi}$ as well). Furthermore, since
\[\eta_i(a\# b) = a\#(\eta_ib),\]
it follows that for $V = \eta_i\partial_{\eta_j}$ we can write
\[\eta_i\partial_{\eta_j}(a\# b) = (\partial_{\eta_j}a)\#(\eta_ib) + a\#(\eta_i\partial_{\eta_j}b) = (\partial_{\eta_j}a)\#(\eta_ib) + a\#(Vb).\]
We again have $\partial_{\eta_j}a\in S^{m-1,l,k+1,j+1}$, while
if $i=2,\dots,n-1$, then
\[\eta_i = \rho^{-1}\tau_i \in S^{1,0,-1,-1}.\]
Thus, $\eta_ib\in S^{m'+1,l',k'-1,j'-1}_{full}$, so the orders of $\partial_{\eta_j}a$ and $\eta_ib$ add to the desired quantities. Hence, the lemma holds for $V = \eta_i\partial_{\eta_j}$ with $i=2,\dots,n-1$. The same argument holds for $V = \xi\partial_{\xi}$,$\eta_i\partial_{\xi}$, and $\xi\partial_{\eta_j}$. For $V = \eta_j\partial_{\eta_1}$, we have
\[V(a\# b) = (\partial_{\eta_1}a)\#(\eta_jb) + a\#(Vb).\]
Noting that $\partial_{\eta_1}a\in S^{m-1,l,k,j}$ since $\partial_{\eta_1}$ is tangent to $\{\tau = 0\}$, and $\eta_jb\in S^{m'+1,l',k',j'}$, we see that the lemma holds here as well. Finally, for $V = x^{1/2}\rho^{-1/2}\partial_{\eta_i}$, we have
\[V(a\# b) = (\partial_{\eta_i}a)\#(x^{1/2}\rho^{-1/2}b) + a\# Vb.\]
In the first term, we have $\partial_{\eta_i}a\in S^{m-1,l,k+1,j+1}\subset S^{m-1/2,l+1/2,k,j}$ and $x^{1/2}\rho^{-1/2}b\in S^{m'+1/2,l'-1/2,k',j'}$, so the sum of the orders add to the desired quantities, while in the second term we have that $Vb$ has the same orders as $b$. Hence the lemma holds here as well; the logic applies similarly to $V=x^{1/2}\rho^{-1/2}\partial_{\xi}$.
\end{proof}

From Theorem \ref{comp-symb}, we obtain the following ellipticity result, following standard arguments in microlocal analysis:
\begin{prop}[Elliptic Parametrix]
\label{ell-param}
Suppose $a\in S^{m,l,k,j}$ satisfies an ellipticity estimate
\[|a|\ge c\rho^{-m}x^{-l}d_{\Sigma}^{-k}d_{\Gamma}^{-j}\quad\text{near }\partial\mathscr{X}_2.\]
If $A$ is the corresponding quantization of $a$, then $A$ admits an elliptic parametrix, i.e. there exists $B\in\Psi^{-m,-l,-k,-j}$ such that
\[BA = \text{Id} + R,\quad R\in\Psi^{-\infty,-\infty}.\]
\end{prop}
Similarly, using Theorem \ref{comp-symb} and Proposition \ref{adj-prop}, we obtain the $L^2$-boundedness property, again following standard arguments in microlocal analysis:
\begin{prop}[$L^2$-boundedness]
\label{l2-bound}
An operator in $\Psi^{0,0,0,0}$ is bounded as a map from $L^2\left([0,\infty)\times\mathbb{R}^{n-1},\frac{dx\,dy}{x^{(n+1)}}\right)$ to itself.
\end{prop}
Recalling that the scattering Sobolev spaces $H^{s,r}_{sc}$ are defined by the property that scattering operators in $\Psi^{s,r}$ map it into $L^2$, we obtain the following corollary of the $L^2$-boundedness of \newname-operators:
\begin{cor}
An operator in $\Psi^{m,l,0,0}$ is bounded as a map from $H^{s,r}_{sc}$ to $H^{s-m,r-l}_{sc}$ for any $s,r\in\mathbb{R}$.
\end{cor}
\begin{remark}
In general, we can conclude mapping properties of operators in $\Psi^{m,l,k,j}$ (for $k$, $j$ not necessarily zero) by including it into a space of the form $\Psi^{m',l',0,0}$: for example, we have
\[\Psi^{m,l,k,j}\subset \Psi^{m+k_+/2,l+j_+/2,0,0},\]
so operators in $\Psi^{m,l,k,j}$ map $H^{s,r}_{sc}$ into $H^{s-(m+k_+/2),r-(l+j_+/2)}_{sc}$ for any $s,r\in\mathbb{R}$.
\end{remark}

\section{Ellipticity of Operators from Elastic Inverse Problem}
\label{op-analysis}

We now investigate the composition of the pseudolinearization operator $I$ with a ``localized'' formal adjoint of the form \eqref{scat-formal-adj}. Recall in this case that we are working near some point $p$ on the boundary, with $\tilde{x}$ a function strictly convex with respect to all dynamics considered satisfying $\tilde{x}(p)=0$ and $d\tilde{x}(p)\ne 0$, and that we have introduced an artificial boundary $\{x=0\}$, where $x=\tilde{x}+c$, with $c>0$ a small number of our choice. We show that the composition of $I$ with such localized formal adjoints are, after conjugation with a factor of $e^{\digamma/x}$, scattering pseudodifferential operators, which normally are not elliptic as scattering operators, but are elliptic in a new operator class to be defined. 

\subsection{The localized operators}
\label{local-ops-sec}

We recall that our operators of interest are the pseudolinearization operators
\[I[u](x_0,\xi_0) = \int_{\mathbb{R}}{A(X(t,x_0,\xi_0),\Xi(t,x_0,\xi_0))u(X(t,x_0,\xi_0))\,dt},\]
with $A(x,\xi) = -E^{\nu}(x,\xi)\frac{\partial\tilde\Xi}{\partial\xi}(x,\xi)$, and a formal adjoint of the form 
\[Lv(x,y) = x^{-2}\int_{\mathbb{S}^2}{\chi(x^{\epsilon}\hat\lambda)\exp\left(-\frac{\digamma^2\hat\lambda^2}{2\alpha}\right)v(\gamma_{\lambda,\omega})\,d\mathbb{S}^2(\lambda,\omega)}\]
(recalling that $x$ is as above and $y=(y_1,y_2)$ are tangential coordinates such that $(x,y)$ form valid coordinates near $p$). Here $0<\epsilon<1/2$, $\digamma>0$, $\chi\in C_c^{\infty}(\mathbb{R})$ is identically $1$ in a neighborhood of $0$, $\hat\lambda = \lambda/x$, $\gamma_{\lambda,\omega}$ is the trajectory passing through $(x,y)$ with tangent vector $\lambda\partial_x+\omega\cdot\partial_y$, and $\alpha=\alpha(x,y,\lambda,\omega)$ is an ``acceleration'' factor in the $x$ direction satisfying
\[\gamma_{\lambda,\omega}-x = \lambda t + \alpha t^2 + O(t^3) \quad(\text{equivalently }\frac{\gamma_{\lambda,\omega}-x}{x^2} = \hat\lambda\hat{t}+\alpha\hat{t}^2+O(x\hat{t}^3)\text{ if }\hat\lambda = \frac{\lambda}{x}, \hat{t} = \frac{t}{x}).\]
Note that this differs slightly in structure from previous papers, as the cutoff chosen is not a compact cutoff in $\hat\lambda$, but rather a Gaussian. Such a cutoff works well with finite point symbol computations, but in previous works one uses a compactly supported approximation of a Gaussian using the Gaussian computation as justification; here we will directly use the Gaussian. However, we will still select for $\hat\lambda$ in a controlled manner, namely that $\hat\lambda$ is at most $O(x^{-\epsilon})$ as $x\to 0$, i.e. $\lambda$ is $O(x^{1-\epsilon})$; as we will see below, this is good enough for dynamics purposes.

We remark that for the pseudolinearization operator $I$, we are free to adjust the factor $A$ so that it is supported in $\{x<2c,|t|<T\}$, where $T$ is larger than the maximum travel time of a trajectory in $\overline\Omega\cap\{x\ge 0\}$, since the function of interest, namely the difference of parameters $r_{\nu}$, is supported in $\overline\Omega$. Since $x = \tilde{x}+c$, by making $c$ arbitrarily small we can arrange for $T$ to be sufficiently small as well; this will turn out to be useful for dynamical assumptions below.

For $N=L\circ I$, we then have
\begin{align*}
N[u](x,y) &= x^{-2}\int_{\mathbb{R}\times\mathbb{S}^2}{\chi(x^{\epsilon}\hat\lambda)\exp\left(-\frac{\digamma^2\hat\lambda^2}{2\alpha}\right)A((X,\Xi)(x,y,t,\lambda,\omega))u(X(x,y,t,\lambda,\omega))\,dt\,d\mathbb{S}^2(\lambda,\omega)}.
\end{align*}
Here, we let $(X,\Xi)(x,y,t,\lambda,\omega)$ denote the Hamilton flow at time $t$ (viewed in the cotangent bundle) starting at $(x,y)$ with initial  (spatial) velocity $\lambda\partial_x+\omega\cdot\partial_y$ (equivalently with initial momentum corresponding to an initial velocity of $\lambda\partial_x+\omega\cdot\partial_y$). We write $A(x,y,t,\lambda,\omega) = A((X,\Xi)(x,y,t,\lambda,\omega))$ and $(x',y') = X(x,y,t,\lambda,\omega)$; this now frees up the letter $X$ to be used as a scattering coordinate. We can thus write its Schwartz kernel as
\begin{align*}
&K(x,y,X,Y) \\
&= x^{-2}\int_{\mathbb{R}\times\mathbb{S}^2}{\chi(x^{\epsilon}\hat\lambda)\exp\left(-\frac{\digamma^2\hat\lambda^2}{2\alpha}\right)A(x,y,t,\lambda,\omega)\delta\left(X-\frac{x'-x}{x^2},Y-\frac{y'-y}{x}\right)\,dt\,d\mathbb{S}^2(\lambda,\omega)} \\
&= \int_{\mathbb{R}_{\hat{t}}\times\mathbb{R}_{\hat\lambda}\times\mathbb{S}^1_{\omega}}{\chi(x^{\epsilon}\hat\lambda)\exp\left(-\frac{\digamma^2\hat\lambda^2}{2\alpha}\right)A(x,y,x\hat{t},x\hat\lambda,\omega)\delta\left(X-\frac{x'-x}{x^2},Y-\frac{y'-y}{x}\right)\,d\hat{t}\,d\hat\lambda\,d\mathbb{S}^1(\omega)}.
\end{align*}
Note that we can write
\[d\mathbb{S}^2(\lambda,\omega) = d\lambda\,d\mathbb{S}^1(\omega)\]
since we are in dimension $3$. It turns out that the correct operator to consider is not $N$ itself, but rather its conjugate $N_{\digamma}:= e^{-\digamma/x}Ne^{\digamma/x}$; its Schwartz kernel is just the above Schwartz kernel times $e^{-\digamma/x}e^{\digamma/x'}$. Hence
\begin{align*}
K_{\digamma}(x,y,X,Y) &=\int_{\mathbb{R}_{\hat{t}}\times\mathbb{R}_{\hat\lambda}\times\mathbb{S}^1_{\omega}}\chi(x^{\epsilon}\hat\lambda)\exp\left(-\digamma\left(\frac{\hat\lambda^2}{2\alpha}-\left(\frac{1}{x'}-\frac{1}{x}\right)\right)\right)A(x,y,x\hat{t},x\hat\lambda,\omega)\\
&\cdot\delta\left(X-\frac{x'-x}{x^2},Y-\frac{y'-y}{x}\right)\,d\hat{t}\,d\hat\lambda\,d\mathbb{S}^1(\omega).
\end{align*}
We can find the scattering symbol of $N$ by taking the $(X,Y)$-Fourier transform of its $(X,Y)$-Schwartz kernel. Thus the symbol has the form
\begin{equation}\label{int0}
a = \int_{\mathbb{R}_{\hat{t}}\times\mathbb{R}_{\hat\lambda}\times\mathbb{S}^1_{\omega}}{e^{i\left(\xi\frac{x'-x}{x^2}+\eta\cdot\frac{y'-y}{x}\right)}\chi(x^{\epsilon}\hat\lambda)\exp\left(-\digamma\left(\frac{\hat\lambda^2}{2\alpha}-\left(\frac{1}{x'}-\frac{1}{x}\right)\right)\right)A(x,y,\hat{\lambda},\omega,\hat{t})\,d\hat{t}\,d\hat\lambda\,d\omega}.
\end{equation}
We now comment on the Hamiltonian dynamics near $p$, which determine what assumptions we may put on the function $A$ above.
Since the function $\tilde{x}$ (and hence $x=\tilde{x}+c$) is chosen to be convex with respect to the Hamiltonian dynamics involved, we can make an estimate on the acceleration component $\alpha\ge C_0$ for some constant $C_0>0$. In fact, since we have $x = \tilde{x}+c$ and are also free to adjust $c$ to consider all sufficiently small $c$, with the region of interest (i.e. $\overline{\Omega_c}=\{\tilde{x}>-c\}\cap\overline{\Omega}$) shrinking to a point as $c\to 0^+$, we may choose $c$ sufficiently small so that $\alpha$ also satisfies an upper bound in a neighborhood of $\overline{\Omega_c}$, say $C_0\le \alpha\le 1.01C_0$; note that by Taylor's theorem this would also imply that 
\[C_0x^2\hat{t}^2\le \frac{x'-x}{x^2}-\hat\lambda\hat{t}\le 1.01C_0x^2\hat{t}^2\]
in this region, where $x'$ is the $x$ component of $\gamma_{\lambda,\omega}(t)$. Note that the lower bound implies the existence of $C'$ such that
\[|\lambda|\le\frac{C}{x^{1/2}}\implies x'\ge 0\text{ for all }t\]
(see \cite{convex} for a justification), and hence the cutoff $\chi(x^{\epsilon}\hat\lambda)$ guarantees that all trajectories with parameters supported in the support of this cutoff stay in the region $\{x>0\}$ for all $t$. Moreover, the function $A$ is smooth, grows at most polynomially as $(\hat{t},\hat{\lambda})\to\infty$, and by the remarks regarding the pseudolinearization operator $I$, we can arrange for $A$ to be supported in a set of the form $\{x<2c,x|\hat{t}|<T\}$ where $c,T>0$ can be arranged to be arbitrarily small.

\subsection{Scattering symbol}
We first show such integrals are indeed scattering symbols. If the Gaussian $\exp\left(-\frac{\digamma^2\hat\lambda^2}{2\alpha}\right)$ were replaced by a $C_c^{\infty}$ function approximating it, then this would be a result of the works in \cite{convex} etc.; however we can apply similar methods even after the Gaussian change. We first show, for each $(x,y)$, that the integral admits an asymptotic expansion in $|(\xi,\eta)|$ as $|(\xi,\eta)|\to\infty$. We then show that applying $b$-vector fields like $x\partial_x$ and $\partial_y$ essentially returns the integral to the same class, for which we also have a similar asymptotic expansion, thus showing the symbolicity of the integral.

The idea is that the phase should have stationary points along precisely the submanifold $\{\hat{t} = 0, \xi\hat{\lambda}+\eta\cdot\omega = 0\}$, with an asymptotic expansion following stationary phase on that submanifold, and the integral away from the stationary set should be suitably controlled.

To rigorously justify this, we first show that the integral, while over a non-compact region in $\hat\lambda$ and $\hat{t}$, is well-controlled due to the (real) exponential factor, which will end up decaying at least exponentially. Specifically, we have the following:

\begin{lemma}
On the support of the integrand (i.e. for $|\hat\lambda|\le Cx^{-\epsilon}$ and $|\hat{t}|\le T/x$), there exist constants $c_1,c_2>0$ such that
\[\frac{\hat\lambda^2}{2\alpha}+\left(\frac{1}{x}-\frac{1}{x'}\right)\ge\begin{cases}c_1(\hat{\lambda}^2+\hat{t}^2) & \text{for }|\hat{t}|\le\frac{c_2}{x^{1/2}}, \\ c_1(\hat{\lambda}^2+|\hat{t}|) &\text{for } |\hat{t}|\ge\frac{c_2}{x^{1/2}} \end{cases}.\]
\end{lemma}

\begin{proof}
We take $c_2 = \frac{1}{5\sqrt{C_0}}$. In the region of support we have $|\hat\lambda|\le Cx^{-\epsilon} = Cx^{1/2-\epsilon}x^{-1/2}$, so by choosing $c$ sufficiently small, we can arrange $|\hat\lambda|\le (100c_2)^{-1}x^{-1/2}$ on the region of support. Then in the region $|\hat{t}|\le c_2/x^{1/2}$ we have
\[x'-x\ge x^2(C_0\hat{t}^2-|\hat\lambda\hat{t}|) \ge x^2(-0.01c_2^{-1}x^{-1/2}\cdot c_2x^{-1/2}) = -0.01x\]
i.e. $x'\ge 0.99x$, while similarly
\[x'-x\le x^2(|\hat\lambda\hat{t}|+1.01C_0\hat{t}^2)\le x^2\left(0.01x^{-1} + 1.01C_0\cdot\frac{1}{25C_0}x^{-1}\right)\le 0.06x\]
so $x'\le 1.06x$. Note that when $x'\le x$, we have $\frac{1}{x}-\frac{1}{x'} = \frac{x'-x}{xx'}\ge\frac{x'-x}{0.99x^2}$ (since $x'-x\le 0$ there), while in $x'\ge x$ we have $\frac{1}{x}-\frac{1}{x'}\ge \frac{x'-x}{1.06x^2}$. Thus, we have
\[\frac{\hat\lambda^2}{2\alpha}+\left(\frac{1}{x}-\frac{1}{x'}\right)\ge\left\{\begin{matrix}0.99\left(\frac{\hat\lambda^2}{1.98\alpha} + \frac{x'-x}{x^2}\right) & x'\le x \\ 1.06\left(\frac{\hat\lambda^2}{2.12\alpha}+\frac{x'-x}{x^2}\right) & x'\ge x \end{matrix}\right.\]
It thus suffices to obtain a lower bound on $\frac{\hat\lambda^2}{2.12\alpha}+\frac{x'-x}{x^2}$ in this region. Note that we can bound
\begin{align*} \frac{\hat\lambda^2}{2.12\alpha}+\frac{x'-x}{x^2} \ge \frac{\hat\lambda^2}{2.12\alpha} + C_0\hat{t}^2 - |\hat\lambda\hat{t}| &\ge\frac{\hat\lambda^2}{2.12\alpha} + C_0\hat{t}^2 - \frac{1}{2}\left(\frac{1}{2(C_0-\epsilon)}\hat\lambda^2 + 2(C_0-\epsilon)\hat{t}^2 \right) \\
&\ge \frac{1}{2}\left(\frac{1}{1.06\alpha}-\frac{1}{2(C_0-\epsilon)}\right)\hat\lambda^2+\epsilon\hat{t}^2.
\end{align*}
For sufficiently small $\epsilon$, we have $2(C_0-\epsilon)>1.06\alpha$ (since by assumption we will always have $\alpha\le 1.01C_0$), which then gives the desired lower bound.

On the other region $|\hat{t}|\ge c_2x^{-1/2}$, we have
\[x'-x\ge x^2(C_0\hat{t}^2-|\hat\lambda\hat{t}|) = x^2|\hat{t}|(C_0|\hat{t}|-|\hat\lambda|).\]
Note that $|\hat\lambda|\le (100c_2)^{-1}x^{-1/2} = \frac{1}{20}\sqrt{C_0}x^{-1/2}=\frac{C_0c_2}{4}x^{-1/2}\le \frac{C_0}{4}|\hat{t}|$, so we have
\[x^2|\hat{t}|(C_0|\hat{t}|-|\hat\lambda|)\ge x^2|\hat{t}|\frac{3}{4}C_0|\hat{t}|\ge \frac{3}{4}x^2C_0(c_2^2x^{-1}) = 0.03x, \]
i.e. $x'\ge 1.03x$ in the region. Thus $\frac{1}{x}-\frac{1}{x'}\ge \left(1-\frac{1}{1.03}\right)\frac{1}{x}$. Since on the region of support we have $x|\hat{t}|\le T$, we have $\frac{1}{x}\ge|\hat{t}|/T$. Thus, we have
\[\frac{\hat\lambda^2}{2\alpha} + \frac{1}{x}-\frac{1}{x'}\ge \frac{1}{2\alpha}\hat\lambda^2 + \left(1-\frac{1}{1.03}\right)\frac{1}{x} \ge \frac{1}{2\alpha}\hat\lambda^2 + \frac{1-\frac{1}{1.03}}{T}|\hat{t}|,\]
as desired. 
\end{proof}

Thus, the exponential term $\exp\left(-\digamma\left(\frac{\hat\lambda^2}{2\alpha}-\left(\frac{1}{x'}-\frac{1}{x}\right)\right)\right)$ will decay at least exponentially as $(\hat{\lambda},\hat{t})\to\infty$, so any polynomial-order growth in the other terms will not affect the integrability of the integral. Furthermore, derivatives of the exponent will also be at most polynomial growth in $(\hat\lambda,\hat{t})$, so derivatives of the exponential term will also not affect integrability.

We now cut off outside a compact neighborhood of $\hat{t}$. To do so, write $\gamma = \left(\frac{x'-x}{x^2},\frac{y'-y}{x}\right)$, i.e.
\[\gamma = (\hat\lambda\hat{t}+\alpha\hat{t}^2+O(x\hat{t}^3),\omega\hat{t}+O(x\hat{t}^2)).\]
Also write $(\xi,\eta) = |(\xi,\eta)|(\tilde\xi,\tilde\eta)$ so $\tilde\xi^2+|\tilde\eta|^2 = 1$. Then the oscillatory exponent becomes $|(\xi,\eta)|\phi$ where $\phi = (\tilde\xi,\tilde\eta)\cdot\gamma$. 

We now calculate derivatives of $\gamma$. Let $\partial_{\theta}$ be the angular derivative on $\mathbb{S}^1$, so that $\partial_{\theta}\omega = \omega^{\perp} := (-\omega_2,\omega_1)$. We have
\begin{align*}
\partial_{\hat{t}}\gamma &= (\hat{\lambda} + 2\alpha\hat{t} + O(x\hat{t}^2),\omega + O(x\hat{t}))\\
\partial_{\hat{\lambda}}\gamma &= (\hat{t} + \partial_{\hat\lambda}\alpha\hat{t}^2 + O(x\hat{t}^3), O(x\hat{t}^2)) = \hat{t}\left(1+\partial_{\hat\lambda}\alpha\hat{t} + O(x\hat{t}^2), O(x\hat{t})\right)\\
\partial_{\theta}\gamma &= (\partial_{\theta}\alpha\hat{t}^2+O(x\hat{t}^3),\omega^{\perp}\hat{t} + O(x\hat{t}^2)) = \hat{t}\left(\partial_{\theta}\alpha\hat{t}+O(x\hat{t}^2),\omega^{\perp}+O(x\hat{t})\right).
\end{align*}
We can actually be a bit more precise with the $\hat{\lambda}$ derivative: note that we can write
\[x'-x = \lambda t + \alpha(x,y,\lambda,\omega)t^2 + \Gamma(x,y,\lambda,\omega,t)t^3\]
for some smooth $\Gamma$, and thus
\[\frac{x'-x}{x^2} = \hat\lambda\hat{t} + \alpha(x,y,x\hat\lambda,\omega)\hat{t}^2+\Gamma(x,y,x\hat{\lambda},\omega,x\hat{t})\hat{t}^3.\]
Thus we can actually write
\[\partial_{\hat\lambda}\left(\frac{x'-x}{x^2}\right) = \hat{t}\left(1 + x\partial_{\lambda}\alpha\hat{t} + \partial_{\lambda}\Gamma x^2\hat{t}^2\right).\]
Recall we always have $x\hat{t} = O(T)$ where we can arrange $T$ to be arbitrarily small. It follows that if we write the vectors $(\frac{\partial_{\hat{\lambda}}\gamma}{\hat{t}},\frac{\partial_{\theta}\gamma}{\hat{t}},\partial_{\hat{t}}\gamma)$ as a matrix, we have
\[\begin{pmatrix} \frac{\partial_{\hat{\lambda}}\gamma}{\hat{t}} & \frac{\partial_{\theta}\gamma}{\hat{t}} & \partial_{\hat{t}}\gamma\end{pmatrix} = \begin{pmatrix} 1 + O(T) & (\partial_{\theta}\alpha + O(T))\hat{t} & \hat{\lambda} + (2\alpha + O(T))\hat{t}^2 \\ \vec{0}_{2\times 1} & \omega & \omega^{\perp} \end{pmatrix} + O(T).\]
Note that this matrix has determinant $1+O(T)$ and is thus invertible, with an inverse whose entries grow at most polynomially along with all derivatives. It follows that
\begin{align*}\begin{pmatrix} \partial_{\hat{\lambda}}\gamma & \partial_{\theta}\gamma & \partial_{\hat{t}}\gamma\end{pmatrix} &= \begin{pmatrix} \frac{\partial_{\hat{\lambda}}\gamma}{\hat{t}} & \frac{\partial_{\theta}\gamma}{\hat{t}} & \partial_{\hat{t}}\gamma\end{pmatrix}\begin{pmatrix} \hat{t} & 0 & 0 \\ 0 & \hat{t} & 0 \\ 0 & 0 & 1 \end{pmatrix}\\
\implies \begin{pmatrix} \partial_{\hat{\lambda}}\gamma & \partial_{\theta}\gamma & \partial_{\hat{t}}\gamma\end{pmatrix}^{-1} &= \begin{pmatrix} 1/\hat{t} & 0 & 0 \\ 0 & 1/\hat{t} & 0 \\ 0 & 0 & 1 \end{pmatrix}\begin{pmatrix} \frac{\partial_{\hat{\lambda}}\gamma}{\hat{t}} & \frac{\partial_{\theta}\gamma}{\hat{t}} & \partial_{\hat{t}}\gamma\end{pmatrix}^{-1},
\end{align*}
i.e. $\begin{pmatrix} \partial_{\hat{\lambda}}\gamma & \partial_{\theta}\gamma & \partial_{\hat{t}}\gamma\end{pmatrix}$ is invertible away from $\hat{t} = 0$, with inverse whose entries are $\frac{1}{\hat{t}}$ times terms growing at most polynomially along with all derivatives. All of this leads to the following:

\begin{lemma}
For each unit vector $(\tilde\xi,\tilde\eta)\in\mathbb{S}^2$ there exists a smooth vector field $V$ on $\mathbb{R}_{\hat{t}}\times\mathbb{R}_{\hat\lambda}\times\mathbb{S}^1_{\omega}$ such that $V\gamma = \hat{t}(\tilde\xi,\tilde\eta)$. The coefficients of $V$ are bounded polynomially and have derivatives also bounded polynomially, with the bounds uniform in $(\tilde\xi,\tilde\eta)$. Moreover, the coefficient of $\partial_{\hat{t}}$ in $V$ will vanish at $\hat{t} = 0$. In particular, on any region of the form $\{|\hat{t}|>\epsilon\}$, there exists a smooth vector field $\tilde{V}$ such that $\tilde{V}\gamma = (\tilde\xi,\tilde\eta)$, with the coefficients of $\tilde{V}$ satisfying the same properties as above on $\{|\hat{t}|>\epsilon\}$.
\end{lemma}

Thus, we can split $a = a_1+a_2$, where $a_1$ is the integral in \eqref{int0} with a cutoff $\chi_1(\hat{t})$ inserted where $\chi_1\in C_c^{\infty}(\mathbb{R})$ is identically one in a neighborhood of $0$. We then have $a_2$ is decreasing to order $|(\xi,\eta)|^{-\infty}$, since we can argue as follows: note that for a fixed $(\tilde\xi,\tilde\eta)$, with $V$ chosen accordingly by the lemma above, we have
\[Ve^{i|(\xi,\eta)|(\tilde\xi,\tilde\eta)\gamma} = |(\xi,\eta)|e^{i|(\xi,\eta)|(\tilde\xi,\tilde\eta)\gamma}.\]
Thus repeated integration by parts w.r.t. $|(\xi,\eta)|^{-1}V$ gives the desired decay in $|(\xi,\eta)|$, noting that since the integrand in $a_2$ will be supported uniformly away from $\{\hat{t} = 0\}$, the coefficients of $V$ will be bounded polynomially (along with its derivatives), and any additional polynomial growth will be harmless against the exponential decay in the integral.

We can assume the cutoff $\chi_1$ is supported say in $[-1/(100C_0),1/(100C_0)]$, and thus we focus on 
\begin{equation}\label{int1}
a_1 = \int_{[-1/(100C_0),1/(100C_0)]_{\hat{t}}\times\mathbb{R}_{\hat\lambda}\times\mathbb{S}^1_{\omega}}{e^{i|(\xi,\eta)|(\tilde\xi,\tilde\eta)\cdot\gamma}\chi_1(\hat{t})A'(x,y,\hat\lambda,\omega,\hat{t})\,d\hat{t}\,d\hat\lambda\,d\omega}
\end{equation}
where $A'(x,y,\hat\lambda,\omega,\hat{t}) = \chi(x^{\epsilon}\hat\lambda)\exp\left(-\digamma\left(\frac{\hat\lambda^2}{2\alpha}-\left(\frac{1}{x'}-\frac{1}{x}\right)\right)\right)A(x,y,\hat{\lambda},\omega,\hat{t})$. The stationary points will end up being at $\{\hat{t} = 0, \tilde\xi\hat\lambda+\tilde\eta\cdot\omega=0\}$. Since the support of $\hat\lambda$ is not uniformly compact (in particular the support increases to $\mathbb{R}$ as $x\to 0$), we will divide the domain into two further regions.

We first suppose $|\tilde\xi|$ is small. Let $\omega_{\parallel} = \frac{\tilde\eta}{|\tilde\eta|}\cdot\omega$. We now divide $\mathbb{S}^1_{\omega}$ via smooth cutoffs, say 
\[1 = \chi_{1,1}(\omega_{\parallel})+\chi_{1,2}(\omega_{\parallel}),\]
where $\chi_{1,1}$ is supported where $|\omega_{\parallel}|$ is small (say at most $2/3$), and $\chi_{1,2}$ is supported where $|\omega_{\parallel}|$ is large (say at least $1/2$). Let $a_1 = a_{1,1}+a_{1,2}$ be the corresponding partition of the integral in \eqref{int1}. In the former region we write $\omega = \omega_{\parallel}\frac{\tilde\eta}{|\tilde\eta|} + \sqrt{1-\omega_{\parallel}^2}\omega_{\perp}$, where $\omega_{\perp}$ is one of two vectors in $\mathbb{S}^1$ perpendicular to $\tilde\eta$, so the integral becomes
\[a_{1,1} = \int_{\mathbb{R}\times\mathbb{S}^0}\int_{-\frac{1}{100C_0}}^{\frac{1}{100C_0}}\int_{-2/3}^{2/3}{e^{i(\xi,\eta)\cdot\gamma}\chi_{1,1}(\omega_{\parallel})\chi_1(\hat{t})A'(x,y,\hat\lambda,\omega,\hat{t})(1-\omega_{\parallel}^2)^{-1/2}\,d\omega_{\parallel}\,d\hat{t}\,d\hat\lambda\,d\omega_{\perp}}.\]
We will use $\hat{t}$ and $\omega_{\parallel}$ as the stationary phase variables, and $\hat\lambda$ and $\omega_{\perp}$ as parameters. We also write the phase as
\[|\eta|\left(\frac{\xi}{|\eta|},\frac{\eta}{|\eta|}\right)\cdot\gamma = |\eta|\left(\frac{\xi}{|\eta|}(\hat\lambda\hat{t} + \alpha\hat{t}^2 + O(x\hat{t}^3)) + \omega_{\parallel}\hat{t} + O(x\hat{t}^2)\right)\]
We split the integral in $\hat\lambda$ into two parts: one where $|\frac{\xi}{|\eta|}\hat\lambda|<1$ and one where $|\frac{\xi}{|\eta|}\hat\lambda|\ge 3/4$. In the former part where $|\frac{\xi}{|\eta|}\hat\lambda|<1$, we can proceed with stationary phase as usual: the stationary set is
\[\left\{\hat{t} = 0, \frac{\xi}{|\eta|}\hat\lambda + \omega_{\parallel} = 0\right\}\]
and the Hessian at the stationary set (w.r.t. $\hat{t},\omega_{\parallel}$) is
\[\begin{pmatrix} \frac{\xi}{|\eta|}\alpha & 1 \\ 1 & 0 \end{pmatrix}.\]
Thus stationary phase gives that the integral satisfies the asymptotic expansion\footnote{Note that the lower order terms are not exactly as described here since the phase is not exactly a quadratic form, but the fact that an asymptotic expansion exists remains true.}
\begin{align*}
\label{scat-symb-asymp}
|\eta|^{-1}\sum_{j=0}^{\infty}\frac{i^j}{j!|\eta|^j} \int_{\mathbb{R}\times\mathbb{S}^0}&\left(\partial_{\hat{t}}\partial_{\omega_{\parallel}} - \frac{\xi}{|\eta|}\partial_{\omega_{\parallel}}^2\right)^j\\
&\left[\chi_{1,1}(\omega_{\parallel})\chi_1(\hat{t})A'(x,y,\hat\lambda,\omega,\hat{t})(1-\omega_{\parallel}^2)^{-1/2}\right]|_{\hat{t} = 0, \omega_{\parallel} = -\frac{\xi}{|\eta|}\hat\lambda}\,d\hat{\lambda}\,d\omega_{\perp}.
\end{align*}
In the latter part where $|\frac{\xi}{|\eta|}\hat\lambda|\ge 3/4$, we note that due to support properties we can arrange for the phase to be non-stationary: note that
\[\partial_{\hat{t}}\left(\left(\frac{\xi}{|\eta|},\frac{\eta}{|\eta|}\right)\cdot\gamma\right) = \frac{\xi}{|\eta|}\hat\lambda + \frac{\xi}{|\eta|}(2\alpha+O(x\hat{t}))\hat{t} + \omega_{\parallel} + O(x\hat{t}),\]
so due to all the support properties (i.e. $\frac{\xi}{|\eta|}\hat\lambda\ge 3/4$, $|\omega_{\parallel}|<2/3$, $\hat{t}$ being sufficiently small (this is where the $1/(100C_0)$ comes in), etc.), we have that the above derivative is bounded below, so we can integrate by parts without problem.

For the second integral $a_{1,2}$, we will not reparametrize $\omega$ by $\omega_{\parallel}$, as it will be a singular parametrization in the region of support. Instead, we use $\omega$ as a parameter and use $\hat{t}$, $\hat\lambda$ as the stationary phase variables. We split up the integral similarly as before using a cutoff in $\tilde\xi\hat\lambda$:
\[1 = \tilde\chi_1\left(\tilde\xi\hat\lambda\right)+\tilde\chi_2\left(\tilde\xi\hat\lambda\right),\]
with the support of $\tilde\chi_1$ contained where $|\tilde\xi\hat\lambda|$ is at least $1/3$ and the support of $\tilde\chi_2$ contained where $|\tilde\xi\hat\lambda|$ is strictly less than $1/2$. On the latter region one can integrate by parts as before (noting that now we have $|\omega_{\parallel}|$ is at least $1/2$ on the region of support). On the former region, we have
\[a_{1,2}' = \int_{\mathbb{S}^1}{\chi_{1,2}(\omega_{\parallel})\left(\int\int{e^{i|(\xi,\eta)|(\tilde\xi,\tilde\eta)\cdot\gamma}\tilde\chi_1\left(\tilde\xi\hat\lambda\right)\chi_1(\hat{t})A'(x,y,\hat\lambda,\omega,\hat{t})\,d\hat{t}\,d\hat{\lambda}}\right)\,d\omega}.\]
When we try to perform stationary phase w.r.t. $(\hat{t},\hat{\lambda})$, we obtain that the stationary set is still $\hat{t} = 0, \tilde\xi\hat{\lambda}+\tilde\eta\cdot\omega = 0$, with the Hessian now
\[\begin{pmatrix} 2\tilde\xi\alpha & \tilde\xi \\ \tilde\xi & 0 \end{pmatrix}.\]
One complication with stationary phase is the following: the determinant of the Hessian is now $-\tilde\xi^2$, which may be small; however for this integral to be nonzero at all requires $\tilde\xi$ to be nonzero! Thus the stationary phase expansion would give
\[a_{1,2}'\sim\sum_{j\ge 0}{\frac{1}{|(\xi,\eta)|^{j+1}}\int_{\mathbb{S}^1}{\frac{\chi_{1,2}(\omega_{\parallel})}{\tilde\xi^2}\left(\frac{\partial_{\hat{t}}\partial_{\hat\lambda}-\alpha\partial_{\hat\lambda}^2}{\tilde\xi}\right)^j\left(\tilde\chi_1\left(\tilde\xi\hat\lambda\right)\chi_1(\hat{t})A'(x,y,\hat\lambda,\omega,\hat{t})\right)\Big|_{\hat{t} = 0, \hat\lambda = (-\tilde\eta\cdot\omega)/\tilde\xi}\,d\omega}}.\]
Note that any derivatives on the integrand $\tilde\chi_1\left(\tilde\xi\hat\lambda\right)\chi_1(\hat{t})A'(x,y,\hat\lambda,\omega,\hat{t})$ equals the exponential factor $\exp\left(-\digamma\left(\frac{\hat\lambda^2}{2\alpha}-\left(\frac{1}{x'}-\frac{1}{x}\right)\right)\right)$ times terms which are smooth in $(x,y,\hat\lambda,\omega,\hat{t})$ (aside from possible $x^{\epsilon}$ terms from the $\chi(x^{\epsilon}\hat{\lambda})$, which are nonetheless conormal to $x=0$). The exponential factor evaluated at $\hat{t} = 0, \hat\lambda = (-\tilde\eta\cdot\omega)/\tilde\xi$ gives the term $\exp\left(-\digamma\frac{\omega_{\parallel}^2}{2\alpha(\xi/|\eta|)^2}\right)$, i.e. a term which is bounded by a Gaussian in $1/\tilde\xi$. It follows that the powers of $1/\tilde\xi$ which arise from the degeneracy of the phase are actually completely controlled by the exponential decay, and thus this term, while not necessarily zero away from $\tilde\xi = 0$, will vanish extremely rapidly (in particular smoothly) as one approaches $\tilde\xi = 0$.

Thus, in the region where $\tilde\xi$ is relatively small, we have that $a$ satisfies an asymptotic expansion in $(\xi,\eta)$ for each fixed $(x,y)$. In the region where $\tilde\xi$ is relatively large, we can perform stationary phase in $(\hat{t},\hat{\lambda})$ like the above case (but without needing to split the integral into regions), noting that the determinant $-\tilde\xi^2$ will be nondegenerate in this region.

Thus overall we have that $a$ satisfies an asymptotic expansion
\[a\sim\sum_{j\ge 0}{a_{-1-j}\left(x,y,\frac{(\xi,\eta)}{|(\xi,\eta)|}\right)|(\xi,\eta)|^{-1-j}}\]
for every $(x,y)$. This shows, for every fixed $(x,y)$, that $a$ is a symbol in $(\xi,\eta)$. To show it is a scattering symbol, we can note that, by following the stationary phase arguments above, that the coefficients will be conormal to $x=0$, as the integrand consists of terms which are smooth, except for $\chi(x^{\epsilon}\hat{\lambda})$ whose derivatives may introduce terms like $x^{\epsilon}$, which are still conormal. Alternatively, we can use the characterization that $a$ is a symbol if for all $(\alpha,\beta)$, we have that $(xD_x)^{\alpha}D_y^{\beta}a$ satisfies a similar asymptotic expansion, i.e. taking vector fields tangent to the boundary does not affect its symbol behavior in $(\xi,\eta)$. Indeed, for $A(x,y,\hat{\lambda},\omega,\hat{t})$, if we let
\[I[A](x,y,\xi,\eta) = \int_{\mathbb{R}_{\hat{t}}\times\mathbb{R}_{\hat\lambda}\times\mathbb{S}^1_{\omega}}{e^{i(\xi,\eta)\cdot\gamma}\chi(x^{\epsilon}\hat\lambda)\exp\left(-\digamma\left(\frac{\hat\lambda^2}{2\alpha}-\left(\frac{1}{x'}-\frac{1}{x}\right)\right)\right)A(x,y,\hat{\lambda},\omega,\hat{t})\,d\hat{t}\,d\hat\lambda\,d\omega}\]
then
\[(xD_x)I[A] = I[A']\]
where
\[A' = xD_xA + \left(\frac{\epsilon x^{\epsilon}\hat\lambda\chi'(x^{\epsilon}\hat\lambda)}{\chi(x^{\epsilon}\hat\lambda)} + (\xi,\eta)\cdot x\partial_x\gamma + xD_x\left(-\digamma\left(\frac{\hat\lambda^2}{2\alpha}-\left(\frac{1}{x'}-\frac{1}{x}\right)\right)\right)\right)A.\]
Essentially the only possibly problematic term is the $(\xi,\eta)\cdot xD_x\gamma$, given that a priori it adds an additional order to the symbol. However, we note that
\[x\partial_x\gamma = x\partial_x(\hat\lambda\hat{t} + \alpha\hat{t}^2 + O(x\hat{t}^3),\omega\hat{t} + O(x\hat{t}^2)) = xO(\hat{t}^2)\]
(where $O(\hat{t}^2)$ means $\hat{t}^2$ times smooth), so using the integration by parts Lemma, we can use one factor of $\hat{t}$ to integrate by parts, thus cancelling the additional $|(\xi,\eta)|$ order. (We note that we can actually do so again because the coefficient of the necessary vector field $V$ will vanish at $\hat{t} = 0$, so it applied to $\hat{t}$ times smooth will keep it that way. In other words, the term is actually one order lower, but we will not need this fact here.) The exponential term is also fine, since
\[\frac{1}{x}-\frac{1}{x'} = \frac{X}{1+xX} = \frac{\hat\lambda\hat{t}+\alpha\hat{t}^2+O(x\hat{t}^3)}{1+x\hat\lambda\hat{t}+x\alpha\hat{t}^2+O(x^2\hat{t}^3)}\]
so
\[x\partial_x\left(\frac{1}{x}-\frac{1}{x'}\right) = O(\hat\lambda\hat{t})+O(\hat{t}^2).\]
Thus this contributes a lower-order contribution to the symbol, but more importantly the derivative is well-behaved as $x\to 0$.

\subsection{Symbol computation}
We will assume our symbol is in fact smooth on the unscaled variables, i.e. we write it as $A(x,y,x\hat{\lambda},\omega,x\hat{t})$ where $A$ is smooth in $(x,y,\lambda,\omega,t)$, and suppose that $A$ satisfies the following property: for a fixed $\eta_0\in\mathbb{S}^1$, we have
\[A\ge 0\text{ everywhere},\quad A(x,y,\lambda,\omega,0) = 0 \iff \omega\cdot \eta_0 = 0.\]
Note this automatically implies that all derivatives of $A$ vanish at $(x,y,\lambda,\omega,0)$ where $\omega\cdot\eta_0=0$. Moreover, for $\omega$ near the equatorial sphere $\omega\cdot\eta_0=0$, decompose $\omega$ as $\omega = \omega_{\parallel}\eta_0 + \sqrt{1-\omega_{\parallel}^2}\omega_{\perp}$, and assume that $\partial^2_{\omega_{\parallel}}A\ge C>0$ when $\omega\cdot\eta_0=0$.

Note that for our operators of interest, we take $\eta_0 = (1,0)$, i.e. $\eta_0\cdot dy = dy_1$, as we have chosen coordinates such that $dy_1$ aligns with the axis of isotropy.

Furthermore, write $y'-y = \omega t + \beta(x,y,\lambda,\omega,t)t^2$, so that $\frac{y'-y}{x} = \omega\hat{t} + x\beta(x,y,x\hat\lambda,\omega,x\hat{t})\hat{t}^2$. Assume as well that
\[(\partial_t\partial_{\omega_{\parallel}}-\eta_0\cdot\beta\partial_{\omega_{\parallel}}^2)A(x,y,x\hat{\lambda},\omega,0) \ge C>0\text{ for all }\omega\in\eta_0^{\perp}.\]
Let $\Sigma = \text{span }\eta_0\frac{dy}{x}$. Let $a$ admit the stationary phase expansion 
\[a \sim \sum_{j=0}^{\infty}{a_{-1-j}(x,y,\tilde\xi,\tilde\eta)|(\xi,\eta)|^{-1-j}}.\]
We then have the following:
\begin{lemma}
\label{new-symbol-prop}
For $a$ as above, the first three terms $a_{-1}$, $a_{-2}$, and $a_{-3}$ have the following properties:
\begin{itemize}
\item $a_{-1}\ge 0$ and vanishes only on $\Sigma$, where it vanishes quadratically; moreover the vanishing is non-degenerate. In particular, $a_{-1}$ satisfies a bound of the form $|a_{-1}|\ge C|\tau|^2$, where $\tau = (\tau_1,\tau_2)$ is a boundary-defining function for $\Sigma\cap\mathbb{S}^2$.
\item $a_{-2}$ is of the form $a_{-2} = ix\tilde{a}$, where $\tilde{a}\in S^{0,0}$ and is real-valued and non-vanishing on $\Sigma$. In particular, $|\tilde{a}|\ge C$ in a region of the form $|\tau|<\epsilon$.
\item $a_{-3}$ is positive on $\Sigma\cap\{x=0\}$, i.e. $a_{-3}\ge C$ in a region of the form $|\tau|+x<\epsilon$.
\end{itemize}
Moreover, all of the lower bounds above have constants which can be chosen to vary continuously as $A$ varies in the $C^2$ norm and as the phase varies in the $C^0$ norm\footnote{In the following sense: the phase can always be written as $(\hat{\lambda}\hat{t},\omega\hat{t}) + O(\hat{t}^2)$; the $C^0$ continuity is in the coefficient of $\hat{t}^2$ in the $O(\hat{t}^2)$ term.}.
\end{lemma}

\begin{proof}
The quadratic vanishing of $a_{-1}$ is well-known; the only comment is that the constant in the lower bound follows from the lower bound of the second derivative $\partial^2_{\omega_{\parallel}}A$. For the other quantities, we calculate the symbol at $(\xi,\eta)\sim(0,\eta_0)$. This simplifies calculations, since in the decomposition $a = a_{1,1} + a_{1,2}$ in the previous section, we only need the $a_{1,1}$ term since $\tilde\xi = 0$. Thus, for the other two terms, it suffices to calculate the stationary phase expansions of
\[\int_{\mathbb{R}\times\mathbb{S}^0}\int_{-1/(100C_0)}^{1/(100C_0)}\int_{-2/3}^{2/3}{e^{i(\xi,\eta)\cdot\gamma}\chi_{1,1}(\omega_{\parallel})\chi_1(\hat{t})A'(x,y,x\hat\lambda,\omega,x\hat{t})(1-\omega_{\parallel}^2)^{-1/2}\,d\omega_{\parallel}\,d\hat{t}\,d\hat\lambda\,d\omega_{\perp}}.\]
For $(\xi,\eta) = \pm|(\xi,\eta)|(0,\eta_0)$, we have that the phase is $|\eta|(\omega_{\parallel}\hat{t} + x\eta_0\cdot\beta\hat{t}^2)$. Thus the term $a_{-2}$ equals
\[i\int_{\mathbb{R}\times\mathbb{S}^0}{\chi(x^{\epsilon}\hat\lambda)\left(\partial_{\hat{t}}\partial_{\omega_{\parallel}}-x\eta_0\cdot\beta\partial_{\omega_{\parallel}}^2\right)\left(e^{-\digamma(\hat\lambda^2+x^{-1}-(x')^{-1})}A(x,y,x\hat{\lambda},\omega,x\hat{t})\right)|_{\hat{t} = 0, \omega_{\parallel} = 0}\,d\hat\lambda\,d\omega_{\perp}}.\]
Note that $A$ vanishes on the set $\{\hat{t} = 0, \omega_{\parallel} = 0\}$, and hence derivatives on the exponential factor times $A$ will vanish except for the terms where both derivatives fall on $A$. Thus the term is\footnote{Note that this calculation implicitly assumes that $\beta$ is a constant; however we can approximate $\beta$ by its value at $\{\hat{t} = 0, \omega_{\parallel} = 0\}$; doing so will introduce an error of order $O(x\omega_{\parallel}\hat{t}^2)+O(x\hat{t}^3)$ into the phase, and this error vanishes to sufficiently high degree at the critical set such that it will not affect the subprincipal behavior.}
\[ix\int_{\mathbb{R}\times\mathbb{S}^0}{\chi(x^{\epsilon}\hat\lambda)(\partial_t\partial_{\omega_{\parallel}}-\eta_0\cdot\beta(x,y,x\hat{\lambda},\omega_{\perp},0)\partial_{\omega_{\parallel}}^2)A(x,y,x\hat{\lambda},\omega_{\perp},0)\,d\hat\lambda\,d\omega_{\perp}}.\]
By assumption, the integrand is bounded from below everywhere; moreover such a bound will vary $C^2$ in $A$ and $C^0$ in the phase (due to the $\beta$ term).

For the term $a_{-3}$, we calculate the term at $x=0$ and on $\Sigma$. At $x=0$, the phase is just $\omega_{\parallel}\hat{t}$, and the term $x^{-1}-(x')^{-1}$ in the exponential is $\hat\lambda\hat{t}+\alpha\hat{t}^2$, and the cutoff $\chi(x^{\epsilon}\hat\lambda)$ is $1$. Hence the term $a_{-3}$ is
\begin{align*}&\frac{i^2}{2}\int_{\mathbb{R}\times\mathbb{S}^0}{(\partial_{\hat{t}}\partial_{\omega_{\parallel}})^2\left(e^{-\digamma\left(\frac{\hat\lambda^2}{2\alpha} + \hat\lambda\hat{t} + \alpha\hat{t}^2\right)}A(0,y,0,\omega,0)\right)|_{\hat{t} = 0, \omega_{\parallel} = 0}\,d\hat\lambda\,d\omega_{\perp}}\\
&=-\frac{1}{2}\int_{\mathbb{R}\times\mathbb{S}^0}{\left(\partial_{\hat{t}}^2\left(e^{-\digamma\left(\frac{\hat\lambda^2}{2\alpha} + \hat\lambda\hat{t} + \alpha\hat{t}^2\right)}\right)\partial_{\omega_{\parallel}}^2A(0,y,0,\omega,0)\right)|_{\hat{t} = 0, \omega_{\parallel} = 0}\,d\hat\lambda\,d\omega_{\perp}}\\
&=\frac{1}{2}\int_{\mathbb{R}\times\mathbb{S}^0}{(2\digamma\alpha-\digamma^2\hat\lambda^2)e^{-\digamma\hat\lambda^2/(2\alpha)}\partial^2_{\omega_{\parallel}}A(0,y,0,\omega_{\perp},0)\,d\hat\lambda\,d\omega_{\perp}}.
\end{align*}
(Note that the second line follows since the $\hat{t}$ derivatives can only fall on the exponential term to produce a nonzero term, and the $\omega_{\parallel}$ derivative must fall at least twice on $A$ to produce a nonzero term since $A$ vanishes to second order in $\omega$.) Since
\[\int_{\mathbb{R}}{(2\digamma\alpha-\digamma^2\hat\lambda^2)e^{-\digamma\hat\lambda^2/(2\alpha)}\,d\hat\lambda} = \sqrt{2\pi}\alpha^{3/2}\digamma^{1/2}\]
it follows that
\[a_{-3} = \frac{\sqrt{\pi\digamma}}{2}\sum_{\omega_{\perp}}{\alpha^{3/2}\partial^2_{\omega_{\parallel}}A(0,y,0,\omega_{\perp},0)}>0.\]
Note again that this depends continuously in the $C^2$ norm for $A$ and in the $C^0$ norm for the phase (via $\alpha$).
\end{proof}
Note that while this means the symbol is not elliptic when considered as a standard scattering symbol, by Lemma \ref{new-symbol-ell} the result gives us that the symbol is elliptic as a symbol in $S^{-1,0,-2,-2}$, at least near fiber infinity.

We still need to show that the symbol is elliptic at finite points, so we discuss this now. Recall that at $x=0$ the symbol becomes
\[a(0,y,\xi,\eta) = \int{e^{i(\xi,\eta)\cdot(\hat\lambda\hat{t}+\alpha\hat{t}^2,\omega\hat{t})}e^{-\digamma\left(\frac{\hat\lambda^2}{2\alpha} + \hat\lambda\hat{t} + \alpha\hat{t}^2\right)}A(0,y,0,\omega,0)\,d\hat{t}\,d\hat\lambda\,d\omega}.\]
where $\alpha = \alpha(0,y,0,\omega)$. We can evaluate this integral as follows:

We make the substitution $X = \hat\lambda\hat{t} + \alpha\hat{t}^2$, $Y = \omega\hat{t}$. Then $dX\,dY = \hat{t}^2\,d\hat{t}\,d\hat\lambda\,d\omega\implies d\hat{t}\,d\hat\lambda\,d\omega = |Y|^{-2}\,dX\,dY$. Furthermore, $\hat\lambda = \frac{X-\alpha|Y|^2}{|Y|}$. Thus the integral becomes
\begin{align*} a &= \int_{\mathbb{R}_X\times\mathbb{R}^2_Y}{e^{i(\xi X + \eta\cdot Y)}e^{-\digamma\left(\frac{(X-\alpha|Y|^2)^2}{2\alpha|Y|^2} + X\right)}A(0,y,0,\frac{Y}{|Y|},0)|Y|^{-2}\,dX\,dY}\\
&=\int_{\mathbb{R}_X\times\mathbb{R}^2_Y}{e^{i(\xi X + \eta\cdot Y)}e^{-\digamma\left(\frac{X^2}{2\alpha|Y|^2} + \frac{\alpha|Y|^2}{2}\right)}A(0,y,0,\frac{Y}{|Y|},0)|Y|^{-2}\,dX\,dY}.
\end{align*}
Evaluating the Fourier transform in $X$, we get
\[a = \sqrt{2\pi}\int_{\mathbb{R}^2_Y}{e^{i\eta\cdot Y}e^{-\frac{\alpha\xi^2|Y|^2}{2\digamma}}e^{-\frac{\digamma\alpha|Y|^2}{2}}A(0,y,0,\frac{Y}{|Y|},0)\sqrt{\alpha/\digamma}|Y|^{-1}\,dY}.\]
Note that all terms in the integrand aside fro $e^{i\eta\cdot Y}$ are real. Hence, we can evaluate the real part\footnote{If we assume $A(0,y,0,\omega,0)$ is even in $\omega$, as one often does in practice, then we actually know $a$ must be real. In any case, we can always estimate the real part from below.} of $a$ by polar coordinates:
\begin{align*}
\text{Re }a &= \sqrt{2\pi}\int_{\mathbb{R}^2_Y}{\cos(\eta\cdot Y)e^{-\frac{\alpha\xi^2|Y|^2}{2\digamma}}e^{-\frac{\digamma\alpha|Y|^2}{2}}A(0,y,0,\frac{Y}{|Y|},0)\sqrt{\alpha/\digamma}|Y|^{-1}\,dY} \\
&= \sqrt{2\pi}\int_{\mathbb{S}^1}{\left(\int_0^{\infty}{\cos(r\eta\cdot\omega)e^{-\frac{\left(\frac{\alpha}{\digamma}\xi^2+\digamma\alpha\right)r^2}{2}}\,dr}\right)A(0,y,0,\omega,0)\sqrt{\alpha/\digamma}\,d\omega}.
\end{align*}
Since
\[\int_0^{\infty}{\cos(r\eta\cdot\omega)e^{-\frac{\left(\frac{\alpha}{\digamma}\xi^2+\digamma\alpha\right)r^2}{2}}\,dr} = \frac{1}{2}\int_{\mathbb{R}}{e^{ir\eta\cdot\omega}e^{-\frac{\left(\frac{\alpha}{\digamma}\xi^2+\digamma\alpha\right)r^2}{2}}\,dr} = \sqrt{\frac{\pi}{2}\frac{\digamma}{\alpha(\xi^2+\digamma^2)}}e^{-\frac{\digamma}{\alpha(\xi^2+\digamma^2)}(\eta\cdot\omega)^2}\]
it follows that
\[\text{Re }a = \pi\sqrt{\frac{\digamma}{\xi^2+\digamma^2}}\int_{\mathbb{S}^1}{e^{-\frac{\digamma}{\alpha(\xi^2+\digamma^2)}(\eta\cdot\omega)^2}A(0,y,0,\omega,0)\sqrt{\alpha/\digamma}\,d\omega}.\]
Noting that the terms in the integrand are all strictly positive aside from $A$, which is positive almost everywhere on $\mathbb{S}^1$, it follows that the integral is positive, and in particular $a$ is nonzero, as desired.
\begin{remark}
By using Laplace's method, i.e. the analogue of the method of stationary phase for real exponentials, one can obtain an asymptotic expansion as $|(\xi,\eta)|\to\infty$, as predicted by stationary phase.
\end{remark}

\section{The Local Recovery Problem}
\label{local-recovery-sec}

In this section we make the argument for the local recovery problem. Once the local recovery result is established, a global recovery result can be obtained via a convex foliation argument, which would follow exactly the argument used in previous papers.


Recall we are in the setting where we assume all of the parameters were known except for one, call it $\nu$, and we assume that the travel time data for the two parameters $a_{\nu}$ and $\tilde{a}_{\nu}$ were the same. For simplicity of notation, we let $I = I^{\nu}$ and $\tilde{I} = \tilde{I}^{\nu}$. Note that
\[I[u](x_0,\xi_0) = \int_{\mathbb{R}}{A(X(t,x_0,\xi_0),\Xi(t,x_0,\xi_0))u(X(t,x_0,\xi_0))\,dt}, \quad A(x,\xi) = -E^{\nu}(x,\xi)\frac{\partial\tilde\Xi}{\partial\xi}(x,\xi)\]
and
\[\tilde{I} = \int_{\mathbb{R}}{\tilde{A}(X(t,x_0,\xi_0),\Xi(t,x_0,\xi_0))u(X(t,x_0,\xi_0))\,dt}\]
with
\[\tilde{A}(x,\xi) = -\partial_xE^{\nu}\frac{\partial\tilde\Xi}{\partial\xi}+\partial_{\xi}E^{\nu}\frac{\partial\tilde\Xi}{\partial x}.\]
Recall then that the pseudolinearization formula reads
\[0 = I[\partial_xr_{\nu},\partial_{y_1}r_{\nu},\partial_{y_2}r_{\nu}] + \tilde{I}[r_{\nu}].\]
If we let
\[Lv(x,y) = x^{-2}\int_{\mathbb{S}^2}{\chi(x^{\epsilon}\hat\lambda)\exp\left(-\frac{\digamma^2\hat\lambda^2}{2\alpha}\right)v(\gamma_{\lambda,\omega})\,d\mathbb{S}^2(\lambda,\omega)}\]
then we have
\begin{equation}
\label{n-equation}
0 = N[\partial_xr_{\nu},\partial_{y_1}r_{\nu},\partial_{y_2}r_{\nu}] + \tilde{N}[r_{\nu}]
\end{equation}
where $N = L\circ I$ and $\tilde{N} = L\circ\tilde{I}$.

We now let $\digamma>0$, and we consider the conjugation of $N$ and $\tilde{N}$ by the function $e^{\gamma/x}$. That is, we consider
\[N_{\digamma} = e^{-\digamma/x}Ne^{\digamma/x},\quad \tilde{N}_{\digamma} = e^{-\digamma/x}\tilde{N}e^{\digamma/x}.\]
As shown in Section \ref{local-ops-sec}, we have that $N_{\digamma}$ is a scattering pseudodifferential operator whose symbol is of the form
\[a = \int_{\mathbb{R}_{\hat{t}}\times\mathbb{R}_{\hat\lambda}\times\mathbb{S}^1_{\omega}}{e^{i\left(\xi\frac{x'-x}{x^2}+\eta\cdot\frac{y'-y}{x}\right)}\chi(x^{\epsilon}\hat\lambda)\exp\left(-\digamma\left(\frac{\hat\lambda^2}{2\alpha}-\left(\frac{1}{x'}-\frac{1}{x}\right)\right)\right)A(x,y,\hat{\lambda},\omega,\hat{t})\,d\hat{t}\,d\hat\lambda\,d\omega}.\]
Hence, from the results in Section \ref{op-analysis}, particularly Lemma \ref{new-symbol-prop}, as well as Lemma \ref{new-symbol-ell}, we have the following:
\begin{theorem}
For any $\digamma>0$, we have $N_{\digamma}\in\Psi^{-1,0,-2,-2}$ and $\tilde{N}_{\digamma}\in\Psi^{-1,0,-1,-1}$. Moreover, if
\[(\partial_t\partial_{\omega_{\parallel}}-\eta_0\cdot\beta\partial_{\omega_{\parallel}}^2)A(x,y,x\hat{\lambda},\omega,0) \ge C>0\text{ for all }\omega\in\eta_0^{\perp}\]
where $\eta_0\in\mathbb{S}^1$ is parallel to the axis of isotropy and $\omega = \omega_{\parallel}\eta_0+\omega_{\perp}$, then $N_{\digamma}$ is in fact elliptic.
\end{theorem}
Note that the subprincipal condition is satisfied under reasonable geological assumptions for the Earth: in fact, working through the definition of $A$ above, we can see that the term in the subprincipal condition is the same as the term
\[\left(\partial_t\xi_T(x,0,\omega')\partial_{\omega_{\parallel}}\xi_T(x,0,\omega)- \frac{\zeta}{|\zeta|}\cdot\alpha(\omega)(\partial_{\omega_{\parallel}}\xi_T(x,0,\omega))^2\right)\]
appearing in \cite{zou-global} when considering the subprincipal symbol in the global setting. (See the discussion in \cite{zou-global}, particularly following Equation 2.10.)

We now let $r_{\nu} = e^{\digamma/x}f$. It now suffices to show, from \eqref{n-equation}, that $f\equiv 0$, assuming $c$ is chosen sufficiently small. Indeed, from \eqref{n-equation} we have
\begin{align*}
0 &= e^{-\digamma/x}N[\partial_x(e^{\digamma/x}f),\partial_{y_1}(e^{\digamma/x}f),\partial_{y_2}(e^{\digamma/x}f)]+e^{-\digamma/x}\tilde{N}[e^{\digamma/x}u] \\
&= N_{\digamma}[(\partial_x-\digamma/x^2)f,\partial_{y_1}f,\partial_{y_2}f] + \tilde{N}_{\digamma}[f].
\end{align*}
For a fixed $\digamma>0$, the operator $N_{\digamma}$ is elliptic, and hence there exists a parametrix $Q\in\Psi^{1,0,2,2}\otimes\text{Mat}_{3\times 3}$ such that $QN_{\digamma} = I+R$, with $R\in\Psi^{-\infty,-\infty}\otimes\text{Mat}_{3\times 3}$. Applying $Q$ to both sides gives
\begin{equation}
\label{e-eq}
0 = \left(\partial_xf -\frac{\digamma}{x^2}f,\partial_{y_1}f,\partial_{y_2}f\right) + E[f]
\end{equation}
where
\[E[f] = Q\circ\tilde{N}[f] + R[\partial_xu-\frac{\digamma}{x^2}f,\partial_{y_1}f,\partial_{y_2}f].\]
Note\footnote{In fact, it would have sufficed for the parametrix error $R$ to be in $\Psi^{-3,-3,3,3}\otimes\text{Mat}_{3\times 3}\subset\Psi^{-2,-2,1,1}\otimes\text{Mat}_{3\times 3}$ (as opposed to all the way in $\Psi^{-\infty,-\infty}$), since we need to cancel one order in derivative and two orders in decay.} then that $E\in\Psi^{0,0,1,1}\otimes\text{Mat}_{3\times 1}$.

Writing $E = (E_0,E_1,E_2)$, we can rewrite \eqref{e-eq} as the three scalar equations
\begin{align*}
0 &= (x^2\partial_x-\digamma)f + x^2E_0[f], \\
0 &= x\partial_{y_i}f + xE_i[f], \quad i=1,2.
\end{align*}
Note that each equation is of the form $0 = Af + Ef$, where $A\in\text{Diff}_{sc}^{1,0}$, and $E$ is at worst in $\Psi^{0,-1,1,1}$. Moreover, the symbols of the scattering differential operators are $i(\xi+i\digamma)$, $i\eta_1$, and $i\eta_2$, respectively, so one can take $\Psi^{0,0}_{sc}$ combinations\footnote{Specifically, we apply $\Psi$DOs whose symbols are $\frac{\xi-i\digamma}{\langle(\xi,\eta)\rangle}$ and $\frac{\eta_i}{\langle(\xi,\eta)\rangle}$, respectively.} of the scattering differential operators to form an elliptic element $B\in\Psi^{1,0}_{sc}$. Applying the $\Psi^{0,0}_{sc}$ combinations to the equations above, we end up with
\[0 = Bf + Ff,\quad B\in\Psi^{1,0}_{sc}\text{ elliptic}, F\in\Psi^{0,-1,1,1}\]
so applying an elliptic parametrix for $B$ yields
\[0 = f + Gf,\quad G\in\Psi^{-1,-1,1,1}\subset\Psi^{-1/2,-1/2,0,0}.\]
Hence, we have
\[\|f\|_{L^2_{sc}}\le \|G\|_{H^{-1/2,-1/2}_{sc}\to L^2_{sc}}\|f\|_{H^{-1/2,-1/2}_{sc}}.\]
It suffices to show that for any $\epsilon>0$ we have $\|f\|_{H^{-1/2,-1/2}_{sc}}\le\epsilon\|u\|_{L^2_{sc}}$ if $f$ is supported sufficiently close to the boundary. But this just follows by noting that if $f$ is supported in $\{x<c\}$, then $x^{1/2}|f|\le c^{1/2}|f|$, and hence
\[\|f\|_{H^{-1/2,-1/2}_{sc}}\le\|f\|_{H^{0,-1/2}_{sc}} = \|x^{1/2}f\|_{L^2_{sc}}\le c^{1/2}\|f\|_{L^2_{sc}}.\]
Thus, taking $c$ sufficiently small, we can conclude that $f$ (restricted to $\{\tilde{x}>-c\}$) is identically zero. This concludes the proof of Theorem \ref{local-recov-thm}, which, as mentioned in the Introduction, when combined with a convex foliation condition proves Theorem \ref{global-recov-scat-thm} as well.

\bibliographystyle{plain}
\bibliography{ti_local_parametrix_122921}

\end{document}